\documentclass[]{article}

\usepackage{amsmath}
\usepackage{amssymb}
\usepackage{amsthm}
\usepackage{mathrsfs}
\usepackage{pxfonts}
\usepackage{sectsty}
\usepackage{enumitem}
\usepackage{microtype}
\usepackage{graphicx}
\usepackage{xcolor}
\usepackage{tikz}
\usetikzlibrary{arrows.meta,calc,fit,positioning}
\usepackage[hidelinks]{hyperref}
\hypersetup{pdftitle={UNIQUENESS FOR DLR EQUATIONS OF Sine beta}}

\allowdisplaybreaks

\sectionfont{\scshape\centering\fontsize{11}{14}\selectfont}
\subsectionfont{\scshape\fontsize{11}{14}\selectfont}

\usepackage{fancyhdr}

\newcommand\shorttitle{UNIQUENESS FOR DLR EQUATIONS OF \(\mathsf{Sine}_\beta\)}
\newcommand\authors{T. Assiotis}

\newtheorem{thm}{Theorem}[section]

\newtheorem{lem}[thm]{Lemma}
\newtheorem{prop}[thm]{Proposition}
\newtheorem{defn}[thm]{Definition}

\theoremstyle{definition}

\theoremstyle{plain}

\newcommand{\E}{\mathbb{E}}
\newcommand{\R}{\mathbb{R}}
\newcommand{\N}{\mathbb{N}}
\newcommand{\Conf}{\mathsf{Conf}}
\newcommand{\Ent}{\mathsf{Ent}}
\newcommand{\Var}{\operatorname{Var}}
\newcommand{\Sine}{\mathsf{Sine}}
\newcommand{\cW}{\mathcal{W}}
\newcommand{\Bin}{\mathsf{Bin}}
\newcommand{\scrF}{\mathscr{F}}
\newcommand{\defeq}{\overset{\mathrm{def}}{=}}
\newcommand{\1}{\mathbf{1}}
\newcommand{\dd}{\mathrm{d}}
\newcommand{\eps}{\varepsilon}
\newcommand{\intE}{\mathrm{int}}
\newcommand{\per}{\mathrm{per}}
\newcommand{\bC}{\boldsymbol{\mathsf{C}}}
\newcommand{\bgamma}{\boldsymbol{\gamma}}
\newcommand{\betaeta}{\boldsymbol{\eta}}
\newcommand{\bxi}{\boldsymbol{\xi}}
\newcommand{\bzeta}{\boldsymbol{\zeta}}
\newcommand{\balpha}{\boldsymbol{\alpha}}
\newcommand{\br}{\boldsymbol{r}}
\newcommand{\bX}{\boldsymbol{\mathsf{X}}}
\newcommand{\bY}{\boldsymbol{\mathsf{Y}}}
\newcommand{\bZ}{\boldsymbol{\mathsf{Z}}}
\newcommand{\bCh}{\boldsymbol{\mathsf{Ch}}}
\newcommand{\bP}{\mathbf{P}}
\newcommand{\bQ}{\mathbf{Q}}
\newcommand{\bR}{\mathbf{R}}

\newcommand{\sN}{\mathsf{N}}
\newcommand{\sD}{\mathsf{D}}
\newcommand{\sS}{\mathsf{S}}
\newcommand{\sB}{\mathsf{B}}

\definecolor{DLRBlue}{HTML}{1F5A99}
\definecolor{EnergyTeal}{HTML}{147D73}
\definecolor{EntropyOrange}{HTML}{C45A14}
\definecolor{RecoveryGreen}{HTML}{568A25}
\definecolor{BoundaryRed}{HTML}{B33C61}
\definecolor{ConclusionPurple}{HTML}{6B55A3}
\definecolor{ArrowGray}{HTML}{4C566A}
\tikzset{
 proofnode/.style={rounded corners=2pt, draw, thick, align=center,
   inner xsep=5pt, inner ysep=5pt, font=\footnotesize,
   text width=3.15cm, minimum height=1.05cm},
 dlrnode/.style={proofnode, draw=DLRBlue, fill=DLRBlue!10},
 energynode/.style={proofnode, draw=EnergyTeal, fill=EnergyTeal!10},
 entropynode/.style={proofnode, draw=EntropyOrange, fill=EntropyOrange!10},
 recoverynode/.style={proofnode, draw=RecoveryGreen, fill=RecoveryGreen!10},
 boundarynode/.style={proofnode, draw=BoundaryRed, fill=BoundaryRed!10},
 conclusionnode/.style={proofnode, draw=ConclusionPurple, fill=ConclusionPurple!10},
 proofarrow/.style={-{Latex[length=2.2mm]}, thick, draw=ArrowGray}
}

\title{\large\bf UNIQUENESS FOR DLR EQUATIONS OF \(\Sine_\beta\)}
\author{\small THEODOROS ASSIOTIS}
\date{}

\begin{document}

\maketitle

\begin{abstract}
We prove that, for all \(\beta>0\), the Dobrushin--Lanford--Ruelle (DLR)
equations for the unit-intensity  \(\mathsf{Sine}_\beta\) point process have a unique stationary solution under a
particular finite electric energy condition.  Uniqueness fails if this
condition is dropped.  This answers a question of
Dereudre--Hardy--Lebl\'e--Ma\"ida \cite{DHLM} and gives a canonical
statistical physics characterisation of \(\mathsf{Sine}_\beta\). The main technical ingredient is a relative entropy estimate which allows us to compare the conditional distribution in a finite box of a DLR solution to a corresponding circular $\beta$ ensemble.
\end{abstract}

\tableofcontents

\section{Introduction}
\subsection{Setting and motivation}
The \(\Sine_\beta\) point process is a universal bulk scaling limit for
broad classes of one-dimensional \(\beta\)-log gases; see, for example,
\cite{BourgadeErdosYau,BourgadeErdosYauNonconvex,ErdosYauBook}.  It
first arose in the constructions of Valk\'o--Vir\'ag and
Killip--Stoiciu \cite{ValkoViragCarousel,KillipStoiciu}.  It has
intrinsic characterisations in terms of a stochastic differential
equation describing its counting function \cite{ValkoViragCarousel}
and as the spectrum of an explicit stochastic operator
\cite{ValkoViragSineOperator}.  For our purposes, the most convenient
definition is as the weak limit of the circular \(\beta\)-ensemble.
Namely, if
\((\theta_1^{(N)},\ldots,\theta_N^{(N)})\in[-\pi,\pi)^N\) has law
\[
 \frac{1}{Z_{N,\beta}^{\mathrm{circ}}}
 \prod_{1\leq j<k\leq N}
 \left|
   \mathrm e^{\mathrm i\theta_j^{(N)}}
   -\mathrm e^{\mathrm i\theta_k^{(N)}}
 \right|^\beta
 \prod_{j=1}^N\frac{\dd\theta_j^{(N)}}{2\pi},
 \qquad
 \operatorname{Law}\!\left(
   \sum_{j=1}^N
   \delta_{\frac{N}{2\pi}\theta_j^{(N)}}
 \right)
 \xrightarrow[N\to\infty]{\mathrm w}
 \Sine_\beta ,
\]
where the convergence is on configuration space equipped with the
vague topology \cite{VV}.  For \(\beta=2\), \(\Sine_2\) is the
determinantal sine-kernel process \cite{ForresterBook}, which is also
conjectured to describe the local statistics of the zeros of the
Riemann zeta function high up the critical line $\Re(z)=1/2$, see
\cite{MontgomeryPairCorrelation,KeatingSnaithZeta,BourgadeKeating}.
The purpose of this paper is to develop further a different,
statistical-physics viewpoint on \(\Sine_\beta\).

This sits within a broad long-term programme, initiated in this form by Sandier and Serfaty and
subsequently developed by Serfaty and several collaborators, which gives a
rigorous statistical-mechanics description of logarithmic, Coulomb and
Riesz gases through suitable notions of renormalised electric energy
\cite{SandierSerfaty2D,SandierSerfaty1D,PS,RougerieSerfaty,
SerfatyCoulombGinzburg,SerfatyLecturesCoulombRiesz}.
For the one-dimensional log gas in a series of fundamental papers, Lebl\'e and Serfaty (LS) introduced the
corresponding free-energy functional and proved a large-deviation
principle \cite{LS}, Erbar, Huesmann and Lebl\'e (EHL) proved that
\(\Sine_\beta\) is its unique stationary minimiser \cite{EHL} and
Dereudre, Hardy, Lebl\'e and Ma\"ida (DHLM) proved that \(\Sine_\beta\) is
number-rigid and satisfies the canonical DLR equations \cite{DHLM}. The DLR formalism
\cite{Dobrushin1968,LanfordRuelle,Georgii} says that, conditionally on
the exterior configuration \(\bxi\) and on the number \(n\) of
particles in a bounded region \(\Lambda\), the interior configuration
has law
\[
 \bP\!\left(
   \dd\betaeta
   \,\middle|\,
   \bC_{\Lambda^c}=\bxi,\ \bC(\Lambda)=n
 \right)
 \propto
 \exp\!\left[-\beta H_\Lambda(\betaeta\mid\bxi)\right]
 \Bin_{n,\Lambda}(\dd\betaeta).
\]
Here \(H_\Lambda\) contains the interior logarithmic interaction and a
renormalised exterior potential and $\Bin_{n,\Lambda}$ is the standard Binomial reference law. For $\Sine_2$ and more general determinantal point processes results of this type first appeared in the works of Bufetov \cite{BufetovConditional,BufetovQuasiSymmetries}. Such DLR, or weaker quasi-Gibbs, descriptions also underlie the
construction and study of infinite-dimensional analogues of Dyson Brownian
motion
\cite{DysonBrownianMotion,Osada2012,Osada2013,OsadaTanemura2016, SuzukiErgodicity,SuzukiCurvature,
AssiotisSuzukiCollision}.
Recent spectacular developments in higher dimension include the DLR theory for the
two-dimensional one-component plasma \cite{LebleDLR2DOCP} and the
Coulomb dynamics studied by Osada and Osada
\cite{OsadaOsadaCoulombISDE}.

The uniqueness question to the DLR equations for $\Sine_\beta$ was raised in \cite{DHLM}, where a positive
answer was conjectured.  One possible approach, discussed in \cite{DHLM} and more recently explicitly
in \cite{LebleHDR}, would be to prove that every stationary DLR
solution minimises the free-energy functional of \cite{LS}, and then
invoke the uniqueness theorem of \cite{EHL}.  This would amount to a
version of the Gibbs variational principle for the singular long-range
logarithmic interaction, for which the classical theory does not apply
directly. It would be interesting if this approach could be carried out rigorously. We instead take a completely different route and prove uniqueness without using the free-energy
functional or its variational characterisation at all.

It is worth mentioning that for \(\beta=2\), uniqueness follows from the work of Kuijlaars and
Mi\~na-D\'iaz \cite{KuijlaarsMinaDiaz}, using the determinantal
structure and the Riemann--Hilbert method.  Neither ingredient is
available for general \(\beta\).  In recent joint work with Najnudel
\cite{AssiotisNajnudel}, we expressed all the correlation functions of
\(\Sine_\beta\) in terms of moments of the Hua--Pickrell stochastic
zeta function \cite{LiValkoCircularJacobi}.  Formal computations, under reasonable heuristics,
suggest that the conditional correlation functions of an arbitrary
DLR solution converge to these formulae as the conditioning interval
exhausts the line.  Proving this would require difficult uniform estimates which,
in \cite{AssiotisNajnudel}, rely on special structure of the
Verblunsky coefficients of the circular \(\beta\)-ensemble and are
unavailable for general DLR solutions.  We strongly believe, but cannot prove, that such
convergence of correlation functions is true. Fortunately, it is considerably stronger than what is
needed for uniqueness.

What we do instead is to compare directly the finite-volume
conditional law of a general stationary DLR solution with a suitably
matched circular \(\beta\)-ensemble. We say more about the argument now. Let \(I_L=[0,L]\).  Number
rigidity implies that, after conditioning on the exterior
configuration \(\bxi\), the conditional law in \(I_L\) is supported on
a single \(n\)-particle fibre.  We dilate the interval by the factor
\(n/L\), so that the resulting interval has length \(n\) and carries a
neutral configuration (i.e. same length as the number of particles), and compare the dilated conditional law with
the circular \(\beta\)-ensemble containing \(n\) particles on a circle
of circumference \(n\).  The two measures are thus Gibbs measures on
the same fixed-particle-number fibre and with respect to the same
binomial reference law.

A direct Gibbs comparison eliminates the unknown partition functions
and importantly produces cancellation of the leading bulk terms.  The remaining
errors are \(o(L)\); the most delicate is the exterior-interaction
term, which is controlled using exact neutrality together with the
sublinear discrepancy estimate supplied by the electric-energy
assumption.  After averaging over the exterior configuration, this
gives
\[
 \mathbb E_{\mathrm{exterior}}\!\left[
   \Ent\!\left(
     \text{neutralised conditional DLR law}
     \,\middle|\,
     \text{matched circular \(\beta\)-ensemble}
   \right)
 \right]
 =o(L).
\]
At this stage, Pinsker's inequality does not directly close the argument, since the right-hand
side is \(o(L)\), rather than \(o(1)\).  Instead, a cyclic averaging argument and a 
transport inequality distribute this entropy cost over the \(n\asymp L\) gaps, yielding
an \(o(1)\) comparison for each fixed rooted block.  The circular $\beta$-ensemble
local limit then identifies the Palm gap laws, and Palm inversion
identifies the original stationary point processes completing the proof. A detailed sketch of the proof is given in Section \ref{SubSectionStrategy}.

Of course, sub-extensive relative-entropy comparisons have a long
history in statistical mechanics, such as in Yau's dynamical
relative-entropy method
\cite{GuoPapanicolaouVaradhan,YauRelativeEntropy} and also in static
entropy-density and Gibbs-variational arguments
\cite{FoellmerEntropy,GeorgiiZessin,DereudreVariational}.
However, the particular comparison used here for the singular
logarithmic interaction is as far as I can tell new.  More broadly, one might
try to compare DLR conditional distributions with suitably chosen tractable
reference laws in higher dimensions or for other interactions, but
the specifics of our argument are essentially one-dimensional and such extensions would
require major new ingredients.

Finally, I believe this DLR characterisation may  be useful for proving convergence to
\(\Sine_\beta\) in pre-limit models with a natural Gibbs structure.
It would furthermore be interesting to connect the DLR equations
rigorously with the loop-equation hierarchy of Bourgade and Huang
\cite{BourgadeHuangLoop}; formally, the latter is obtained by
integration by parts in the finite-volume conditional Gibbs law.

\subsection{The DLR equations and main result}

Let \(\Conf(\R)\) be the space of locally finite integer-valued Radon
measures on \(\R\), with the Borel sigma-field generated by the vague
topology (equivalently, by
\(\bgamma\mapsto\int_{\R}f(x)\dd\bgamma(x)\), with \(f\) continuous and
compactly supported). For
\(\bgamma\in\Conf(\R)\), write \(\bgamma_{\mathscr A}\) for its restriction
to a Borel set \(\mathscr A\).  Also,
\(\delta_x\) is Dirac mass at \(x\), \(\delta_\varnothing\) is the
probability law concentrated at the empty configuration, and
\(0\log(0)=0\) throughout.

A point process is a Borel probability
measure on \(\Conf(\R)\), and \(\E_\bP\) denotes expectation under
a point process \(\bP\). Translation by \(t\) is given by
\((\theta_t\bgamma)(\mathscr A)=\bgamma(\mathscr A+t)\), and a point process
\(\bP\) is stationary when \(\bP\circ\theta_t^{-1}=\bP\) for every
\(t\in\R\).

The following is our introductory, slightly informal, definition of the DLR equations. Let \(\beta>0\), let \(\bP\) be a stationary point process, and let
\(\Lambda\subset\R\) be bounded and Borel.  Following
\cite[Theorems~1.1 and~2.1]{DHLM}, define the exterior weight, whenever the
limit exists as a positive finite quantity, by
\begin{equation}\label{EqIntroExteriorWeight}
 \boldsymbol{\omega}_\Lambda(x|\boldsymbol{\gamma}_{\Lambda^c})
 \defeq
 \lim_{p\to\infty}
 \prod_{\substack{u\in\bgamma_{\Lambda^c}\\ |u|\leq p}}
 \left|1-\frac{x}{u}\right|^\beta,
 \qquad x\in\Lambda,
\end{equation}
where atoms are counted with multiplicity.  For
\(n=\boldsymbol{\gamma}(\Lambda)\), put
\begin{align}
 Z_\Lambda(\bgamma_{\Lambda^c},n)
 &\defeq
 \int_{\Lambda^n}
 \prod_{1\leq j<k\leq n}|x_j-x_k|^\beta
 \prod_{j=1}^n
 \boldsymbol{\omega}_\Lambda(x_j|\boldsymbol{\gamma}_{\Lambda^c})
 \prod_{j=1}^n\dd x_j,
 \label{EqIntroPartition}\\
 \boldsymbol{\rho}_{\Lambda,\boldsymbol{\gamma}}(x_1,\ldots,x_n)
 &\defeq
 \frac{1}{Z_\Lambda(\bgamma_{\Lambda^c},n)}
 \prod_{1\leq j<k\leq n}|x_j-x_k|^\beta
 \prod_{j=1}^n
 \boldsymbol{\omega}_\Lambda(x_j|\boldsymbol{\gamma}_{\Lambda^c}).
 \label{EqIntroDensity}
\end{align}
We say that \(\bP\) satisfies the concrete canonical \(\beta\)-DLR equations
when, for every such \(\Lambda\), these objects are defined on a full
\(\bP\)-measure set, \(0<Z_\Lambda(\bgamma_{\Lambda^c},n)<\infty\), and
\begin{equation}\label{EqIntroDLR}
 \int_{\Conf(\R)}F(\bgamma)\bP(\dd\bgamma)
 =
 \int_{\Conf(\R)}\!\left[
 \int_{\Lambda^{\boldsymbol{\gamma}(\Lambda)}}
 F\left(\sum_{j=1}^{\boldsymbol{\gamma}(\Lambda)}\delta_{x_j}
       +\bgamma_{\Lambda^c}\right)
 \boldsymbol{\rho}_{\Lambda,\boldsymbol{\gamma}}
 \left(x_1,\ldots,x_{\boldsymbol{\gamma}(\Lambda)}\right)
 \prod_{j=1}^{\boldsymbol{\gamma}(\Lambda)}\dd x_j
 \right]\bP(\dd\bgamma)
\end{equation}
for every bounded Borel \(F\).

When \(\boldsymbol{\gamma}(\Lambda)=0\), the inner integral in
\eqref{EqIntroDLR} has its usual empty-product meaning.  The adjective canonical means that the resampling keeps the number of particles
in \(\Lambda\) fixed.  Thus, \eqref{EqIntroDLR} initially conditions on the
exterior configuration together with this number.  Once number rigidity is
known, i.e. the number of particles in \(\Lambda\) itself is exterior-measurable, the same formula is the
full exterior conditional law appearing in \cite[Theorem~1.1(C)]{DHLM}. In Section~\ref{SectionSetting} we give a more rigorous formulation which will be used in the actual theorem and proofs.

Our main result identifies the only stationary finite-energy process
compatible with the canonical conditional laws above as the \(\Sine_\beta\) point process
\cite{VV,DHLM}.

\begin{thm}\label{ThmMain}
For \(\beta>0\),   \(\Sine_\beta\) is the unique Borel probability
measure \(\bP\) on \(\Conf(\R)\) satisfying
\[
 \bP\circ\theta_t^{-1}=\bP\quad(t\in\R),\qquad
 \cW^{\mathrm{div}}(\bP)<\infty,
\]
and the canonical \(\beta\)-DLR equations of
Definition~\ref{DefDLREquations}.  Here \(\cW^{\mathrm{div}}\) is the
process-level divergence-compatible renormalised electric energy of
Definition~\ref{DefElectricEnergies}.
\end{thm}

We note that if one drops the finite renormalised-electric energy condition \(\cW^{\mathrm{div}}(\bP)<\infty\) natural counterexamples exist, see Section \ref{SectionCounterexamples}.

\subsection{Strategy of proof}\label{SubSectionStrategy}

We give a roadmap to the proof and we also point to Figure \ref{FigProofArchitecture} which explains the general architecture of the proof rather faithfully. We will introduce quite a bit of notation, some informally; each such item will be recalled at its first technical use.  Let \(\bP\) satisfy the hypotheses of
Theorem~\ref{ThmMain}.  We write
\begin{equation*}
\bC\colon\Conf(\R)\to\Conf(\R), \ \bC(\bgamma)=\bgamma,
\end{equation*}
for the
canonical configuration.  For \(s>0\), put
\[
 I_s=[0,s],\qquad I_0=\varnothing,\qquad
 \sN_s=\bC(I_s),\qquad
 \sD_s=\sN_s-s,\qquad v(s)=\E_{\bP}[\sD_s^2],
\]
and write \(T_a(x)=ax\) for \(a>0\).  Let \(\balpha_L\) be the law of
the exterior restriction \(\bC_{I_L^c}\).  The symbol \(\Ent\)
 denotes the standard relative entropy and \(W_I^{\intE}\) the intrinsic ordered energy, namely for a simple configuration
\(\betaeta=\sum_{i=1}^n\delta_{x_i}\),
\[
W_I^{\intE}(\betaeta)
\defeq
-2\sum_{1\leq i<j\leq n}\log|x_i-x_j|
+2\sum_{i=1}^n\int_I\log|x_i-y|\dd y
-\iint_{I^2}\log|x-y|\dd x\dd y,
\]
The structural input from the series of papers \cite{LS,EHL,DHLM} gives number rigidity for random configuration distributed according to $\bP$: conditionally on the exterior
of \(I_L\), the number of particles inside is the exterior-measurable
integer \(n_L(\bxi)\).  Write \(K_{L,\bxi}\) for this conditional law.
Section~\ref{SectionCanonicalFibre} constructs, for
\(\balpha_L\)-almost every \(\bxi\), a finite exterior potential
\(A_{\bxi}\) on the fixed particle-number fibre.  Namely, if we put
\(\Lambda_p=[-p/2,p/2]\), \(n=n_L(\bxi)\), and write
\(
 \betaeta=\sum_{i=1}^n\delta_{x_i},
 \br_n=\sum_{i=1}^n\delta_{a_i}
\)
for a fixed deterministic reference configuration, then
\[
 A_{\bxi}(\betaeta)
 =
 -\limsup_{\substack{p\to\infty\\p\in\N}}
 \sum_{y\in\bxi\cap(\Lambda_p\setminus I_L)}
 \log\left|
 \frac{\prod_{i=1}^n(x_i-y)}
      {\prod_{i=1}^n(a_i-y)}
 \right|.
\]
Changing \(\br_n\) changes \(A_{\bxi}\) only by an additive constant, which
cancels from the following normalised formula.
For \(n=n_L(\bxi)\geq1\) and every bounded Borel \(F\), we can write $K_{L,\bxi}(\dd\betaeta)$ explicitly,
\begin{equation}\label{EqStrategyConditionalCoordinates}
 \int_{\Conf(I_L)}F(\betaeta)K_{L,\bxi}(\dd\betaeta)
 \propto\displaystyle\int_{I_L^n}
 F\!\left(\sum_i\delta_{x_i}\right)
 \prod_{i<j}|x_i-x_j|^\beta
 \exp\!\left[-\beta A_{\bxi}\left(\sum_i\delta_{x_i}\right)\right]
 \prod_i\dd x_i.
\end{equation}
This is simply the particle-coordinate form of the gauge-invariant kernel from
Definition~\ref{DefDLREquations}. Here gauge invariance means that adding to \(A_{\bxi}\) a constant which may
depend on \(\bxi\), but not on the interior configuration, multiplies the
numerator and denominator by the same factor and hence leaves the normalised
conditional law unchanged. We neutralise (here and below, neutrality means that the relevant signed charge has total
mass zero; for a configuration \(\bzeta\) on an interval \(I\), this means
\((\bzeta-\mathrm{Leb}_I)(I)=0\), or equivalently \(\bzeta(I)=|I|\)) the conditional configuration by dilating \(I_L\) to the
interval \(I_n\), whose length equals its particle number:
\[
 \nu_{L,\bxi}=(T_{n/L})_*K_{L,\bxi}.
\]
On the empty sector set
\(\nu_{L,\bxi}=\delta_\varnothing\).
Thus, in the same particle coordinates (here $n=n_L(\bxi)$ is treated as fixed; otherwise we can view as integrals over $\Conf(\R)$),
\[
 \int_{\Conf(I_n)}F(\bzeta)\nu_{L,\bxi}(\dd\bzeta)
 \propto\displaystyle\int_{I_L^n}
 F\!\left(\sum_i\delta_{n x_i/L}\right)
 \prod_{i<j}|x_i-x_j|^\beta
 \exp\!\left[-\beta A_{\bxi}\left(\sum_i\delta_{x_i}\right)\right]
 \prod_i\dd x_i.
\]
Recall, the reference law is the circular \(\beta\)-ensemble on the circle of
circumference \(n\), $\bQ_{n,\beta}$ (this convention, used from now on, is slightly more convenient than the centered one used in the introductory part).  Explicitly, for \(n\geq1\),
\begin{equation}\label{EqStrategyCircularCoordinates}
 \int_{\Conf(I_n)}F(\bzeta)\bQ_{n,\beta}(\dd\bzeta)
\propto
 \displaystyle\int_{I_n^n}
 F\!\left(\sum_{i=1}^n\delta_{u_i}\right)
 \prod_{1\leq i<j\leq n}
 \left|2\sin\left(\frac{\pi(u_i-u_j)}n\right)\right|^\beta
 \prod_{i=1}^n\dd u_i.
\end{equation}
Set \(\bQ_{0,\beta}=\delta_\varnothing\).
The following relative entropy estimate that compares (in an averaged sense) $\nu_{L,\bxi}$ and $\bQ_{n_L(\bxi),\beta}$ is the key ingredient:
\begin{equation}\label{EqStrategyCircularEntropy}
 \mathcal H_L^{\mathrm{circ}}
 \defeq
 \int_{\Conf(I_L^c)}
 \Ent(\nu_{L,\bxi}|\bQ_{n_L(\bxi),\beta})
 \balpha_L(\dd\bxi)=o(L).
\end{equation}
Here and throughout, \(t_L=o(L)\) means that
\(\lim_{L\to\infty}t_L/L=0\).  We will explain how this estimate is proved
but first we explain how it completes the proof.

We note that the estimate \(\mathcal H_L^{\mathrm{circ}}=o(L)\) does not close the
argument directly through Pinsker's inequality, which would yield
total-variation convergence of the full laws only from an \(o(1)\) entropy
bound.  A suitable speed-\(n\) large-deviation principle for the empirical
field of \(\bQ_{n,\beta}\), together with uniqueness of its minimiser, would
instead transfer exponential concentration from the circular ensemble to the
candidate DLR solution law by a standard entropy inequality argument, since \(n\asymp L\) and
the entropy cost is \(o(L)\).  The large-deviation principle of
\cite{LS} is not stated in the circular periodic geometry required here and so does not immediately apply (although a technical adaptation must surely be possible, we do not pursue it).
The cyclic-gap argument below is more specialised and perhaps less intuitive,
but it is essentially elementary and completely self-contained.

Returning to the argument, put
\(p_{L,n}=\bP(\sN_L=n)\), and, when \(p_{L,n}>0\), average over all
exteriors in the same sector:
\begin{equation}\label{EqStrategyFixedSectorAverage}
 \overline\nu_{L,n}
 =\frac1{p_{L,n}}
 \int_{\{n_L(\bxi)=n\}}\nu_{L,\bxi}\balpha_L(\dd\bxi).
\end{equation}
For \(n\geq1\), write
\(x_1(\bgamma)<\cdots<x_n(\bgamma)\) for the ordered atoms of the simple
restriction \(\bgamma_{I_L}\).  Its particle-coordinate form is
\[
 \int_{\Conf(I_n)}F(\bzeta)\overline\nu_{L,n}(\dd\bzeta)
 =\frac1{p_{L,n}}
 \int_{\{\bgamma(I_L)=n\}}
 F\!\left(\sum_{i=1}^n\delta_{n x_i(\bgamma)/L}\right)\bP(\dd\bgamma).
\]
Thus \(\overline\nu_{L,n}\) is precisely the law of \(\bZ_L\), conditional
on \(\sN_L=n\), where
\(\bZ_L=(T_{\sN_L/L})_*\bC_{I_L}\) on \(\{\sN_L\geq1\}\)
and \(\bZ_L=\varnothing\) on the empty sector; on the zero-particle sector
\(\overline\nu_{L,0}=\delta_\varnothing\), and
when \(p_{L,n}=0\), set \(\overline\nu_{L,n}=\bQ_{n,\beta}\).  Convexity
of entropy gives the estimate
\begin{equation}\label{EqStrategyFixedSectorEntropy}
 \sum_{n\geq0}p_{L,n}
 \Ent(\overline\nu_{L,n}|\bQ_{n,\beta})
 \leq\mathcal H_L^{\mathrm{circ}}=o(L).
\end{equation}
The structural input also gives \(v(L)=o(L)\), so Chebyshev's inequality
shows that
\[
 \sum_{\{n<L/2\text{ or }n>2L\}}p_{L,n}
 \leq\frac{4v(L)}{L^2}=o(1).
\]

For \(n\geq1\), identify the endpoints of \(I_n\).  If
\(\bzeta=\sum_{i=1}^n\delta_{u_i}\), with
\(0<u_1<\cdots<u_n<n\), put \(u_{i+n}=u_i+n\), average uniformly over
\(j\in\{1,\ldots,n\}\), and record
\[
 (u_{j+1}-u_j,\ldots,u_{j+n}-u_{j+n-1})
 \in\mathcal S_n^\circ,
 \qquad
 \mathcal S_n^\circ=\left\{x\in(0,\infty)^n:\sum_{i=1}^nx_i=n\right\}.
\]
Let \(\rho_{L,n}\) and \(\mu_n\) be the resulting cyclic-gap laws under
\(\overline\nu_{L,n}\) and \(\bQ_{n,\beta}\), respectively.  Choosing
the root uniformly and then recording the corresponding cyclic-gap
vector cannot increase relative entropy.  Hence
\[
 \Ent(\rho_{L,n}\mid\mu_n)
 \leq
 \Ent(\overline\nu_{L,n}\mid\bQ_{n,\beta}),
\]
and the right-hand side is controlled after averaging over \(n\) by
\eqref{EqStrategyFixedSectorEntropy}.
Now, for a differentiable convex function \(\mathsf{V}\), recall that its Bregman divergence
from \(x\) to \(y\) is
\[
 D_\mathsf{V}(y,x)
 \defeq
 \mathsf{V}(y)-\mathsf{V}(x)-\nabla \mathsf{V}(x)\mathbin{\cdot}(y-x)\geq0.
\]
It measures the amount by which \(\mathsf{V}(y)\) lies above the tangent plane to
\(\mathsf{V}\) at \(x\).  The Bregman divergence of the circular gap potential
controls
\[
 \beta\sum_i\psi(y_i/x_i),
 \qquad
 \psi(q)=q-1-\log(q).
\]
Using this Bregman-divergence lower bound and adapting standard
transport inequalities from \cite{CETransport,BLTransport} to the
circular gap potential, we prove in Section~\ref{SectionPalm} that
there exists a cyclically invariant coupling \(\pi_{L,n}\) of random
gap vectors
\((\mathsf X^{L,n},\mathsf Y^{L,n})\), with respective laws
\(\mu_n\) and \(\rho_{L,n}\), such that
\begin{equation}\label{EqStrategyOneGap}
 \sum_{\substack{n\geq3\\L/2\leq n\leq2L}}p_{L,n}
 \E_{\pi_{L,n}}\!\left[
 \psi\left(\frac{\mathsf Y_1^{L,n}}{\mathsf X_1^{L,n}}\right)\right]
 \leq\frac{2\mathcal H_L^{\mathrm{circ}}}{\beta L}=o(1)
\end{equation}
and cyclic invariance gives the same bound for any sectorwise choice
\(i_n\in\{1,\ldots,n\}\). Since
\(\E_{\pi_{L,n}}[\mathsf X_i^{L,n}]=1\), a cutoff and the fact that \(\psi(q)\) is strictly positive away from \(q=1\) turn ratio control into absolute gap control. Note that, this is the decisive gain furnished by cyclic averaging.  Although
\(n\asymp L\) also tends to infinity, cyclic symmetry distributes the
averaged total cost \(\mathcal H_L^{\mathrm{circ}}=o(L)\) among
\(n\asymp L\) gaps, yielding \(o(1)\) control of any prescribed gap. We stress that we do not prove, or need for the subsequent argument, a comparison of the full growing gap vector. Then, a  union bound
couples every fixed block of gaps on both sides of the uniform root.  Tightness
of the reference rooted laws turns this fixed-block coupling into a coupling
in the local vague topology.

Now, recall that if \(\bQ\) is a stationary simple point process of intensity one, its
reduced Palm law \(\bQ^{!0}\) (the root particle is
deleted; this is what the symbol \({!0}\) is supposed to denote) is characterised by
\[
 \int_{\Conf(\R)}\sum_{x\in\bgamma}
 h\bigl(x,\theta_x(\bgamma-\delta_x)\bigr)\bQ(\dd\bgamma)
 =\int_\R\int_{\Conf(\R)}h(x,\betaeta)
 \bQ^{!0}(\dd\betaeta)\dd x
\]
for every non-negative Borel function
\(h:\R\times\Conf(\R)\to[0,+\infty]\). Returning to our setup, periodically reconstructing the reduced rooted configuration from the
reference gaps gives the reduced Palm law of the stationary periodic lift of
\(\bQ_{n,\beta}\).  It can be shown that this law converges to
\((\Sine_\beta)^{!0}\).
On the other hand for the DLR solution $\bP$, the corresponding ordinary, non-periodised rooted law is
\begin{equation}\label{EqActualRootLaw}
 \begin{split}
 \mathfrak R_L(F)=\bP(\sN_L=0)F(\varnothing)+\E_{\bP}\!\left[
 \1_{\{\sN_L>0\}}\frac1{\sN_L}
 \sum_{x\in\bC\cap I_L}
 F\!\left((T_{\sN_L/L})_*
 \theta_x(\bC_{I_L}-\delta_x)\right)\right].
 \end{split}
\end{equation}
Campbell's
formula, the fact that \(\sN_L/L\to1\) also under the probability law obtained by selecting the root through the Campbell measure, and some technical work show that \(\mathfrak R_L\) converges to
\(\bP^{!0}\).  The gap coupling explained above then shows that the same law converges to
\((\Sine_\beta)^{!0}\).  Hence, we can conclude
\[
 \bP^{!0}=(\Sine_\beta)^{!0}.
\]
Since the processes are stationary, simple and have the same intensity one, Palm inversion
gives \(\bP=\Sine_\beta\) which completes the proof.

We now explain how the key relative entropy estimate
\eqref{EqStrategyCircularEntropy} is proved.  The structural input collected
in Proposition~\ref{PropInputs} gives
\begin{equation}\label{EqStrategyVariance}
 \E_{\bP}[\sN_L]=L,\qquad
 v(L)\leq C_{\bP}L,\qquad v(L)=o(L).
\end{equation}
Put \(b=\beta/2\).  For \(n\geq1\), dilate the circular ensemble back to
\(I_L\),
\[
 \bR_{n,L}=(T_{L/n})_*\bQ_{n,\beta},
 \qquad \bR_{0,L}=\delta_\varnothing.
\]
When \(n=n_L(\bxi)\geq1\), the laws \(K_{L,\bxi}\) and
\(\bR_{n,L}\) are Gibbs measures with respect to the same binomial law on
\(\Conf_n(I_L)\).  
Sections~\ref{SectionCanonicalFibre} and
\ref{SectionCircularRecovery} identify their potentials, up to fibrewise
constants, as follows
\[
 V_K(\betaeta)
 =bW_{I_L}^{\intE}(\betaeta)
  +\beta\widehat A_{\bxi}(\betaeta),
 \qquad
 V_R(\betaeta)
 =bW_n^{\per}\bigl((T_{n/L})_*\betaeta\bigr),
\]
where with \(\betaeta=\sum_{i=1}^n\delta_{x_i}\), one has
\begin{align*}
 \widehat A_{\bxi}(\betaeta)
 &=
 A_{\bxi}(\betaeta)
 +\sum_{i=1}^n
 \left[
   L-x_i\log x_i-(L-x_i)\log(L-x_i)
 \right],\\
 W_n^{\per}\bigl((T_{n/L})_*\betaeta\bigr)
 &=
 -2\sum_{1\leq i<j\leq n}
 \log\left|
 2\sin\left(\frac{\pi(x_i-x_j)}{L}\right)
 \right|
 +n\log\left(\frac{n}{2\pi}\right).
\end{align*}
On the empty sector both laws are \(\delta_\varnothing\), so the entropy and
all comparison terms vanish. For two Gibbs laws on the same fibre, with integrable potential difference
\(\Phi=V_R-V_K\), Jensen's inequality gives
\begin{equation}\label{EqEntropyIntroStrategyInequality}
 \Ent(K_{L,\bxi}|\bR_{n,L})
 \leq
 \int_{\Conf_n(I_L)}
\int_{\Conf_n(I_L)}
[\Phi(\bX)-\Phi(\bY)]
 K_{L,\bxi}(\dd\bX)\bR_{n,L}(\dd\bY).
\end{equation}
For \(\bzeta\in\Conf_n(I_n)\), define the centered linear statistic
\[
 \mathcal A_n(\bzeta)
 \defeq
 \int_{I_n}
 \left[n-u\log u-(n-u)\log(n-u)\right]
 \dd(\bzeta-\mathrm{Leb}_{I_n})(u),
\]
with \(0\log0=0\). It is the centred linear statistic generated by the change in the background measure under the dilation from \(I_L\) to \(I_n\). Now, the exact dilation identity of Section~\ref{SectionNeutralisation} gives, with
\(D=n-L\),
\[
 \begin{aligned}
 W_n^{\per}\bigl((T_{n/L})_*\betaeta\bigr)-W_{I_L}^{\intE}(\betaeta)
 &=-\mathsf{Err}_n^{\per}\bigl((T_{n/L})_*\betaeta\bigr)
   -\frac{2D}{n}\mathcal A_n\bigl((T_{n/L})_*\betaeta\bigr)+c_{n,L},\\
 c_{n,L}&=D^2\log(L)-\frac32D^2-n\log(L/n),
 \end{aligned}
\]
where \(c_{n,L}\) is constant on the fibre and hence cancels from the
comparison. This is key as we cannot really control this quantity on its own.  Rotation invariance of the circular ensemble moreover gives 
\(\bQ_{n,\beta}(\mathcal A_n)=0\).  On the other hand, writing
\(\mathsf{Br}_L(s)=\sD_s-(s/L)\sD_L\), Stieltjes integration by parts and
stationarity give, respectively, the pathwise representation and the second-moment bound for \(\mathsf{Br}_L(s)\)
\[
 \frac{\sD_L}{\sN_L}\mathcal A_{\sN_L}(\bZ_L)
 =-\frac{\sD_L}{L}\int_0^L
 \mathsf{Br}_L(s)\log\left(\frac{L-s}{s}\right)\dd s,
 \qquad
 \E_{\bP}[\mathsf{Br}_L(s)^2]
 \leq C_{\bP}\frac{s(L-s)}L,
\]
with the quotient set equal to zero on the empty sector.  Consequently,
\[
 \E_{\bP}\!\left[
  \left|\frac{\sD_L}{\sN_L}
  \mathcal A_{\sN_L}(\bZ_L)\right|
 \right]
 \leq C\sqrt{Lv(L)}=o(L).
\]

We now turn our attention to exterior contribution and note that it is never averaged as a difference of two
separately chosen gauges.  Indeed, \(\widehat A_{\bxi}\) is defined only
up to an additive constant depending on \(\bxi\), so we first form its
difference at two configurations in the same particle-number fibre, where
this constant cancels, and only then average over the exterior.  For two interior configurations \(\bX,\betaeta\) in the same
particle-number fibre, define the gauge-invariant boundary difference
\[
 \sB_{I_L}(\betaeta,\bX,\bxi)
 \defeq
 \widehat A_{\bxi}(\betaeta)-\widehat A_{\bxi}(\bX).
\]Thus, if
\(\bX_L\) has conditional law \(K_{L,\bxi}\) and \(\bY_L\) has law
\(\bR_{n,L}\), independently given the exterior, then
\[
 \bR_{n,L}(\widehat A_{\bxi})
 -K_{L,\bxi}(\widehat A_{\bxi})
 =
 \E\!\left[
 \sB_{I_L}(\bY_L,\bX_L,\bxi)\mid\bxi
 \right].
\]
Write \(\bP_L^{\mathrm{rep}}\) for the resulting joint law of
\((\bC,\bY_L)\).  Under this law, \(\bX_L\) and \(\bY_L\) have the same
rigid particle number \(n\), and hence the replacement charge is exactly
neutral:
\[
 (\bY_L-\bX_L)(I_L)=n-n=0.
\]
This neutrality cancels configuration-independent constants and importantly the leading
far-field contribution to the exterior interaction. The Stieltjes
representation in Section~\ref{SectionBoundary}, together with the variance
bounds in \eqref{EqStrategyVariance}, yields
\[
 \E_{\bP_L^{\mathrm{rep}}}
 \left[|\sB_{I_L}(\bY_L,\bX_L,\bC_{I_L^c})|\right]=o(L).
\]

Finally, put
\[
 \mathsf{Err}_n^{\mathrm{per}}
 =W_{I_n}^{\intE}-W_n^{\mathrm{per}},
 \qquad \mathsf{Err}_0^{\mathrm{per}}(\varnothing)=0.
\]
An elementary technical computation gives an exact integral formula for \( \mathsf{Err}_n^{\mathrm{per}}\) and further estimates give the following periodic-to-intrinsic comparison under the circular ensemble
\[
 \E_{\bQ_{n,\beta}}[|\mathsf{Err}_n^{\mathrm{per}}|]
 \leq C_\beta\log^2(2+n).
\]
With this aforementioned exact formula as input plugging in the actual neutralised restriction
Lemma~\ref{LemActualPeriodicError} gives us the following
\begin{equation}\label{EqStrategyActualPeriodicError}
 \E_{\bP}\!\left[
  |\mathsf{Err}_{\sN_L}^{\mathrm{per}}(\bZ_L)|
 \right]=o(L).
\end{equation}
Averaging the fixed-fibre comparison entropy inequality \eqref{EqEntropyIntroStrategyInequality} and taking absolute values therefore
gives
\begin{equation}\label{EqStrategyDirectComparison}
 \begin{aligned}
 \mathcal H_L^{\mathrm{circ}}
 \leq{}&
 b\E_{\bP}[|\mathsf{Err}_{\sN_L}^{\mathrm{per}}(\bZ_L)|]
 +b\sum_{n\geq0}p_{L,n}
   \E_{\bQ_{n,\beta}}[|\mathsf{Err}_n^{\mathrm{per}}|]\\
 &+2b\E_{\bP}\!\left[
   \left|\frac{\sD_L}{\sN_L}
   \mathcal A_{\sN_L}(\bZ_L)\right|\right]
 +\beta\E_{\bP_L^{\mathrm{rep}}}
 [|\sB_{I_L}(\bY_L,\bX_L,\bC_{I_L^c})|].
 \end{aligned}
\end{equation}
Finally we can easily bound,
\[
 \sum_{n\geq0}p_{L,n}
 \E_{\bQ_{n,\beta}}[|\mathsf{Err}_n^{\mathrm{per}}|]
 \leq C_\beta\E_{\bP}[\log^2(2+\sN_L)]
 \leq C_\beta(1+\sqrt L)=o(L).
\]
Hence, all the terms on the right-hand side of
\eqref{EqStrategyDirectComparison} are thus \(o(L)\) and this proves
\eqref{EqStrategyCircularEntropy}.

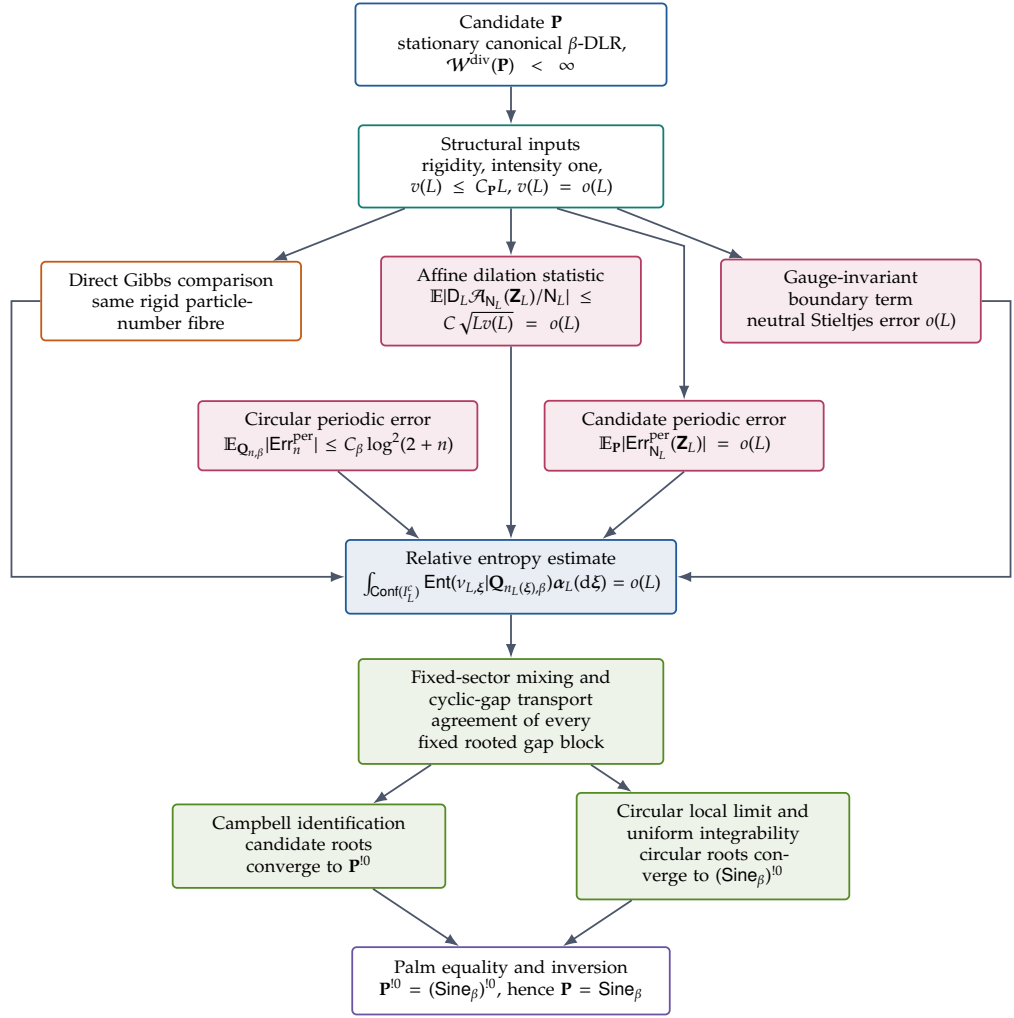
\begin{figure}[p]
\centering
\resizebox{.97\linewidth}{!}{%
\begin{tikzpicture}
\node[dlrnode,fill=white,text width=4.5cm] (hyp) at (0,0)
 {Candidate \(\bP\)\\stationary canonical \(\beta\)-DLR,\\
  \(\cW^{\mathrm{div}}(\bP)<\infty\)};

\node[energynode,fill=white,text width=4.4cm] (inputs) at (0,-1.9)
 {Structural inputs\\rigidity, intensity one,\\
  \(v(L)\leq C_{\bP}L\), \(v(L)=o(L)\)};

\node[entropynode,fill=white,text width=3.7cm] (fibre) at (-5.3,-4.0)
 {Direct Gibbs comparison\\same rigid particle-number fibre};
\node[boundarynode,text width=3.7cm] (affine) at (0,-4.0)
 {Affine dilation statistic\\
  \(\E|\sD_L\mathcal A_{\sN_L}(\bZ_L)/\sN_L|
  \leq C\sqrt{Lv(L)}=o(L)\)};
\node[boundarynode,text width=3.7cm] (boundary) at (5.3,-4.0)
 {Gauge-invariant boundary term\\neutral Stieltjes error \(o(L)\)};

\node[boundarynode,text width=4.0cm] (referr) at (-2.7,-6.1)
 {Circular periodic error\\
  \(\E_{\bQ_{n,\beta}}|\mathsf{Err}^{\mathrm{per}}_n|
  \leq C_\beta\log^2(2+n)\)};
\node[boundarynode,text width=4.0cm] (acterr) at (2.7,-6.1)
 {Candidate periodic error\\
  \(\E_{\bP}|\mathsf{Err}^{\mathrm{per}}_{\sN_L}(\bZ_L)|=o(L)\)};

\node[dlrnode,text width=4.8cm] (bridge) at (0,-8.3)
 {Relative entropy estimate\\
  \(\int_{\Conf(I_L^c)}\Ent(\nu_{L,\bxi}|\bQ_{n_L(\bxi),\beta})
    \balpha_L(\dd\bxi)=o(L)\)};

\node[recoverynode,text width=4.4cm] (gaps) at (0,-10.4)
 {Fixed-sector mixing and cyclic-gap transport\\
  agreement of every fixed rooted gap block};

\node[recoverynode,text width=3.9cm] (campbell) at (-3.15,-12.5)
 {Campbell identification\\candidate roots converge to \(\bP^{!0}\)};
\node[recoverynode,text width=3.9cm] (vv) at (3.15,-12.5)
 {Circular local limit and uniform integrability\\
  circular roots converge to \((\Sine_\beta)^{!0}\)};

\node[conclusionnode,fill=white,text width=4.6cm] (palm) at (0,-14.6)
 {Palm equality and inversion\\
  \(\bP^{!0}=(\Sine_\beta)^{!0}\), hence \(\bP=\Sine_\beta\)};

\draw[proofarrow] (hyp)--(inputs);
\draw[proofarrow] (inputs)--(fibre);
\draw[proofarrow] (inputs)--(affine);
\draw[proofarrow] (inputs)--(boundary);
\draw[proofarrow]
 ($(inputs.south)!0.30!(inputs.south east)$)
 --($(acterr.north)+(0,2.4)$)--(acterr.north);
\draw[proofarrow] (fibre.west)--++(-.45,0)|-(bridge.west);
\draw[proofarrow] (affine)--(bridge);
\draw[proofarrow] (boundary.east)--++(.45,0)|-(bridge.east);
\draw[proofarrow]
 (referr.south)--($(bridge.north)!0.55!(bridge.north west)$);
\draw[proofarrow]
 (acterr.south)--($(bridge.north)!0.55!(bridge.north east)$);
\draw[proofarrow] (bridge)--(gaps);
\draw[proofarrow] (gaps)--(campbell);
\draw[proofarrow] (gaps)--(vv);
\draw[proofarrow] (campbell)--(palm);
\draw[proofarrow] (vv)--(palm);
\end{tikzpicture}%
}
\caption{This is a  cartoon of the proof architecture.  To prove the main
relative entropy estimate, shaded in blue, we use a direct Gibbs comparison
on the same rigid particle-number fibre between the conditional law
\(K_{L,\bxi}\) of the candidate configuration in \(I_L\) and the rescaled
circular reference law \(\bR_{n,L}\), where \(n=n_L(\bxi)\).  The four main
estimates entering this comparison are shaded in red; the technically most
delicate is the estimate of the gauge-invariant boundary interaction term.
The principal quantitative input on the candidate side is the sublinear
discrepancy estimate \(v(L)=o(L)\), while the circular periodic error is
controlled directly under \(\bQ_{n,\beta}\).  Finally, the relative entropy
estimate feeds into the cyclic-gap transport and Palm-identification
argument, whose main components are shaded in green, which completes the proof.}
\label{FigProofArchitecture}
\end{figure}

\paragraph{Acknowledgements} In chronological order, I am grateful to Sasha Bufetov for first telling me about conditional measures of determinantal point processes and DLR equations, I am grateful to Arno Kuijlaars for explaining the $\beta=2$ proof of uniqueness to me, I am grateful to Kohei Suzuki, Thomas Lebl\'{e}, David Dereudre, Martin Huesmann for mentioning the general $\beta$ uniqueness problem to me. I am particularly grateful to Kohei Suzuki for our collaboration on diffusions on configuration space.

\paragraph{AI disclosure} I have made use of AI tools (ChatGPT 5.5, 5.6 Sol) for  some exploration of ideas but mainly for help with computations and abstract measure theoretic arguments, literature review and editing of the manuscript. I have verified all mathematical arguments and modified them in the course of the past months. Any remaining mistakes are my own responsibility.

\section{Setting and imported results}\label{SectionSetting}

This section fixes the rigorous configuration, DLR, entropy, and electric
energy framework and recalls the fundamental imported results from the literature that we use.

\subsection{Configuration space and the rigorous DLR formulation}

We now supply the measure-theoretic formulation behind
\eqref{EqIntroExteriorWeight}--\eqref{EqIntroDLR}.  For a Borel set
\(\mathscr A\subset\R\), put
\[
 \Conf(\mathscr A)\defeq
 \{\bgamma\in\Conf(\R):\bgamma(\mathscr A^c)=0\},
\]
with the sigma-field inherited from \(\Conf(\R)\).  A configuration \(\bgamma\) is simple if
\(\bgamma(\{x\})\leq1\) for every \(x\in\R\).  We use
\[
 \mathsf{g}(x)\defeq-\log(|x|)\quad(x\neq0),\qquad \mathsf{g}(0)\defeq0,
 \qquad
 \Delta\defeq\{(x,x):x\in\R\}.
\]
The assigned value \(\mathsf{g}(0)=0\) is immaterial because the atomic diagonal is
removed, and the diagonal has zero Lebesgue measure in the remaining terms.
For a bounded Borel set \(A\subset\R\), write \(|A|\) for its Lebesgue
measure, \(\mathrm{Leb}_A\) for Lebesgue measure restricted to \(A\), and
\(\mathrm{Leb}=\mathrm{Leb}_{\R}\).

Let \(\N=\{1,2,\ldots\}\) and \(\mathbb Z_+=\N\cup\{0\}\).  If
\(\Lambda\subset\R\) is bounded and Borel with \(|\Lambda|>0\), put
\[
 \Conf_n(\Lambda)\defeq
 \{\betaeta\in\Conf(\Lambda):\betaeta(\Lambda)=n\}
 \quad(n\in\mathbb Z_+).
\]
Let \(\Bin_{n,\Lambda}\) be the law on \(\Conf_n(\Lambda)\) of
\(\sum_{j=1}^n\delta_{\mathsf U_j}\), where
\(\mathsf U_1,\ldots,\mathsf U_n\) are independent with common law
\(|\Lambda|^{-1}\mathrm{Leb}_\Lambda\), and set
\(\Bin_{0,\Lambda}=\delta_{\varnothing}\).
Equivalently, for every bounded Borel \(F\),
\[
 \int_{\Conf_n(\Lambda)}F(\betaeta)\Bin_{n,\Lambda}(\dd\betaeta)
 =\frac1{|\Lambda|^n}\int_{\Lambda^n}
 F\left(\sum_{i=1}^n\delta_{x_i}\right)
 \prod_{i=1}^n\dd x_i.
\]
  A function on \(\Conf(\R)\) is
local if it depends only on the restriction to some bounded interval.
The first intensity measure of a point-process law \(\bQ\)
is
\[
 A\longmapsto\int_{\Conf(\R)}\bgamma(A)\bQ(\dd\bgamma).
\]
A universally measurable set is measurable for the completion of every
Borel probability measure; completed measurability allows modification on a
null set.  Recall that \(\bC\) denotes the canonical configuration.
The process has intensity one when
\[
 \E_\bP\left[\bC(\mathscr A)\right]=|\mathscr A|
\]
 for every bounded Borel \(\mathscr A\subset\R\).

The following definition packages the finite-volume Hamiltonian, exterior
move cost, and canonical Gibbs kernel used in the DLR equation and in
Section~\ref{SectionCanonicalFibre}.

\begin{defn}\label{DefCanonicalKernel}
Let \(\beta>0\) and let \(\Lambda\subset\R\) be bounded and Borel with
\(|\Lambda|>0\).  For each integer \(p\geq1\), set
\(\Lambda_p=[-p/2,p/2]\).  For
\(\bgamma\in\Conf(\R)\), \(n=\bgamma(\Lambda)\), and
\(\betaeta\in\Conf_n(\Lambda)\), define
\[
 \mathsf{H}_\Lambda(\betaeta)\defeq
 \frac12\iint_{\Lambda^2\setminus\Delta}
 \mathsf{g}(x-y)\dd\betaeta(x)\dd\betaeta(y)
\]
when \(\betaeta\) is simple, and set \(\mathsf{H}_\Lambda(\betaeta)=+\infty\)
otherwise.  For an integer \(p\geq1\) such that
\(\Lambda\subset\Lambda_p\), define the truncated move function
\[
 \mathsf{M}_{\Lambda,\Lambda_p}(\betaeta,\bgamma)\defeq
 \iint_{\Lambda\times(\Lambda_p\setminus\Lambda)}
 \mathsf{g}(x-y)\dd(\betaeta-\bgamma_\Lambda)(x)\dd\bgamma(y).
\]
We call \(\bgamma\) \((\beta,\Lambda)\)-admissible if there is a finite
function \(\mathsf{M}_{\Lambda,\R}(\mathord\cdot,\bgamma)\) on
\(\Conf_n(\Lambda)\) such that
\[
 \lim_{\substack{p\to\infty\\p\in\N}}
 \sup_{\betaeta\in\Conf_n(\Lambda)}
 \left|\mathsf{M}_{\Lambda,\Lambda_p}(\betaeta,\bgamma)
 -\mathsf{M}_{\Lambda,\R}(\betaeta,\bgamma)\right|=0
\]
and
\[
 0<Z_{\Lambda,\R}^{\beta}(\bgamma)\defeq
 \int_{\Conf_n(\Lambda)}
 \exp\left[-\beta\left(\mathsf{H}_\Lambda(\betaeta)
 +\mathsf{M}_{\Lambda,\R}(\betaeta,\bgamma)\right)\right]
 \Bin_{n,\Lambda}(\dd\betaeta)<\infty.
\]
For \((\beta,\Lambda)\)-admissible \(\bgamma\), define
\[
 \mathsf G_{\Lambda,\R}^{\beta}(\dd\betaeta,\bgamma)\defeq
 \frac{\exp\left[-\beta\left(\mathsf{H}_\Lambda(\betaeta)
 +\mathsf{M}_{\Lambda,\R}(\betaeta,\bgamma)\right)\right]}
 {Z_{\Lambda,\R}^{\beta}(\bgamma)}
 \Bin_{n,\Lambda}(\dd\betaeta).
\]
 Equivalently, for every bounded Borel \(F\), its particle-coordinate form is
 \begin{equation}\label{EqCanonicalKernelCoordinates}
 \begin{split}
 &\int_{\Conf_n(\Lambda)}F(\betaeta)
  \mathsf G_{\Lambda,\R}^{\beta}(\dd\betaeta,\bgamma)\\
 &\quad=\frac1{|\Lambda|^n Z_{\Lambda,\R}^{\beta}(\bgamma)}
  \int_{\Lambda^n}
  F\!\left(\sum_{i=1}^n\delta_{x_i}\right)
  \prod_{1\leq i<j\leq n}|x_i-x_j|^\beta
  \exp\!\left[-\beta\mathsf M_{\Lambda,\R}
   \left(\sum_{i=1}^n\delta_{x_i},\bgamma\right)\right]
  \prod_{i=1}^n\dd x_i.
 \end{split}
 \end{equation}
\end{defn}

For a Borel set \(\mathscr A\subset\R\), let
\[
 \scrF_{\mathscr A}\defeq
 \sigma\{\bgamma\mapsto\bgamma(\mathscr B):
 \mathscr B\subset\mathscr A\text{ is bounded and Borel}\},
 \qquad
 \scrF_\Lambda^{\mathrm{can}}\defeq
 \scrF_{\Lambda^c}\vee\sigma\{\bgamma\mapsto\bgamma(\Lambda)\}.
\]

For a sigma-field \(\mathscr F\), write
\(\overline{\mathscr F}^{\,\bP}\) for its \(\bP\)-completion. We can finally define the DLR equations.

\begin{defn}\label{DefDLREquations}
Let \(\beta>0\).  A point process \(\bP\) satisfies the canonical
\(\beta\)-DLR equations if, for every bounded Borel
\(\Lambda\subset\R\) with \(|\Lambda|>0\), there is a universally measurable
\(\bP\)-full set \(\mathscr A_\Lambda^{\mathrm{DLR}}\) of
\((\beta,\Lambda)\)-admissible configurations such that the kernel in
Definition~\ref{DefCanonicalKernel}, defined on
\(\mathscr A_\Lambda^{\mathrm{DLR}}\),
admits an
\(\overline{\scrF_\Lambda^{\mathrm{can}}}^{\,\bP}\)-measurable
probability-kernel extension, still denoted by
\(\mathsf G_{\Lambda,\R}^{\beta}\), satisfying
\begin{equation}\label{EqDLR}
 \int_{\Conf(\R)}F(\bgamma)\bP(\dd\bgamma)
 =
 \int_{\Conf(\R)}
 \int_{\Conf(\Lambda)}
 F(\betaeta+\bgamma_{\Lambda^c})
 \mathsf G_{\Lambda,\R}^{\beta}(\dd\betaeta,\bgamma)\bP(\dd\bgamma)
\end{equation}
for every bounded Borel \(F\colon\Conf(\R)\to\R\).
\end{defn}

If \(\bgamma(\{0\})=0\) and
\(\betaeta=\sum_{j=1}^n\delta_{x_j}\), then at a finite cutoff the density in
Definition~\ref{DefCanonicalKernel} is, up to a factor independent of the
resampled points,
\[
 \prod_{1\leq j<k\leq n}|x_j-x_k|^\beta
 \prod_{j=1}^n
 \prod_{\substack{u\in\bgamma_{\Lambda^c}\\|u|\leq p/2}}
 \left|1-\frac{x_j}{u}\right|^\beta.
\]
Consequently, whenever the exterior product
\eqref{EqIntroExteriorWeight} exists, normalization gives exactly
\eqref{EqIntroDensity}; the multiplicative gauge cancels. 
Definition~\ref{DefDLREquations} is the precise, gauge-invariant formulation
used below.

Constants denoted by \(C\), with or without subscripts, are finite and may
change from line to line; their subscripts record their permitted
dependencies.

\subsection{Configurations, entropy, and ordered energy}

The intrinsic ordered energy is the finite-volume quantity entering the exact
affine comparison with the periodic circular energy.

\begin{defn}
Let \(I\subset\R\) be a bounded interval.  For
\(\betaeta=\sum_{i=1}^n\delta_{x_i}\in\Conf(I)\) simple, define
\begin{equation}\label{EqIntrinsicDefinition}
 \begin{aligned}
 W_I^{\intE}(\betaeta)
 &\defeq
 \iint_{I^2\setminus\Delta}\mathsf{g}(x-y)
 \dd(\betaeta-\mathrm{Leb}_I)(x)\dd(\betaeta-\mathrm{Leb}_I)(y)\\
 &=\sum_{i\ne j}\mathsf{g}(x_i-x_j)
 -2\sum_{i=1}^n\int_I\mathsf{g}(x_i-y)\dd y
 +\iint_{I^2}\mathsf{g}(x-y)\dd x\dd y.
 \end{aligned}
\end{equation}
 Set
\(W_I^{\intE}(\betaeta)=+\infty\) otherwise.
\end{defn}
The adjective ordered records that the double integral counts both
orientations of each pair; this is the source of the factor \(2\) in the
Hamiltonian comparison below.

We next fix the relative-entropy convention used below.

\begin{defn}
For probability measures \(\mu\) and \(\nu\)
on a common measurable space \(\mathsf X\), define
\[
 \Ent(\mu|\nu)\defeq
 \begin{cases}
 \displaystyle\int_{\mathsf X}
 \log\left(\left(\frac{\dd\mu}{\dd\nu}\right)(x)\right)
 \mu(\dd x),&\mu\ll\nu,\\
 +\infty,&\mu\not\ll\nu.
 \end{cases}
\]
Here \(\mu\ll\nu\) means that \(\mu\) is absolutely continuous with respect
to \(\nu\).
\end{defn}

\subsection{Electric energies}

The following compatibility and truncation conventions turn configurations
into two-dimensional electric fields and underlie every electric estimate.
For every open set \(U\subseteq\R^2\), write
\(\mathscr D'(U)\defeq\bigl(C_c^\infty(U)\bigr)'\) for the space of
distributions on \(U\).

\begin{defn}
Fix \(1<q_{\mathrm{el}}<2\), and put
\(\mathscr E=\mathrm{L}^{q_{\mathrm{el}}}_{\mathrm{loc}}(\R^2;\R^2)\),
endowed with its strong local \(\mathrm L^{q_{\mathrm{el}}}\) topology and
Borel sigma-field.
For \(\mathrm E\in\mathscr E\) and \(t\in\R\), set
\[
 (\theta_t\mathrm E)(X)\defeq \mathrm E(X+(t,0)).
\]
A field \(\mathrm E\in\mathscr E\) is divergence-compatible with
\(\bgamma\in\Conf(\R)\) if
\[
 -\operatorname{div}(\mathrm E)=2\pi(\bgamma-\mathrm{Leb})
\]
in \(\mathscr D'(\R^2)\), where \(\bgamma\) and \(\mathrm{Leb}\) are embedded in
\(\R\times\{0\}\).
For \(a\in\R\), write \(a^+=\max(a,0)\) and
\(a^-=\max(-a,0)\).  For \(\eps\in(0,1)\), set
\begin{equation}\label{EqTruncationDefinition}
 f_\eps(X)\defeq
 \left(\log\left(\frac{\eps}{|X|}\right)\right)^+
 \quad(X\neq0),\qquad f_\eps(0)\defeq0,\qquad
 \mathrm E_\eps\defeq \mathrm E-\int_\R\nabla f_\eps
 (\mathord\cdot-(x,0))\dd\bgamma(x).
\end{equation}
The gradient \(\nabla f_\eps\) is the weak gradient; the assigned value at
\(X=0\) is immaterial.
\end{defn}

The resulting process-level energy defined next appears as the main non-trivial hypothesis in our  result.

\begin{defn}\label{DefElectricEnergies}
Define
\[
 \widetilde{\cW}^{\mathrm{div}}(\bgamma)
 \defeq\inf_{\mathrm E:\,\mathrm E\text{ divergence-compatible with }\bgamma}
 \lim_{\eps\downarrow0}\limsup_{R\to\infty}
 \left(\frac{1}{2R}\frac{1}{2\pi}
 \int_{[-R,R]\times\R}|\mathrm E_\eps(X)|^2\dd X+\log(\eps)\right).
\]
The infimum of the empty set is \(+\infty\).  The extended-real cutoff
limits and their normalization are those of
\cite[(2.23)--(2.25) and Sections~2.7.3--2.7.4]{LS} and
\cite[Definition~2.6]{EHL}.  For a stationary process \(\bP\), put
\[
 \cW^{\mathrm{div}}(\bP)
 \defeq\E_\bP\left[\widetilde{\cW}^{\mathrm{div}}(\bC)\right].
\]
\end{defn}

We also use the universal lower bound
\(\widetilde{\cW}^{\mathrm{div}}\geq-C_{\mathrm{el}}\),
which follows from the truncation-monotonicity argument of
\cite[Proposition~2.4]{PS}, and the Gauss--flux consequence
\cite[Lemma~2.1]{PS} that every configuration of finite renormalised
electric energy has asymptotic density one.  

\subsection{The imported results}

We collect the following fundamental results from \cite{PS,LS,EHL,DHLM}. Part~\textup{(d)} simply records the
uniform admissibility already contained in Definition~\ref{DefDLREquations}.

\begin{prop}\label{PropInputs}
Let \(\bP\) satisfy the hypotheses of Theorem~\ref{ThmMain}.
\begin{enumerate}[label=\textup{(\alph*)}]
\item \(\bP\) is number-rigid: for every bounded Borel
\(\Lambda\subset\R\), the random variable \(\bC(\Lambda)\) is measurable
with respect to the \(\bP\)-completion of \(\scrF_{\Lambda^c}\).
\item \(\bP\) has intensity one, and there is \(C_{\bP}<\infty\) such
that, for every bounded interval \(J\subset\R\),
\[
 \E_{\bP}[(\bC(J)-|J|)^2]\leq C_{\bP}|J|.
\]
\item The number variance is sublinear:
\[
 \lim_{L\to\infty}
 \frac{\E_{\bP}[(\bC([0,L])-L)^2]}{L}=0.
\]
\item For every bounded Borel set \(\Lambda\subset\R\) with
\(|\Lambda|>0\), there is a \(\bP\)-full set on which the exterior move
functions converge uniformly on the fixed-particle-number fibre:
\[
 \lim_{\substack{p\to\infty\\p\in\N}}
 \sup_{\betaeta\in\Conf_{\bC(\Lambda)}(\Lambda)}
 \left|
 \mathsf M_{\Lambda,\Lambda_p}(\betaeta,\bC)
 -\mathsf M_{\Lambda,\R}(\betaeta,\bC)
 \right|=0.
\]
The limit is finite on the entire fibre.
\end{enumerate}
\end{prop}

The DLR density is absolutely continuous on every particle-number fibre and
therefore rules out collisions giving the following.

\begin{lem}\label{LemSimplicity}
Every process satisfying Definition~\ref{DefDLREquations} is supported on
simple configurations.
\end{lem}

Finally, DHLM prove that \(\Sine_\beta\) satisfies the canonical
\(\beta\)-DLR equations.  Their uniform convergence of the exterior move
functions identifies their normalized conditional density with the kernel
in Definition~\ref{DefDLREquations}.

\begin{lem}\label{LemDHLMKernelBridge}
For every \(\beta>0\), the process \(\Sine_\beta\) satisfies
Definition~\ref{DefDLREquations}.
\end{lem}

The canonical resampling identity is
\cite[Lemma~2.3 and Theorem~2.1\textup{(C*)}]{DHLM}, written in the
normalization of Definition~\ref{DefCanonicalKernel}.  Passing to a Borel
version of the regular conditional law and intersecting the equality sets
for a countable determining class gives the completed-kernel formulation
required in Definition~\ref{DefDLREquations}.

\section{Some finite-volume algebra}\label{SectionNeutralisation}

This section records the exact affine algebra needed in the direct Gibbs
comparison.  Dilation separates a random non-neutral sector from an exactly
neutral interval, and a second-moment bound for the centered counting discrepancy makes the resulting mixed
linear statistic sublinear after averaging.

\subsection{Exact dilation algebra}

Recall that \(I_s=[0,s]\) for \(s>0\), \(I_0=\varnothing\), and
\(T_a(u)=au\) for \(a>0\); push-forwards under this map are
written \((T_a)_*\mu\).
For \(m>0\), define
\[
 U_m(u)\defeq\int_0^m \mathsf{g}(u-v)\dd v.
\]
For \(n\in\N\) and \(\bzeta\in\Conf_n(I_n)\), set
\[
 \mathcal A_n(\bzeta)\defeq
 \int_{I_n}U_n(u)\dd(\bzeta-\mathrm{Leb}_{I_n})(u).
\]
Set \(\mathcal A_0(\varnothing)=0\). The statistic \(\mathcal A_n\) is the cross term created when dilation
changes the density of the uniform background.  It is the only
configuration-dependent correction in the affine identity below.
Direct integration gives
\begin{equation}\label{EqUm}
 \begin{aligned}
 U_m(u)&=m-u\log(u)-(m-u)\log(m-u)
          &&(0\leq u\leq m),\\
 U_m'(u)&=\log\left(\frac{m-u}{u}\right)
          &&(0<u<m).
 \end{aligned}
\end{equation}

Neutralising a random particle-number sector requires an exact dilation
identity.  The next lemma separates its intrinsic energy into a neutral-sector
term, a linear statistic, and explicit discrepancy corrections.

\begin{lem}
Let \(L>0\), let \(n\in\N\), put \(D=n-L\), and let \(\bzeta\) be a simple
configuration of \(n\) points in \(I_n\).  If
\(\betaeta=(T_{L/n})_*\bzeta\) is the image of \(\bzeta\) under
\(u\mapsto (L/n)u\), then
\begin{equation}\label{EqAffineIdentity}
 W_{I_L}^{\intE}(\betaeta)+D^2\log(L)
 =
 W_{I_n}^{\intE}(\bzeta)
 +\frac{2D}{n}\mathcal A_n(\bzeta)
 +\frac32D^2+n\log\left(\frac{L}{n}\right).
\end{equation}
\end{lem}

\begin{proof}
Set \(a=L/n\), \(\omega=\bzeta-\mathrm{Leb}_{I_n}\), and
\(\lambda=D/n\).  The image under
\(x\mapsto x/a\) of \(\betaeta-\mathrm{Leb}_{I_L}\) is
\[
 \bzeta-a\mathrm{Leb}_{I_n}
 =\omega+\lambda\mathrm{Leb}_{I_n},
\]
and, for \(u\neq v\),
\[
 \mathsf{g}(a(u-v))=\mathsf{g}(u-v)-\log(a).
\]
Since
\(\omega(I_n)=0\),
\[
 \begin{aligned}
 &\iint_{I_n^2\setminus\Delta}\mathsf{g}(u-v)
 \dd(\omega+\lambda\mathrm{Leb}_{I_n})(u)
 \dd(\omega+\lambda\mathrm{Leb}_{I_n})(v)\\
 &\qquad=
 W_{I_n}^{\intE}(\bzeta)+2\lambda \mathcal A_n(\bzeta)
 +\lambda^2\iint_{I_n^2}\mathsf{g}(u-v)\dd u\dd v.
 \end{aligned}
\]
Scaling \(u=ns\) and \(v=nt\) gives
\[
 \iint_{I_n^2}\mathsf{g}(u-v)\dd u\dd v
 =n^2\left(\frac32-\log(n)\right).
\]
The signed measure
\(\omega+\lambda\mathrm{Leb}_{I_n}
=\bzeta-a\mathrm{Leb}_{I_n}\) has total mass
\(D\).  Its product measure assigns mass \(n\) to the removed atomic
diagonal, while every term containing Lebesgue measure assigns zero mass to
\(\Delta\).  Hence
\[
 \iint_{I_n^2\setminus\Delta}
 1\,\dd(\omega+\lambda\mathrm{Leb}_{I_n})(u)
 \dd(\omega+\lambda\mathrm{Leb}_{I_n})(v)=D^2-n.
\]
Consequently,
\[
 W_{I_L}^{\intE}(\betaeta)
 =
 W_{I_n}^{\intE}(\bzeta)+\frac{2D}{n}\mathcal A_n(\bzeta)
 +D^2\left(\frac32-\log(n)\right)
 -(D^2-n)\log\left(\frac{L}{n}\right).
\]
Adding \(D^2\log(L)\) and collecting the logarithmic terms gives
\eqref{EqAffineIdentity}.

\end{proof}

\subsection{The averaged dilation statistic}

Throughout this subsection, let \(\bP\) satisfy the hypotheses of
Theorem~\ref{ThmMain}.  For \(s>0\), recall
\[
 \sN_s=\bC(I_s),\qquad
 \sD_s=\sN_s-s,\qquad
 v(s)=\E_{\bP}[\sD_s^2],
\]
and set \(I_0=\varnothing\) and \(\sN_0=\sD_0=v(0)=0\).  For \(L>0\), define
\[
 \bZ_L=(T_{\sN_L/L})_*\bC_{I_L}
 \quad\text{on }\{\sN_L\geq1\},
 \qquad
 \bZ_L=\varnothing\quad\text{on }\{\sN_L=0\}.
\]
Thus \(\bZ_L\) has \(\sN_L\) particles in an interval of length
\(\sN_L\), and is therefore exactly neutral.  It is the random-sector
candidate configuration that will be compared with an exactly neutral
circular ensemble. Proposition~\ref{PropInputs}(b) ensures the required finite second moments.
By Lemma~\ref{LemSimplicity}, \(\bC\) and all its restrictions are simple
almost surely.

The random map \(\bZ_L\) is Borel on the countable union of the
particle-number fibres.  More generally, if, for every \(n\in\mathbb Z_+\),
\(F_n\) is a Borel extended-real function on \(\Conf_n(I_n)\), then
\[
 \bgamma\longmapsto
 \begin{cases}
 \displaystyle
 F_{\bgamma(I_L)}
 \left((T_{\bgamma(I_L)/L})_*\bgamma_{I_L}\right),
 &\bgamma(I_L)\geq1,\\[1ex]
 F_0(\varnothing),&\bgamma(I_L)=0.
 \end{cases}
\]
This map is Borel by countable pasting over the sets
\(\{\bgamma(I_L)=n\}\).  This observation applies to all variable-sector
statistics and errors used below.

For \(L>0\) and \(0\leq s\leq L\), put
\[
 \mathsf{Br}_L(s)\defeq\sD_s-\frac{s}{L}\sD_L,
 \qquad
 q_L(s)\defeq\E_{\bP}[\mathsf{Br}_L(s)^2].
\]
The function \(\mathsf{Br}_L\) is the counting discrepancy after
subtracting the linear interpolation of its endpoint value, and hence
vanishes at both endpoints.  Its second moment \(q_L\) is the common
quantity controlling both the affine dilation statistic and the later
periodic error under the candidate law.
\begin{lem}\label{LemAveragedDilationStatistic}
For \(0\leq s\leq L\),
\begin{equation}\label{EqBridgeVarianceLinearBound}
 0\leq q_L(s)
 \leq2C_{\bP}\frac{s(L-s)}L.
\end{equation}
For every fixed \(0<x<1\),
\begin{equation}\label{EqRescaledBridgeLimit}
 \frac{q_L(Lx)}L\longrightarrow0
 \qquad\text{as }L\to\infty.
\end{equation}
Moreover, there is \(C<\infty\) such that, with the quotient defined to be
zero on \(\{\sN_L=0\}\),
\begin{equation}\label{EqAveragedDilationBound}
 \E_{\bP}\left[
  \left|\frac{\sD_L}{\sN_L}
  \mathcal A_{\sN_L}(\bZ_L)\right|
 \right]
 \leq C\sqrt{Lv(L)}=o(L).
\end{equation}
\end{lem}

\begin{proof}
Stationarity and intensity one imply
\(\bP\{\bC(\{t\})=0\}=1\) for every deterministic \(t\in\R\).  We work on
the resulting full-probability event for \(t=0,L\).  Write \(n=\sN_L\).
For \(0<s<L\), put \(\ell_L(s)=\log((L-s)/s)\), and set
\(\ell_L(0)=\ell_L(L)=0\); the endpoint values are immaterial below.

For fixed \(0\leq s\leq L\), set
\[
 \widetilde{\sD}_{s,L}
 \defeq\bC((s,L])-(L-s).
\]
Then
\[
 \sD_L=\sD_s+\widetilde{\sD}_{s,L},
 \qquad
 \mathsf{Br}_L(s)
 =\left(1-\frac{s}{L}\right)\sD_s
  -\frac{s}{L}\widetilde{\sD}_{s,L}.
\]
Stationarity gives
\(\E_{\bP}[\widetilde{\sD}_{s,L}^{\,2}]=v(L-s)\).  Therefore
\[
 q_L(s)
 \leq2\left(1-\frac{s}{L}\right)^2v(s)
 +2\left(\frac{s}{L}\right)^2v(L-s),
\]
and Proposition~\ref{PropInputs}\textup{(b)} proves
\eqref{EqBridgeVarianceLinearBound}.  If \(0<x<1\), the same inequality
gives
\[
 \frac{q_L(Lx)}L
 \leq2x(1-x)\left[
  (1-x)\frac{v(Lx)}{Lx}
  +x\frac{v(L(1-x))}{L(1-x)}
 \right].
\]
Proposition~\ref{PropInputs}\textup{(c)} proves
\eqref{EqRescaledBridgeLimit}.

If \(n\geq1\), integration by parts in \eqref{EqUm} gives
\[
 \mathcal A_n(\bZ_L)
 =-\int_0^n[\bZ_L([0,u])-u]
       \log\left(\frac{n-u}{u}\right)\dd u.
\]
Indeed, the cumulative function
\(u\mapsto(\bZ_L-\mathrm{Leb}_{I_n})([0,u])\) is bounded and has zero limits
at \(0\) and \(n\), while \(U_n\in W^{1,1}(0,n)\) and
\(U_n'=\log((n-u)/u)\).  Applying Stieltjes integration by parts first on
\([\delta,n-\delta]\) and then letting \(\delta\downarrow0\) is legitimate
because \(U_n\) has finite endpoint limits, \(U_n'\in\mathrm L^1(0,n)\),
and \(\bZ_L\) has no endpoint atoms.  Under \(u=(n/L)s\),
\[
 \bZ_L([0,u])-u=\mathsf{Br}_L(s).
\]
Consequently,
\begin{equation}\label{EqARepresentation}
 \frac{\sD_L}{\sN_L}\mathcal A_{\sN_L}(\bZ_L)
 =
 -\frac{\sD_L}{L}\int_0^L
 \mathsf{Br}_L(s)\ell_L(s)\dd s.
\end{equation}
For \(n=0\), we assign the left-hand side the value \(0\); the right-hand
side is already \(0\), since \(\mathsf{Br}_L\equiv0\).

By \eqref{EqARepresentation}, Tonelli's theorem, Cauchy--Schwarz, and
\eqref{EqBridgeVarianceLinearBound},
\[
 \begin{aligned}
 \E_{\bP}\left[
  \left|\frac{\sD_L}{\sN_L}
  \mathcal A_{\sN_L}(\bZ_L)\right|\right]
 &\leq\frac{\sqrt{v(L)}}L
 \int_0^L\sqrt{q_L(s)}\,|\ell_L(s)|\dd s\\
 &\leq C_{\bP}\sqrt{Lv(L)}
 \int_0^1\sqrt{u(1-u)}
 \left|\log\left(\frac{1-u}{u}\right)\right|\dd u.
 \end{aligned}
\]
The last integral is finite.  Proposition~\ref{PropInputs}\textup{(c)}
therefore proves \eqref{EqAveragedDilationBound}.
\end{proof}

\section{The conditional Gibbs kernel and boundary potential}
\label{SectionCanonicalFibre}

This section identifies the conditional distribution inside a bounded
interval as the canonical Gibbs kernel determined by its exterior
configuration and rigid particle number.  We construct an exterior-Borel
version of this kernel and isolate the gauge-invariant boundary potential
which appears when it is compared directly with the rescaled circular Gibbs
law in Section~\ref{SectionCircularEntropy}.

\subsection{Exterior measurability of the Gibbs kernel}

The analytic content of the two lemmas below is contained in the DHLM
move-function convergence recorded in
Proposition~\ref{PropInputs}\textup{(d)} and in the observation that the
normalised kernel depends only on the exterior configuration and the rigid
particle number.  We include concise proofs of the exterior-Borel and
disintegration points because these versions are used later in the direct
relative-entropy comparison.

For the remainder of the proof, let \(\bP\) satisfy the hypotheses of
Theorem~\ref{ThmMain}.  Fix \(L>0\), let \(\balpha_L\) be the law of
\(\bC_{I_L^c}\), and fix a Borel regular conditional kernel
\(\bxi\mapsto K_{L,\bxi}\) for \(\bC_{I_L}\) given \(\bC_{I_L^c}\).

Proposition~\ref{PropInputs}(a) gives a Borel function
\[
 n_L:\Conf(I_L^c)\longrightarrow\mathbb Z_+
 \quad\text{such that}\quad
 \bC(I_L)=n_L(\bC_{I_L^c})\quad \bP\text{-almost surely}.
\]
For \(j\in\mathbb Z_+\), choose a
deterministic simple \(\br_j\in\Conf_j(I_L)\), with
\(\br_0=\varnothing\), and set
\[
 p_{I_L}\defeq\min\{p\in\N:I_L\subset\Lambda_p\}.
\]
For \(p\in\N\) with \(p\geq p_{I_L}\), define
\[
 A^{(p)}(\bxi,\betaeta)\defeq
 \iint_{I_L\times(\Lambda_p\setminus I_L)}
 \mathsf{g}(x-y)\dd(\betaeta-\br_{n_L(\bxi)})(x)\dd\bxi(y)
\]
when \(\betaeta(I_L)=n_L(\bxi)\), and set \(A^{(p)}=0\) otherwise.
Put \(A^{(p)}=0\) for \(p<p_{I_L}\), and define
\begin{equation}\label{EqExteriorPotentialDefinition}
 A_{\bxi}(\betaeta)\defeq
 \limsup_{\substack{p\to\infty\\p\in\N}}
 A^{(p)}(\bxi,\betaeta).
\end{equation}

Rigidity fixes the interior particle number from the exterior.  The next lemma
constructs a Borel exterior potential \(A_{\bxi}\); its differences form the
cocycle, meaning that \(A_{\bxi}(\betaeta)-A_{\bxi}(\bX)\) is exactly the
exterior move cost.

\begin{lem}\label{LemExteriorPotential}
The maps \(A^{(p)}\) and the extended-valued map in
\eqref{EqExteriorPotentialDefinition} are Borel on
\(\Conf(I_L^c)\times\Conf(I_L)\).  There is a
\(\balpha_L\)-full Borel set
\(\mathscr X_{I_L}^{\mathrm{pot}}\subset\Conf(I_L^c)\) such that, for every
\(\bxi\in\mathscr X_{I_L}^{\mathrm{pot}}\),
\[
 A^{(p)}(\bxi,\mathord\cdot)\longrightarrow
 A_{\bxi}
 \quad\text{uniformly on }\Conf_{n_L(\bxi)}(I_L),
\]
and \(A_{\bxi}\) is finite and continuous on this fibre.  Moreover, for
every \(\bX,\betaeta\in\Conf_{n_L(\bxi)}(I_L)\), the configuration
\(\bX+\bxi\) is \((\beta,I_L)\)-admissible and
\begin{equation}\label{EqMoveCocycle}
 \mathsf{M}_{I_L,\R}(\betaeta,\bX+\bxi)
 =A_{\bxi}(\betaeta)-A_{\bxi}(\bX).
\end{equation}
\end{lem}

\begin{proof}
For fixed \(p\) and fixed particle number \(j\), the finite-cutoff expression
\(A^{(p)}(\bxi,\betaeta)\) is a parameterised integral of a Borel kernel
against finite restrictions of \(\bxi\) and \(\betaeta-\br_j\); hence it is
Borel.  For fixed \(\bxi\), the exterior configuration in the bounded cutoff
region is finite and disjoint from the compact interval \(I_L\).  Therefore
\(A^{(p)}(\bxi,\mathord\cdot)\) is continuous on
\(\Conf_j(I_L)\).  Since \(I_L\) is compact, every fibre
\(\Conf_j(I_L)\) is compact.  The measurable maximum theorem, applied
separately on each fibre, shows that
\[
 \bxi\longmapsto
 \sup_{\betaeta\in\Conf_{n_L(\bxi)}(I_L)}
 \left|A^{(p)}(\bxi,\betaeta)-A^{(q)}(\bxi,\betaeta)\right|
\]
is Borel.  Consequently, the set of exteriors for which
\(A^{(p)}(\bxi,\mathord\cdot)\) is uniformly Cauchy on the prescribed fibre
is Borel.

Disintegrating the full-probability assertion in
Proposition~\ref{PropInputs}\textup{(d)}, for
\(\balpha_L\)-almost every \(\bxi\) there is
\(\bX\in\Conf_{n_L(\bxi)}(I_L)\) for which the truncated move functions
converge uniformly in the replacement configuration.  For every
\(p\geq p_{I_L}\),
\[
 \mathsf M_{I_L,\Lambda_p}(\betaeta,\bX+\bxi)
 =
 A^{(p)}(\bxi,\betaeta)-A^{(p)}(\bxi,\bX).
\]
Taking \(\betaeta=\br_{n_L(\bxi)}\) first shows that
\(A^{(p)}(\bxi,\bX)\) converges.  The same identity then shows that
\(A^{(p)}(\bxi,\mathord\cdot)\) converges uniformly on the whole fibre.
Let \(\mathscr X_{I_L}^{\mathrm{pot}}\) be this Borel full set.  Its uniform
limit is the limsup in \eqref{EqExteriorPotentialDefinition}; it is finite
and continuous.

For arbitrary \(\bX,\betaeta\) in the prescribed fibre, put
\[
 \widetilde{\mathsf M}_{\bxi,\bX}(\betaeta)
 \defeq A_{\bxi}(\betaeta)-A_{\bxi}(\bX).
\]
The finite-cutoff identity and the preceding uniform convergence show that
\(\mathsf M_{I_L,\Lambda_p}(\mathord\cdot,\bX+\bxi)\) converges uniformly to
\(\widetilde{\mathsf M}_{\bxi,\bX}\).  Moreover, \(A_{\bxi}\) is bounded on
the compact fibre, while \(\mathsf H_{I_L}\) is bounded below and finite on a
binomial-full set.  The corresponding partition function is therefore
positive and finite.  Thus \(\bX+\bxi\) is
\((\beta,I_L)\)-admissible, and uniqueness of the uniform limit gives
\[
 \mathsf M_{I_L,\R}(\betaeta,\bX+\bxi)
 =
 A_{\bxi}(\betaeta)-A_{\bxi}(\bX).
\]
\end{proof}

For \(\bxi\in\mathscr X_{I_L}^{\mathrm{pot}}\), changing the deterministic
reference configuration from \(\br_j\) to another \(\br'_j\) gives, on the
prescribed fibre,
\[
 A'_{\bxi}(\betaeta)
 =A_{\bxi}(\betaeta)-A_{\bxi}(\br'_{n_L(\bxi)}).
\]
Indeed, uniform convergence in Lemma~\ref{LemExteriorPotential} makes the
corresponding finite-cutoff constants converge.  Thus the normalised Gibbs
kernel, the cocycle, and all differences of the functions
\(\widehat A_{\bxi}\) below
are unchanged.

The DLR equation and the move cocycle now identify the regular conditional law,
on a Borel full subset of \(\mathscr X_{I_L}^{\mathrm{pot}}\), with the canonical
Gibbs density.  This formula is used in the direct comparison below.

\begin{lem}
\label{LemConditionalGibbs}
For \(\bxi\in\mathscr X_{I_L}^{\mathrm{pot}}\), put
\[
 Z_{\bxi}\defeq
 \int_{\Conf_{n_L(\bxi)}(I_L)}
 \exp[-\beta(\mathsf{H}_{I_L}(\betaeta)+A_{\bxi}(\betaeta))]
 \Bin_{n_L(\bxi),I_L}(\dd\betaeta).
\]
Then \(0<Z_{\bxi}<\infty\), and there is a
\(\balpha_L\)-full Borel set
\(\mathscr X_{I_L}^\star\subset\mathscr X_{I_L}^{\mathrm{pot}}\) such that
\begin{equation}\label{EqConditionalGibbs}
 K_{L,\bxi}(\dd\betaeta)=
 \frac{1}{Z_{\bxi}}
 \exp[-\beta(\mathsf{H}_{I_L}(\betaeta)+A_{\bxi}(\betaeta))]
 \Bin_{n_L(\bxi),I_L}(\dd\betaeta)
\end{equation}
for every \(\bxi\in\mathscr X_{I_L}^\star\).
\end{lem}

\begin{proof}
For \(\bxi\in\mathscr X_{I_L}^{\mathrm{pot}}\), compactness of
\(\Conf_{n_L(\bxi)}(I_L)\) and continuity of \(A_{\bxi}\) make
\(A_{\bxi}\) bounded on the prescribed fibre.  Since
\(\mathsf H_{I_L}\) is bounded below there and is finite on a
\(\Bin_{n_L(\bxi),I_L}\)-full set, one has \(0<Z_{\bxi}<\infty\).

Joint Borelness from Lemma~\ref{LemExteriorPotential}, parameterised
integration, and countable pasting over \(n_L(\bxi)\) show that
\(\bxi\mapsto Z_{\bxi}\) is Borel on
\(\mathscr X_{I_L}^{\mathrm{pot}}\).  Set \(Z_{\bxi}=1\) off this set.
Define
\[
 \widehat K_{L,\bxi}(\dd\betaeta)
 =
 \frac{1}{Z_{\bxi}}
 \exp[-\beta(\mathsf H_{I_L}(\betaeta)+A_{\bxi}(\betaeta))]
 \Bin_{n_L(\bxi),I_L}(\dd\betaeta)
\]
on \(\mathscr X_{I_L}^{\mathrm{pot}}\), and extend it outside that set by
\(\delta_{\br_{n_L(\bxi)}}\).

For \(\bX\in\Conf_{n_L(\bxi)}(I_L)\), the cocycle gives
\[
 \mathsf H_{I_L}(\betaeta)
 +\mathsf M_{I_L,\R}(\betaeta,\bX+\bxi)
 =
 \mathsf H_{I_L}(\betaeta)
 +A_{\bxi}(\betaeta)-A_{\bxi}(\bX).
\]
The final term cancels on normalisation.  Thus the canonical DLR kernel at
\(\bX+\bxi\) is \(\widehat K_{L,\bxi}\).

Applying the DLR equation to product tests
\[
 F(\betaeta+\bxi)=\varphi(\betaeta)h(\bxi)
\]
shows that \(\widehat K_{L,\bxi}\) is a regular conditional law of
\(\bC_{I_L}\) given \(\bC_{I_L^c}\).  Uniqueness of regular conditional
probabilities gives equality with \(K_{L,\bxi}\) almost everywhere.
Intersecting the equality sets for a countable convergence-determining class
of \(\varphi\)'s produces a Borel \(\balpha_L\)-full set
\(\mathscr X_{I_L}^{\star}\) on which
\eqref{EqConditionalGibbs} holds.
\end{proof}

\subsection{The compensated exterior potential}

The purpose of this subsection is to separate the conditional Gibbs energy
into an intrinsic bulk energy and a gauge-invariant boundary correction.
Recall \(U_L\) from \eqref{EqUm}, and define
\[
 \widehat A_{\bxi}(\betaeta)
 \defeq
 A_{\bxi}(\betaeta)+\int_{I_L}U_L(x)\dd\betaeta(x).
\]
The added term compensates for the interaction with the uniform
background in \(I_L\).  With this choice, the conditional Gibbs energy
splits into the intrinsic bulk term
\(\frac12W_{I_L}^{\intE}\) and the remaining exterior term
\(\widehat A_{\bxi}\), up to a fibrewise constant. The function \(U_L\) is bounded and continuous on \(I_L\).  Hence
\((\bxi,\betaeta)\mapsto\widehat A_{\bxi}(\betaeta)\) is Borel, and, for
every \(\bxi\in\mathscr X_{I_L}^{\mathrm{pot}}\),
\(\widehat A_{\bxi}\) is bounded on
\(\Conf_{n_L(\bxi)}(I_L)\).

Expanding the definition of the intrinsic energy and using
\eqref{EqUm}, we obtain, for every
\(\bxi\in\mathscr X_{I_L}^{\mathrm{pot}}\) and every simple
\(\betaeta\in\Conf_{n_L(\bxi)}(I_L)\),
\begin{equation}\label{EqMasterFibreIdentity}
 \mathsf H_{I_L}(\betaeta)+A_{\bxi}(\betaeta)
 =
 \frac12W_{I_L}^{\intE}(\betaeta)
 +\widehat A_{\bxi}(\betaeta)
 -\frac12\iint_{I_L^2}\mathsf{g}(x-y)\dd x\dd y.
\end{equation}
The last term is constant on the particle-number fibre and therefore
disappears upon normalization.  Thus
\eqref{EqMasterFibreIdentity} rewrites the conditional law in
Lemma~\ref{LemConditionalGibbs} as a Gibbs law whose bulk potential is
\(\frac{\beta}{2}W_{I_L}^{\intE}\) and whose remaining exterior
potential is \(\beta\widehat A_{\bxi}\).

We record the integrability needed for the later entropy comparison.
For \(\bxi\in\mathscr X_{I_L}^{\star}\), the function
\(A_{\bxi}\) is bounded on the prescribed fibre, while
\(\mathsf H_{I_L}\) is bounded below and is finite outside the
binomial-null collision set.  Equation~\eqref{EqConditionalGibbs} and
the elementary bound
\(\sup_{t\geq0}t\mathrm e^{-\beta t}<\infty\), applied after a
fibrewise shift, therefore give
\[
 K_{L,\bxi}\bigl(
 |\mathsf H_{I_L}+A_{\bxi}|
 \bigr)<\infty.
\]
Together with \eqref{EqMasterFibreIdentity} and the boundedness of
\(\widehat A_{\bxi}\), this implies
\[
 K_{L,\bxi}\bigl(|W_{I_L}^{\intE}|\bigr)<\infty.
\]

Finally, on the Borel set of triples
\[
 \left\{(\bxi,\bX,\betaeta):
 \bxi\in\mathscr X_{I_L}^{\star},\
 \bX,\betaeta\in\Conf_{n_L(\bxi)}(I_L)\right\},
\]
define
\begin{equation}\label{EqBoundaryVariable}
 \begin{aligned}
 \sB_{I_L}(\betaeta,\bX,\bxi)
 &\defeq
 \widehat A_{\bxi}(\betaeta)-\widehat A_{\bxi}(\bX)\\
 &=
 \mathsf M_{I_L,\R}(\betaeta,\bX+\bxi)
 +\int_{I_L}U_L(x)\dd(\betaeta-\bX)(x).
 \end{aligned}
\end{equation}
The second equality follows from the cocycle identity
\eqref{EqMoveCocycle}.  Extend \(\sB_{I_L}\) by zero outside the
preceding set.  It is then a jointly Borel, real-valued function.
Changing the reference configuration used to define \(A_{\bxi}\)
adds the same constant to both terms in the first line of
\eqref{EqBoundaryVariable}; hence \(\sB_{I_L}\) is gauge invariant.

Thus \(\sB_{I_L}(\betaeta,\bX,\bxi)\) records the change in compensated
exterior energy when \(\bX\) is replaced by \(\betaeta\) on the same
particle-number fibre.  The first line of
\eqref{EqBoundaryVariable} is used in the direct Gibbs comparison,
whereas the second is the form estimated in
Section~\ref{SectionBoundary}.
\section{The circular ensemble, periodic comparison, and Palm limit}
\label{SectionCircularRecovery}

This section records the two inputs supplied by the circular ensemble.  Its
periodic Gibbs law is the reference measure in the direct relative-entropy
comparison, and its stationary periodic lift converges, together with its
reduced Palm law, to \(\Sine_\beta\).  We also compare its periodic energy
with the intrinsic interval energy and obtain the error estimate needed
below.

\subsection{The circular reference and its local limit}

Recall that \(b=\beta/2\).  For \(n\in\N\), put
\(\mathbb T_n=\R/n\mathbb Z\), and
let \(\lambda_n\) be Haar measure on \(\mathbb T_n\), normalised by
\(\lambda_n(\mathbb T_n)=n\).  Let \(\Conf(\mathbb T_n)\) be the space
of integer-valued Radon measures on \(\mathbb T_n\), endowed with its
vague Borel sigma-field, and put
\[
 \Conf_j(\mathbb T_n)\defeq
 \{\betaeta\in\Conf(\mathbb T_n):\betaeta(\mathbb T_n)=j\},
 \qquad
 \Delta_n\defeq\{(u,u):u\in\mathbb T_n\}.
\]
Whenever a configuration on \(I_n\) occurs in a circular expression below,
it is understood to be pushed forward under the quotient map
\(I_n\to\mathbb T_n\).  Define
\[
 \mathsf{g}_n(t)\defeq-\log\left(\left|2\sin\frac{\pi t}{n}\right|\right),
 \qquad
 \mathsf{h}_n(t)\defeq-\log\left(\left|\frac{n}{\pi}
 \sin\frac{\pi t}{n}\right|\right)
 \quad(t\notin n\mathbb Z),
\]
and set \(\mathsf{g}_n(t)=\mathsf{h}_n(t)=+\infty\) for \(t\in n\mathbb Z\).
Both kernels are \(n\)-periodic and hence descend to extended-real
functions on \(\mathbb T_n\).  If
\(\bzeta=\sum_{i=1}^n\delta_{u_i}\) is simple on \(\mathbb T_n\), put
\[
 \mathsf{H}_n^{\per}(\bzeta)\defeq\sum_{1\leq i<j\leq n}\mathsf{g}_n(u_i-u_j),
\]
and set \(\mathsf{H}_n^{\per}(\bzeta)=+\infty\) otherwise.  Put
\[
 Z_{n,\beta}^{\mathrm{circ}}\defeq
 \int_{\Conf_n(I_n)}
 \exp[-\beta \mathsf{H}_n^{\per}(\bzeta)]
 \Bin_{n,I_n}(\dd\bzeta).
\]
\begin{defn}\label{DefCircularEnsemble}
The circular \(\beta\)-ensemble is
\[
 \bQ_{n,\beta}(\dd\bzeta)\defeq
 \frac{\exp[-\beta \mathsf{H}_n^{\per}(\bzeta)]}{Z_{n,\beta}^{\mathrm{circ}}}
 \Bin_{n,I_n}(\dd\bzeta).
\]
Set \(\bQ_{0,\beta}=\delta_\varnothing\).
\end{defn}
Since \(\mathsf{g}_n\geq-\log2\), the circular Hamiltonian is bounded below.  It is
finite outside the binomial-null collision set, and hence (in this paper $\sim$ means mutually absolutely continuous and not asymptotic equivalence)
\[
 0<Z_{n,\beta}^{\mathrm{circ}}<\infty,
 \qquad
 \bQ_{n,\beta}\sim\Bin_{n,I_n}.
\]

We use the same symbol \(\bQ_{n,\beta}\) for the interval law in
Definition~\ref{DefCircularEnsemble} and for its quotient push-forward to
\(\Conf_n(\mathbb T_n)\).  Interval dilations, interval energies, and
fixed-cut coordinates refer to the former realization, whereas rotations,
Haar measure, \(W_n^{\per}\), and \(\operatorname{Per}_n\) refer to the
latter.  Here and below, the seam is the distinguished point of \(\mathbb T_n\)
obtained by identifying the two endpoints \(0\) and \(n\) of the interval
representative \(I_n\). Since \(\bQ_{n,\beta}\sim\Bin_{n,I_n}\), a sampled configuration has no
particle at the seam almost surely, and this convention causes no ambiguity.

For a simple \(\bzeta\in\Conf_n(\mathbb T_n)\), define
\[
 W_n^{\per}(\bzeta)\defeq
 \iint_{\mathbb T_n^2\setminus\Delta_n}
 \mathsf{h}_n(u-v)\dd(\bzeta-\lambda_n)(u)
 \dd(\bzeta-\lambda_n)(v),
\]
where only the point--point diagonal is removed, and set
\(W_n^{\per}(\bzeta)=+\infty\) when \(\bzeta\) is not simple on
\(\mathbb T_n\).  Since
\[
 \mathsf{h}_n=\mathsf{g}_n-c_n^{\mathrm{per}},\qquad
 c_n^{\mathrm{per}}\defeq\log\left(\frac{n}{2\pi}\right),\qquad
 \int_0^n\mathsf{g}_n(t)\dd t=0,
\]
the point--point, point--background, and background--background terms
equal, respectively,
\[
 2\mathsf{H}_n^{\per}-n(n-1)c_n^{\mathrm{per}},\qquad
 2n^2c_n^{\mathrm{per}},\qquad
 -n^2c_n^{\mathrm{per}}.
\]
Thus, putting everything together
\begin{equation}\label{EqPeriodicEnergyIdentity}
 W_n^{\per}=2\mathsf{H}_n^{\per}+n\log\left(\frac{n}{2\pi}\right).
\end{equation}
The elementary Gibbs-integrability argument also gives
\(\mathsf H_n^{\per}\in\mathrm L^1(\bQ_{n,\beta})\).  Consequently,
\eqref{EqPeriodicEnergyIdentity} yields
\[
 W_n^{\per}\in\mathrm L^1(\bQ_{n,\beta}).
\]
Set \(W_0^{\per}(\varnothing)=0\).

For \(\bzeta=\sum_{i=1}^n\delta_{u_i}\in\Conf_n(\mathbb T_n)\), choose
representatives \(u_i\in[0,n)\) and set
\[
 \operatorname{Per}_n(\bzeta)
 \defeq\sum_{i=1}^n\sum_{k\in\mathbb Z}\delta_{u_i+kn},
 \qquad
 \bP_n\defeq(\operatorname{Per}_n)_*\bQ_{n,\beta}.
\]
This does not depend on the representatives.  Rotation invariance of
\(\bQ_{n,\beta}\) makes \(\bP_n\) stationary, and its intensity is
\(n/n=1\).

We record the LS--DHLM finite-energy and local-convergence input together
with the Palm convergence needed below.
\begin{lem}\label{LemSineFiniteEnergy}
The stationary periodic lifts \(\bP_n\) defined above satisfy
\begin{equation}\label{EqCircularLocalAndPalmLimit}
 \bP_n\overset{\mathrm{w}}{\longrightarrow}\Sine_\beta,
 \qquad
 \bP_n^{!0}\overset{\mathrm{w}}{\longrightarrow}
 (\Sine_\beta)^{!0}.
\end{equation}
Moreover,
\[
 \cW^{\mathrm{div}}(\Sine_\beta)<\infty.
\]
\end{lem}

\begin{proof}
The finite-energy assertion follows from the energy bounds in
\cite[Corollary~1.2 and Sections~2.7.3--2.7.4]{LS}.  For the first convergence in
\eqref{EqCircularLocalAndPalmLimit}, after translating the fundamental period
to the centered cell, the present
circular ensemble $\bQ_{n,\beta}$ is exactly the finite periodic log-gas of
\cite[Definition~2.6]{DHLM}.  On every fixed compact set its periodic lift
agrees, for all sufficiently large \(n\), with the representative in that
cell.  Hence \cite[Proposition~2.9]{DHLM} gives the first convergence in
\eqref{EqCircularLocalAndPalmLimit}.  Moreover,
\cite[Lemma~2.24]{DHLM} gives, for every fixed bounded interval \(J\) and
all sufficiently large \(n\),
\[
 \E_{\bP_n}[\bC(J)^2]\leq C_{\beta,J}.
\]
Equivalently, this follows by adding the square of the rotation-invariant
mean \(\E_{\bP_n}[\bC(J)]=|J|\) to the variance bound in that lemma; the
finitely many remaining \(n\) are harmless.  Thus the second moments are
uniform in \(n\).

It remains to justify the Palm convergence, which does not follow from
unrooted local convergence explained above alone.  Let \(F:\Conf(\R)\to\R\) be bounded,
local, and continuous, and choose a non-negative
\(\varrho\in C_c(\R)\) with \(\int_\R\varrho(x)\dd x=1\).  Consider the
local Campbell observable
\[
 \mathcal K_{F,\varrho}(\bgamma)
 =\sum_{x\in\bgamma}\varrho(x)
 F\bigl(\theta_x(\bgamma-\delta_x)\bigr).
\]
For \(M\geq1\), put
\(\tau_M(z)=(-M)\vee(z\wedge M)\).  A local labelling of atoms shows that
\(\mathcal K_{F,\varrho}\) is continuous at every simple configuration;
the continuity and compact support of \(\varrho\) make atoms entering or
leaving its support contribute vanishingly.  Hence
\(\tau_M\circ\mathcal K_{F,\varrho}\) is a bounded local observable whose
discontinuity set is contained in the non-simple configurations.  The limit
is simple by Lemmas~\ref{LemDHLMKernelBridge} and \ref{LemSimplicity}, so local
convergence applies to each clipped observable.  Moreover,
\[
 |\mathcal K_{F,\varrho}(\bgamma)|
 \leq\|F\|_\infty\|\varrho\|_\infty
       \bgamma(\operatorname{supp}(\varrho)).
\]
The uniform local count second moments therefore give uniform integrability.
Choose a bounded open interval \(J_0\) containing
\(\operatorname{supp}(\varrho)\).  By the portmanteau theorem and the
preceding uniform bound,
\[
 \E_{\Sine_\beta}[\bC(J_0)^2]
 \leq
 \liminf_{n\to\infty}\E_{\bP_n}[\bC(J_0)^2]
 <\infty.
\]
Thus the clipping errors vanish both uniformly for the approximating laws
and under the limiting law.
First letting \(n\to\infty\) for fixed \(M\), and then \(M\to\infty\), gives
\[
 \E_{\bP_n}[\mathcal K_{F,\varrho}]
 \longrightarrow
 \E_{\Sine_\beta}[\mathcal K_{F,\varrho}].
\]
The Campbell identity and intensity one identify the two sides as
\(\bP_n^{!0}(F)\) and \((\Sine_\beta)^{!0}(F)\), respectively.  To record
tightness explicitly, let \(J\subset\R\) be a bounded interval, put
\(A=\operatorname{supp}(\varrho)\), and choose a compact interval \(K\)
containing \(A\cup(A+J)\).  Tonelli's theorem and the Campbell identity give
\[
 \E_{\bP_n^{!0}}[\bC(J)]
 =\E_{\bP_n}\left[
   \sum_{x\in\bC}\varrho(x)
   \bigl(\theta_x(\bC-\delta_x)\bigr)(J)
  \right]
 \leq \|\varrho\|_\infty
       \E_{\bP_n}[\bC(K)^2].
\]
The last quantity is uniformly bounded by the preceding local-count estimate,
so \(\{\bP_n^{!0}\}\) is locally tight.  Every subsequential limit agrees
with \((\Sine_\beta)^{!0}\) on bounded local continuous functions; in
particular, the compactly supported Laplace functionals form a
convergence-determining subclass.  This proves the second assertion of
\eqref{EqCircularLocalAndPalmLimit}.
\end{proof}

\subsection{From the periodic to the intrinsic energy}

The periodic energy uses the torus Green kernel and Haar background, whereas
the intrinsic energy uses the real-line logarithmic kernel and Lebesgue
background on \(I_n\). The purpose of this subsection is to quantitatively compare the error between the two energies.

Define the Borel function \(\mathsf{r}_n:[-n,n]\to\R\) by
\[
 \mathsf{r}_n(t)\defeq
 \log\left(\left|\frac{\sin(\pi t/n)}{\pi t/n}\right|\right)
 \quad(0<|t|<n),\qquad
 \mathsf{r}_n(-n)=\mathsf{r}_n(0)=\mathsf{r}_n(n)\defeq0.
\]
For \(\bzeta\in\Conf_n(I_n)\), put
\[
 \mathsf{Err}_n^{\mathrm{per}}(\bzeta)\defeq
 \iint_{I_n^2}\mathsf{r}_n(u-v)
 \dd(\bzeta-\mathrm{Leb}_{I_n})(u)
 \dd(\bzeta-\mathrm{Leb}_{I_n})(v),
\]
where the integral is the finite expansion into its point--point,
point--background, and background--background terms.  Since
\(\mathsf{h}_n(t)=\mathsf{g}(t)-\mathsf{r}_n(t)\) for \(0<|t|<n\), every simple
\(\bzeta\in\Conf_n(I_n)\) with no point at either endpoint
satisfies
\begin{equation}\label{EqPeriodicIntrinsicDifference}
 W_{I_n}^{\intE}(\bzeta)
 =W_n^{\per}(\bzeta)+\mathsf{Err}_n^{\mathrm{per}}(\bzeta).
\end{equation}
This applies \(\bQ_{n,\beta}\)-almost surely because
\(\bQ_{n,\beta}\ll\Bin_{n,I_n}\).
On the empty sector set
\(\mathsf{Err}_0^{\mathrm{per}}(\varnothing)=0\). For \(0<|t|<n\), define
\[
 k_n(t)\defeq-\mathsf{r}_n''(t)
 =\frac{\pi^2}{n^2}\csc^2\left(\frac{\pi t}{n}\right)-\frac1{t^2},
 \qquad k_n(0)\defeq\frac{\pi^2}{3n^2}.
\]
Whenever \(k_n(u-v)\) is integrated over \([0,n]^2\), we use the
almost-everywhere representative obtained by assigning an arbitrary value,
say zero, at \(u-v=\pm n\).  These two corner values do not affect any
Lebesgue integral.

Thus \(\mathsf{Err}_n^{\mathrm{per}}\) is the exact discrepancy between
the intrinsic interval energy and the periodic circular energy.  It
occurs under both laws in the direct Gibbs comparison, so it must be
controlled separately under the circular reference and under the
neutralised candidate configuration.

\begin{lem}\label{LemPeriodicIntrinsic}
There is \(C_\beta<\infty\) such that, for every \(n\in\N\),
the first intensity measure of \(\bQ_{n,\beta}\) is \(\lambda_n\),
\(\mathsf{Err}_n^{\mathrm{per}}\in\mathrm{L}^1(\bQ_{n,\beta})\), and
\begin{equation}\label{EqPeriodicIntrinsicBound}
 \E_{\bQ_{n,\beta}}\left[|\mathsf{Err}_n^{\mathrm{per}}|\right]
 \leq C_\beta\log^2(2+n).
\end{equation}
Moreover, every simple \(\bzeta\in\Conf_n(I_n)\) with no point at
\(0\) or \(n\) satisfies, for
\(Y_{\bzeta}(u)=\bzeta([0,u])-u\),
\begin{equation}\label{EqRIntegrationByParts}
 \mathsf{Err}_n^{\mathrm{per}}(\bzeta)
 =\int_0^n\int_0^n
 k_n(u-v)Y_{\bzeta}(u)Y_{\bzeta}(v)\dd u\dd v.
\end{equation}
\end{lem}

\begin{proof}
We first work under \(\bQ_{n,\beta}\).  For the function
\(Y_{\bzeta}\) defined in the statement, the associated Stieltjes measure is
\(\dd Y_{\bzeta}=\dd(\bzeta-\mathrm{Leb}_{I_n})\).  The endpoints are
unoccupied \(\bQ_{n,\beta}\)-almost surely, while neutrality gives
\(\bzeta(I_n)=n\); hence \(Y_{\bzeta}(0)=Y_{\bzeta}(n)=0\).  For
\(0<|t|<n\), one has \(k_n(t)\geq0\).
There is a universal
\(C<\infty\) such that
\begin{equation}\label{EqKernelBound}
 0\leq k_n(t)\leq C\left(\frac1{n^2}
 +\frac1{(n-|t|)^2}\right),
 \qquad |t|<n.
\end{equation}
Indeed, \(|\sin x|\leq|x|\) gives non-negativity.  If
\(|t|\leq n/2\), the function
\(\csc^2x-x^{-2}\), extended continuously at zero, is bounded on
\([-\pi/2,\pi/2]\).  If \(n/2<|t|<n\), then
\[
 \sin\frac{\pi|t|}{n}
 =\sin\frac{\pi(n-|t|)}{n}
 \geq\frac{2(n-|t|)}{n},
\]
which proves \eqref{EqKernelBound}.

Both \(\Bin_{n,I_n}\) and \(\mathsf{H}_n^{\per}\) are invariant under
translations of \(\mathbb T_n\).  Hence the first intensity measure of
\(\bQ_{n,\beta}\) is translation-invariant.  Its mass is \(n\), so
uniqueness of Haar measure gives the first intensity measure
\(\lambda_n\).
If \(\bzeta=\sum_{i=1}^n\delta_{u_i}\) and
\(\phi_i=2\pi u_i/n\), then, relative to independent uniform angles,
the density of \(\bQ_{n,\beta}\) is proportional to
\[
 \prod_{1\leq i<j\leq n}
 \left|2\sin\frac{\phi_i-\phi_j}{2}\right|^\beta
 =\prod_{1\leq i<j\leq n}
 |\mathrm{e}^{\mathrm{i}\phi_i}-\mathrm{e}^{\mathrm{i}\phi_j}|^\beta.
\]
Thus this is exactly the circular \(\beta\)-ensemble of
\cite[Theorem~1]{NV}; after translation by \(-n/2\), it is also the
finite periodic log-gas of \cite[Definition~2.6]{DHLM}.  Moreover,
\(\bQ_{n,\beta}\ll\Bin_{n,I_n}\), so deterministic arc
endpoints are unoccupied almost surely.

Let
\[
 m(u)\defeq\min(u,n-u),\qquad
 \ell(s)\defeq\min(s,\log(2+s)).
\]
For an arc \(A\subset\mathbb T_n\), the circular count-minus-background
quantity is \(\bzeta(A)-\lambda_n(A)\).  Since
\(\bzeta(\mathbb T_n)-\lambda_n(\mathbb T_n)=0\) and endpoints carry no mass,
\[
 \bzeta(A)-\lambda_n(A)
 =-[\bzeta(A^c)-\lambda_n(A^c)].
\]
We may therefore replace \([0,u]\) by that arc or its complement,
whichever has length \(m(u)\), and rotate the shorter arc into the
centred fundamental domain.  After the angular rescaling
\(t\mapsto2\pi t/n\), with the fixed \(2\pi\) convention absorbed into
\(C_\beta\),
\cite[Theorem~1]{NV} gives
\[
 \Var_{\bQ_{n,\beta}}(\bzeta([0,u]))
 \leq C_\beta\log(2+m(u)),
\]
and \cite[Lemma~2.24]{DHLM} gives
\[
 \Var_{\bQ_{n,\beta}}(\bzeta([0,u]))
 \leq C_\beta m(u).
\]
Consequently,
\begin{equation}\label{EqCBEVariance}
 \E_{\bQ_{n,\beta}}\left[Y_{\bzeta}(u)^2\right]
 \leq C_\beta\ell(m(u)),\qquad 0\leq u\leq n.
\end{equation}

Fix \(\chi_{\mathrm{cap}}\in C^\infty([0,\infty);[0,1])\) which vanishes on
\([0,1]\) and equals one on \([2,\infty)\).  For
\(0<\delta<n/4\), define the even function
\[
 \mathsf{r}_{n,\delta}(t)\defeq
 \begin{cases}
 \mathsf{r}_n(n-2\delta),&n-|t|\leq\delta,\\
 \displaystyle
 \chi_{\mathrm{cap}}\left(\frac{n-|t|}{\delta}\right)\mathsf{r}_n(t)
 +\left[1-\chi_{\mathrm{cap}}\left(\frac{n-|t|}{\delta}\right)\right]
 \mathsf{r}_n(n-2\delta),&n-|t|>\delta.
 \end{cases}
\]
Then \(\mathsf{r}_{n,\delta}\in C^2([-n,n])\), it agrees with \(\mathsf{r}_n\) on
\([-n+2\delta,n-2\delta]\), and
\begin{equation}\label{EqCappedKernelBound}
 |\mathsf{r}_{n,\delta}''(t)|
 \leq C\left(\frac1{n^2}
 +\frac1{(n-|t|+\delta)^2}\right).
\end{equation}
To verify the bound on the transition region, write
\(d=n-|t|\in[\delta,2\delta]\) and
\(\vartheta(t)=\chi_{\mathrm{cap}}(d/\delta)\).  Directly from the formula for
\(\mathsf{r}_n\),
\[
 |\mathsf{r}_n(t)-\mathsf{r}_n(n-2\delta)|\leq C,\qquad
 |\mathsf{r}_n'(t)|\leq C\delta^{-1},\qquad
 |\mathsf{r}_n''(t)|\leq C\delta^{-2},
\]
while \(|\vartheta'|\leq C\delta^{-1}\) and
\(|\vartheta''|\leq C\delta^{-2}\).  Therefore
\[
 \left|[\vartheta(t)
 (\mathsf{r}_n(t)-\mathsf{r}_n(n-2\delta))]''\right|
 \leq C\delta^{-2},
\]
which is \eqref{EqCappedKernelBound} on the transition region.  In the
endpoint cap the function is constant, and on
\([-n+2\delta,n-2\delta]\) the assertion follows from
\eqref{EqKernelBound}.

Two Stieltjes integrations by parts, using
\(Y_{\bzeta}(0)=Y_{\bzeta}(n)=0\), give
\begin{equation}\label{EqCappedIntegrationByParts}
 \iint_{I_n^2}\mathsf{r}_{n,\delta}(u-v)
 \dd Y_{\bzeta}(u)\dd Y_{\bzeta}(v)
 =
 -\int_0^n\int_0^n
 \mathsf{r}_{n,\delta}''(u-v)Y_{\bzeta}(u)Y_{\bzeta}(v)\dd u\dd v.
\end{equation}
Almost surely, there is \(\delta_{\mathrm{sep}}>0\) such that
\[
 Y_{\bzeta}(t)=-t,\qquad Y_{\bzeta}(n-t)=t
 \quad(0<t<\delta_{\mathrm{sep}}).
\]
For \(\delta<\delta_{\mathrm{sep}}/2\), the equality
\(\mathsf{r}_{n,\delta}(u-v)\neq \mathsf{r}_n(u-v)\) can hold only when
\(n-|u-v|<2\delta\).  Hence the exceptional set is contained in the two corner
triangles
\[
 \{(u,v)\in I_n^2:u-v>n-2\delta\}
 \quad\text{and}\quad
 \{(u,v)\in I_n^2:v-u>n-2\delta\}.
\]
Both variables then lie in the atom-free endpoint neighbourhoods from
 the preceding display.  Hence the signed measure \(\dd Y_{\bzeta}\) equals
\(-\mathrm{Leb}_{I_n}\) on each such neighbourhood, and its product
there equals \(\mathrm{Leb}_{I_n}^{\otimes2}\).  The difference between the
left-hand side of \eqref{EqCappedIntegrationByParts} and \(\mathsf{Err}_n^{\mathrm{per}}\) is
therefore an ordinary Lebesgue integral over those two triangles.  At
the corner \(u=n-s\), \(v=t\),
\[
 |\mathsf{r}_{n,\delta}(n-s-t)-\mathsf{r}_n(n-s-t)|
 \leq C\left[1+\left|\log\left(\frac{s+t}{\delta}\right)\right|\right]
\]
on \(s,t\geq0\), \(s+t\leq2\delta\).  Thus each corner contributes at
most
\[
 C\int_{\substack{s,t\geq0\\s+t\leq2\delta}}
 \left[1+\left|\log\left(\frac{s+t}{\delta}\right)\right|\right]
 \dd s\dd t
 \leq C\delta^2.
\]
The left-hand side of \eqref{EqCappedIntegrationByParts} therefore
converges almost surely to \(\mathsf{Err}_n^{\mathrm{per}}\).

By Cauchy--Schwarz and \eqref{EqCBEVariance},
\begin{equation}\label{EqYProductBound}
 \E_{\bQ_{n,\beta}}\left[|Y_{\bzeta}(u)Y_{\bzeta}(v)|\right]
 \leq C_\beta\sqrt{\ell(m(u))\ell(m(v))}.
\end{equation}
Set
\[
 \mathcal I_n\defeq
 \int_0^n\int_0^n
 \left(\frac1{n^2}+\frac1{(n-|u-v|)^2}\right)
 \sqrt{\ell(m(u))\ell(m(v))}\dd u\dd v.
\]
The first term satisfies
\[
 \frac1{n^2}
 \left(\int_0^n\sqrt{\ell(m(u))}\dd u\right)^2
 \leq\log(2+n).
\]
For the second term, use symmetry, restrict to \(u\geq v\), and put
\(s=n-u\), \(t=v\).  Since \(s+t=n-|u-v|\),
\(m(u)\leq s\), and \(m(v)\leq t\), a second symmetry and an enlargement
of the integration domain give
\begin{align*}
 \mathcal I_n
 &\leq\log(2+n)
 +4\int_0^n\frac{\sqrt{\ell(s)}}{s^2}
 \left(\int_0^s\sqrt{\ell(t)}\dd t\right)\dd s\\
 &\leq\log(2+n)+4\int_0^n\frac{\ell(s)}{s}\dd s
 \leq C\log^2(2+n).
\end{align*}
The second inequality uses the monotonicity of \(\ell\), and the
integrand at zero is defined by continuity.

For almost every \(|t|<n\), one has
\[
 \mathsf{r}_{n,\delta}''(t)\longrightarrow \mathsf{r}_n''(t)
 \quad\text{as }\delta\downarrow0,
\]
and \eqref{EqKernelBound} and \eqref{EqCappedKernelBound} give the
common bound
\[
 |\mathsf{r}_{n,\delta}''(t)-\mathsf{r}_n''(t)|
 \leq C\left(\frac1{n^2}
 +\frac1{(n-|t|)^2}\right).
\]
Equations~\eqref{EqCappedKernelBound} and \eqref{EqYProductBound}, and
the finiteness of \(\mathcal I_n\), permit dominated convergence in
\[
 \begin{aligned}
 &\E_{\bQ_{n,\beta}}\left[\left|
 \int_0^n\int_0^n
 (\mathsf{r}_{n,\delta}''-\mathsf{r}_n'')(u-v)
 Y_{\bzeta}(u)Y_{\bzeta}(v)\dd u\dd v\right|\right]\\
 &\quad\leq
 \int_0^n\int_0^n
 |(\mathsf{r}_{n,\delta}''-\mathsf{r}_n'')(u-v)|
 \E_{\bQ_{n,\beta}}\left[|Y_{\bzeta}(u)Y_{\bzeta}(v)|\right]\dd u\dd v.
 \end{aligned}
\]
The right-hand side tends to zero.  Hence the right-hand side of
\eqref{EqCappedIntegrationByParts} converges in
\(\mathrm{L}^1(\bQ_{n,\beta})\) to
\[
 -\int_0^n\int_0^n\mathsf{r}_n''(u-v)Y_{\bzeta}(u)Y_{\bzeta}(v)\dd u\dd v.
\]
The capped left-hand side of
\eqref{EqCappedIntegrationByParts} converges almost surely to \(\mathsf{Err}_n^{\mathrm{per}}\).
The displayed \(\mathrm{L}^1\)-convergence has an almost-surely
convergent subsequence, whose limit is the displayed double integral.
Uniqueness of almost-sure limits therefore identifies that integral
with \(\mathsf{Err}_n^{\mathrm{per}}\).  The capped-kernel argument before
taking expectations applies verbatim to every fixed simple endpoint-free
configuration: the two endpoint neighbourhoods contain no atoms, and all remaining integrals involve a bounded piecewise-affine function with
finitely many jumps.  More explicitly, at
an opposite corner \(u=n-s\), \(v=t\), one has
\(|Y_{\bzeta}(u)Y_{\bzeta}(v)|=st\) for all sufficiently small \(s,t\), while
\[
 k_n(u-v)|Y_{\bzeta}(u)Y_{\bzeta}(v)|
 \leq C\left(\frac{st}{n^2}+\frac{st}{(s+t)^2}\right)\leq C.
\]
The capped derivatives satisfy the corresponding bound.  Indeed,
\eqref{EqCappedKernelBound} and \eqref{EqKernelBound} give, on the same
corner,
\[
 \left|\mathsf{r}_{n,\delta}''(u-v)-\mathsf{r}_n''(u-v)\right|st
 \leq
 Cst\left(
 \frac1{n^2}+\frac1{(s+t+\delta)^2}+\frac1{(s+t)^2}
 \right)
 \leq C.
\]
The two exceptional corner triangles have total area \(O(\delta^2)\), so
their contribution tends to zero.
This supplies deterministic domination at the only singular corners.
Thus the identity is
the pathwise statement \eqref{EqRIntegrationByParts}, not merely a
\(\bQ_{n,\beta}\)-almost-sure identity.
Finally, \eqref{EqKernelBound}, \eqref{EqYProductBound},
\eqref{EqRIntegrationByParts}, and the estimate on \(\mathcal I_n\)
give \eqref{EqPeriodicIntrinsicBound}.
\end{proof}

\section{The exterior boundary estimate}\label{SectionBoundary}

This section controls the long-range interaction left after replacing the
conditional interior configuration by the rescaled circular reference law.  The main
result, Proposition~\ref{PropBoundary}, proves that its expected absolute size
is sublinear in the interval length.  This is the gauge-invariant boundary
term in the direct Gibbs comparison proving
Proposition~\ref{PropCircularEntropy}.

Fix \(L\geq1\).  For \(n\geq1\), put
\[
 \bR_{n,L}\defeq(T_{L/n})_*\bQ_{n,\beta},
 \qquad \bR_{0,L}\defeq\delta_\varnothing.
\]
The dilation sends \(\Bin_{n,I_n}\) to \(\Bin_{n,I_L}\); consequently
\(\bR_{n,L}\sim\Bin_{n,I_L}\).  Its first intensity is
\((n/L)\mathrm{Leb}_{I_L}\), and
\(n\mapsto\bR_{n,L}\) is a Borel kernel on the countable sector space.

\begin{defn}
 Define the replacement law by
\begin{equation}\label{EqReplacementJointLaw}
 \bP_L^{\mathrm{rep}}(\dd\bC,\dd\bY_L)
 \defeq
 \bP(\dd\bC)\bR_{\bC(I_L),L}(\dd\bY_L).
\end{equation}
Under this law put \(\bX_L=\bC_{I_L}\).    
\end{defn}
 By
Proposition~\ref{PropInputs}(a),
\(\bC(I_L)=n_L(\bC_{I_L^c})\) almost surely.
The image of the joint law under
\((\bC,\bY_L)\mapsto(\bC_{I_L^c},\bX_L,\bY_L)\) is
\[
 \balpha_L(\dd\bxi)K_{L,\bxi}(\dd\bX)
 \bR_{n_L(\bxi),L}(\dd\betaeta).
\]
All fibrewise statements below are understood for
\(\balpha_L\)-almost every \(\bxi\), after modification on an exterior
null set if necessary.
Thus \(\bX_L\) and \(\bY_L\) are conditionally independent given
\(\bC_{I_L^c}\), and
\[
 \sN_L=\bX_L(I_L)=\bY_L(I_L)=n_L(\bC_{I_L^c}),\qquad
 \bCh_L\defeq\bY_L-\bX_L,\qquad
 \bCh_L(I_L)=0.
\]
Under \(\bP_L^{\mathrm{rep}}\), the configuration \(\bY_L\) is a
rescaled circular replacement of the actual interior configuration
\(\bX_L\), in the same interval and with the same rigid particle number.
Consequently, the signed replacement charge
\(\bCh_L=\bY_L-\bX_L\) is exactly neutral, which removes the leading
far-field contribution to the exterior interaction.
\begin{defn}
At the right endpoint define
\[
 \sS_{\mathrm R}(s)\defeq\bCh_L((L-s,L]),\qquad
 \sD_{\mathrm{out}}^{\mathrm R}(t)\defeq \bC((L,L+t])-t.
\]
At the left endpoint define
\[
 \sS_{\mathrm L}(s)\defeq\bCh_L([0,s)),\qquad
 \sD_{\mathrm{out}}^{\mathrm L}(t)\defeq \bC([-t,0))-t.
\]
\end{defn}
Thus \(\sS_{\mathrm L},\sS_{\mathrm R}\) are cumulative replacement
discrepancies, while \(\sD_{\mathrm{out}}^{\mathrm L},
\sD_{\mathrm{out}}^{\mathrm R}\) are the original exterior count
discrepancies.
The first intensity measure of \(\bP\) is Lebesgue measure, so \(\bC\)
charges no deterministic point almost surely.  For \(n\geq1\),
\(\bR_{n,L}\ll\Bin_{n,I_L}\), while \(\bR_{0,L}=\delta_\varnothing\); hence
\(\bY_L\) also charges no deterministic point almost surely.

The maps defining \(\sS_{\mathrm L},\sS_{\mathrm R}\) and
\(\sD_{\mathrm{out}}^{\mathrm L},\sD_{\mathrm{out}}^{\mathrm R}\) are jointly Borel in the
configuration variables and their time arguments.  Indeed, each is the
integral of a jointly Borel indicator function against a counting
measure.  Consequently the corresponding non-negative absolute integrals
are measurable.  The signed infinite-tail integrals below are defined only
after Lemma~\ref{LemBoundaryRepresentation} proves their absolute
convergence.

We fix the Lebesgue--Stieltjes convention used below.  The signed
measures \(\dd\sS_{\mathrm R}\) and \(\dd\sS_{\mathrm L}\) are, respectively, the push-forward
of \(\bCh_L\) under \(x\mapsto L-x\) and the measure
\(\bCh_L\) itself on
\([0,L]\).  The signed measures \(\dd\sD_{\mathrm{out}}^{\mathrm R}\) and \(\dd\sD_{\mathrm{out}}^{\mathrm L}\) are the
push-forwards of
\[
 (\bC-\mathrm{Leb})_{(L,\infty)}
 \quad\text{under }y\mapsto y-L,
 \qquad
 (\bC-\mathrm{Leb})_{(-\infty,0)}
 \quad\text{under }y\mapsto-y,
\]
respectively.  Thus, for \(A\in\{\mathrm L,\mathrm R\}\),
\(0\leq s\leq L\), and \(t\geq0\),
\[
 \sS_A(s)=\dd\sS_A([0,s)),\qquad
 \sD_{\mathrm{out}}^A(t)=\dd\sD_{\mathrm{out}}^A((0,t]).
\]
Fixed endpoints are almost surely atom-free, so no Stieltjes boundary atom
occurs there.  Choices at variable jump times differ only on a countable
set of parameter values and do not affect the later Lebesgue integrals.
Recall that Proposition~\ref{PropInputs}(b)--(c) gives
\begin{equation}\label{EqVarianceAllScales}
 v(t)\leq C_{\bP}t\quad(t>0),\qquad
 \lim_{t\to\infty}\frac{v(t)}t=0.
\end{equation}

\subsection{Second moments of the replacement discrepancy}

The replacement law shares its total particle number with the original
interior but has controlled partial-count fluctuations.  This second-moment
bound is the input to the Stieltjes boundary estimate.

\begin{lem}\label{LemReplacementVariance}
There is \(C=C(\bP,\beta)<\infty\) such that, for every \(L\geq1\),
\(0\leq s\leq L\), and \(A\in\{\mathrm L,\mathrm R\}\),
\begin{equation}\label{EqReplacementVariance}
 \E_{\bP_L^{\mathrm{rep}}}\left[\sS_A(s)^2\right]
 \leq C_\beta s+2\frac{s^2}{L^2}v(L)+2v(s)
 \leq Cs.
\end{equation}
\end{lem}

\begin{proof}
The assertion is immediate for \(s=0\).  For \(s=L\), exact neutrality
\(\bCh_L(I_L)=0\), together with the almost-sure absence of atoms at \(0\)
and \(L\), gives \(\sS_A(L)=0\) for either endpoint convention.  Fix
\(0<s<L\), condition on
\(\bC\), and write \(n=\sN_L\).  If \(n=0\), then
\(\sS_A(s)=0\).  Suppose that \(n\geq1\), and put
\[
 r\defeq\frac{ns}{L},\qquad a\defeq\min(r,n-r).
\]
Under the inverse dilation \(x\mapsto(n/L)x\), either boundary interval
has preimage an arc of length \(r\) in \(\mathbb T_n\).  Put
\[
 J_{L,s}^{\mathrm R}\defeq(L-s,L],\qquad
 J_{L,s}^{\mathrm L}\defeq[0,s).
\]
Conditionally on \(\bC\), the replacement \(\bY_L\) has law
\(\bR_{n,L}\), whose first intensity is
\((n/L)\mathrm{Leb}_{I_L}\).  Therefore
\[
 \E_{\bP_L^{\mathrm{rep}}}
 \left[\bY_L(J_{L,s}^A)\mid\bC\right]
 =\frac{ns}{L}=r.
\]
Under inverse dilation, the centered count
\(\bY_L(J_{L,s}^A)-r\) is the circular count discrepancy in an arc of
length \(r\).  Replacing that arc by its
complement when necessary and using rotation invariance,
\eqref{EqCBEVariance} gives
\[
 \Var(\bY_L(J_{L,s}^A)\mid\bC)
 \leq C_\beta\ell(a)\leq C_\beta a\leq C_\beta r.
\]

Let
\[
 \sD_{\mathrm{in}}^A(s)\defeq\bX_L(J_{L,s}^A)-s.
\]
Since \(|\sS_A(s)|\leq\sN_L\) and
\(\E_{\bP_L^{\mathrm{rep}}}\left[\sN_L^2\right]=L^2+v(L)<\infty\), the conditional second-moment identity
is applicable.  Moreover,
\[
 \E_{\bP_L^{\mathrm{rep}}}\left[\sS_A(s)\mid\bC\right]
 =r-\bX_L(J_{L,s}^A)
 =\frac{s}{L}\sD_L-\sD_{\mathrm{in}}^A(s).
\]
The \(\bC\)-marginal of \(\bP_L^{\mathrm{rep}}\) is \(\bP\), so
stationarity and deterministic endpoint atomlessness give
\[
 \E_{\bP_L^{\mathrm{rep}}}
 \left[(\sD_{\mathrm{in}}^A(s))^2\right]=v(s).
\]
Using this identity, \(\E_{\bP_L^{\mathrm{rep}}}[r]=s\), and
\eqref{EqVarianceAllScales}, we obtain
\begin{align*}
 \E_{\bP_L^{\mathrm{rep}}}[\sS_A(s)^2]
 &=
 \E_{\bP_L^{\mathrm{rep}}}\left[\Var(\bY_L(J_{L,s}^A)\mid\bC)\right]
 +\E_{\bP_L^{\mathrm{rep}}}\left[\left(\frac{s}{L}\sD_L-\sD_{\mathrm{in}}^A(s)\right)^2\right]\\
 &\leq C_\beta s
 +2\frac{s^2}{L^2}v(L)+2v(s)
 \leq Cs.
\end{align*}
This proves \eqref{EqReplacementVariance}.
\end{proof}

\subsection{Stieltjes representation}

The boundary error can be rewritten as the sum of two bilinear Stieltjes tail
integrals.  
For a finite signed measure \(\mu\), \(\|\mu\|_{\mathrm{TV}}\)
denotes its total variation norm; for a bounded function \(h\) on \(I_L\), put
\(\operatorname{osc}_{I_L}h=\sup_{I_L}h-\inf_{I_L}h\).

\begin{lem}\label{LemBoundaryRepresentation}
There is \(C=C(\bP,\beta)<\infty\) such that, for every \(L\geq1\) and
\(A\in\{\mathrm L,\mathrm R\}\),
\begin{equation}\label{EqBoundaryAbsolute}
 \E_{\bP_L^{\mathrm{rep}}}\left[\int_0^L\int_0^\infty
 \frac{|\sD_{\mathrm{out}}^A(t)\sS_A(s)|}{(s+t)^2}\dd t\dd s\right]
 \leq CL.
\end{equation}
The integrals
\[
 \widetilde{\sB}_{I_L}^A\defeq
 \int_0^L\int_0^\infty
 \frac{\sD_{\mathrm{out}}^A(t)\sS_A(s)}{(s+t)^2}\dd t\dd s
\]
are absolutely convergent almost surely and belong to \(\mathrm{L}^1\).
At the single point \((s,t)=(0,0)\), the displayed integrands are assigned
the value zero; this choice is immaterial for the Lebesgue integrals.
The
boundary variable in \eqref{EqBoundaryVariable} satisfies
\begin{equation}\label{EqBoundaryRepresentation}
 \sB_{I_L}(\bY_L,\bX_L,\bC_{I_L^c})
 =\widetilde{\sB}_{I_L}^{\mathrm R}+\widetilde{\sB}_{I_L}^{\mathrm L}
 \quad\text{almost surely}.
\end{equation}
\end{lem}

\begin{proof}
Almost surely, \(\bCh_L\) is a finite signed point measure of total mass
zero with no atom at \(0\) or \(L\).  Finiteness and endpoint atomlessness
make each \(\sS_A\) vanish in a neighbourhood of \(s=0\); exact neutrality
then makes it vanish in a neighbourhood of \(s=L\).  In particular, each
\(\sS_A\) is of bounded variation.

For \(y\in I_L^c\), put
\[
 \mathsf U_L^{\mathrm{rep}}(y)\defeq
 \int_{I_L}\mathsf{g}(x-y)\dd\bCh_L(x).
\]
Since each \(\sS_A\) vanishes on a neighbourhood of \(0\), the
representations below extend to absolutely continuous functions on every
compact interval \([0,R]\).  In particular, the Stieltjes integrations by
parts are legitimate at \(t=0\).  Thus, for \(y=L+t>L\), since
\(\sS_{\mathrm R}(0)=\sS_{\mathrm R}(L)=0\), Stieltjes integration by parts gives
\begin{equation}\label{EqPsiRepresentation}
 \mathsf U_L^{\mathrm{rep}}(L+t)=\int_0^L\frac{\sS_{\mathrm R}(s)}{s+t}\dd s,\qquad
 \frac{\dd}{\dd t}\mathsf U_L^{\mathrm{rep}}(L+t)
 =-\int_0^L\frac{\sS_{\mathrm R}(s)}{(s+t)^2}\dd s.
\end{equation}
For a finite exterior cutoff \(R\),
\begin{equation}
\begin{split}
 \int_{(L,L+R]}\mathsf U_L^{\mathrm{rep}}(y)\dd(\bC-\mathrm{Leb})(y)
 =
 \mathsf U_L^{\mathrm{rep}}(L+R)\sD_{\mathrm{out}}^{\mathrm R}(R)+\int_0^R\int_0^L
 \frac{\sD_{\mathrm{out}}^{\mathrm R}(t)\sS_{\mathrm R}(s)}{(s+t)^2}\dd s\dd t.
\end{split}
\label{EqStieltjesFinite}
\end{equation}
For \(y=-t<0\), Stieltjes integration by parts in the interior variable gives
\[
 \mathsf U_L^{\mathrm{rep}}(-t)=\int_0^L\frac{\sS_{\mathrm L}(s)}{s+t}\dd s,
 \qquad
 \frac{\dd}{\dd t}\mathsf U_L^{\mathrm{rep}}(-t)
 =-\int_0^L\frac{\sS_{\mathrm L}(s)}{(s+t)^2}\dd s.
\]
The push-forward under \(y\mapsto-y\) of
\((\bC-\mathrm{Leb})_{(-\infty,0)}\) has cumulative function
\(\sD_{\mathrm{out}}^{\mathrm L}\), and hence
\[
 \begin{aligned}
 \int_{[-R,0)}\mathsf U_L^{\mathrm{rep}}(y)\dd(\bC-\mathrm{Leb})(y)
 &=\int_{(0,R]}\mathsf U_L^{\mathrm{rep}}(-t)\dd\sD_{\mathrm{out}}^{\mathrm L}(t)\\
 &=\mathsf U_L^{\mathrm{rep}}(-R)\sD_{\mathrm{out}}^{\mathrm L}(R)
 -\int_0^R\sD_{\mathrm{out}}^{\mathrm L}(t)
 \frac{\dd}{\dd t}\mathsf U_L^{\mathrm{rep}}(-t)\dd t\\
 &=\mathsf U_L^{\mathrm{rep}}(-R)\sD_{\mathrm{out}}^{\mathrm L}(R)
 +\int_0^R\int_0^L
 \frac{\sD_{\mathrm{out}}^{\mathrm L}(t)\sS_{\mathrm L}(s)}{(s+t)^2}\dd s\dd t.
 \end{aligned}
\]

Since the \(\bC\)-marginal of the replacement law is \(\bP\),
stationarity gives
\[
 \E_{\bP_L^{\mathrm{rep}}}
 \left[(\sD_{\mathrm{out}}^A(t))^2\right]=v(t).
\]
By Lemma~\ref{LemReplacementVariance},
\eqref{EqVarianceAllScales}, and Cauchy--Schwarz,
\[
 \int_0^L\int_0^\infty
 \frac{\E_{\bP_L^{\mathrm{rep}}}\left[|\sD_{\mathrm{out}}^A(t)\sS_A(s)|\right]}{(s+t)^2}\dd t\dd s
 \leq
 C\int_0^L\int_0^\infty
 \frac{\sqrt{st}}{(s+t)^2}\dd t\dd s
 =\frac{C\pi L}{2}.
\]
Tonelli's theorem proves \eqref{EqBoundaryAbsolute} and the asserted
absolute convergence.

Here \(\|\mathord\cdot\|_r\) denotes the \(\mathrm{L}^r\)-norm under the
joint law for \(r\in\{1,2\}\).  Also, for \(R\geq L\),
\eqref{EqPsiRepresentation} gives
\[
 \|\mathsf U_L^{\mathrm{rep}}(L+R)\|_2
 \leq C\int_0^L\frac{\sqrt s}{R+s}\dd s
 \leq C\frac{L^{3/2}}{R}.
\]
Thus
\[
 \E_{\bP_L^{\mathrm{rep}}}\left[|\mathsf U_L^{\mathrm{rep}}(L+R)
 \sD_{\mathrm{out}}^{\mathrm R}(R)|\right]
 \leq C\frac{L^{3/2}}{\sqrt R}\longrightarrow0.
\]
The same estimate holds at the left endpoint.  Moreover,
\[
 \int_0^L\int_R^\infty
 \frac{\sqrt{st}}{(s+t)^2}\dd t\dd s
 \leq
 2R^{-1/2}\int_0^L\sqrt{s}\dd s
 =\frac{4L^{3/2}}{3\sqrt R}.
\]
It follows from \eqref{EqStieltjesFinite} and its left analogue that the
centred right and left exterior integrals converge in
\(\mathrm{L}^1\), as their
cutoffs tend to infinity, to \(\widetilde{\sB}_{I_L}^{\mathrm R}\) and
\(\widetilde{\sB}_{I_L}^{\mathrm L}\), respectively.

For an integer \(p\geq4L\), put
\[
 \sB_{I_L,p}\defeq \mathsf{M}_{I_L,\Lambda_p}(\bY_L,\bC)
 +\int_{I_L}U_L(x)\dd\bCh_L(x).
\]
Put \(V_p(x)=\int_{\Lambda_p}\mathsf{g}(x-y)\dd y\).  Adding and subtracting
Lebesgue measure outside \(I_L\) gives
\begin{equation*}
 \sB_{I_L,p}
 =
 \iint_{I_L\times(\Lambda_p\setminus I_L)}
 \mathsf{g}(x-y)\dd\bCh_L(x)\dd(\bC-\mathrm{Leb})(y)
 +\int_{I_L}V_p(x)\dd\bCh_L(x),
\end{equation*}
Since
\(\bCh_L(I_L)=0\) and
\[
 V_p'(x)=\log\left(\frac{p/2-x}{p/2+x}\right),\qquad
 \sup_{x\in I_L}|V_p'(x)|\leq C\frac Lp,
\]
we have \(\operatorname{osc}_{I_L}V_p\leq CL^2/p\).  Also,
\(\|\bCh_L\|_{\mathrm{TV}}\leq\bY_L(I_L)+\bX_L(I_L)=2\sN_L\) and
\(\E_{\bP_L^{\mathrm{rep}}}\left[\sN_L\right]=L\), so
\[
 \E_{\bP_L^{\mathrm{rep}}}\left[\left|\int_{I_L}V_p(x)\dd\bCh_L(x)\right|\right]
 =
 \E_{\bP_L^{\mathrm{rep}}}\left[\left|\int_{I_L}[V_p(x)-V_p(0)]\dd\bCh_L(x)\right|\right]
 \leq C\frac{L^3}{p}.
\]
The right and left tail lengths are \(p/2-L\) and \(p/2\), respectively.
Since \(p\geq4L\), both lengths are at least \(L\), so the preceding
estimates for \(R\geq L\) apply.
The preceding tail and endpoint estimates yield
\[
 \left\|\sB_{I_L,p}
 -\widetilde{\sB}_{I_L}^{\mathrm R}-\widetilde{\sB}_{I_L}^{\mathrm L}\right\|_1
 \leq C\left(\frac{L^3}{p}
 +\frac{L^{3/2}}{\sqrt{p/2-L}}\right).
\]
Work on the common full set where rigidity, exterior admissibility, and the
 uniform move limit of Proposition~\ref{PropInputs}\textup{(d)} all hold.  On
this set the move converges uniformly over the \(\sN_L\)-particle fibre.
Consequently, \(\sB_{I_L,p}\) converges almost surely under the joint law to
the variable in \eqref{EqBoundaryVariable}.
The preceding \(\mathrm{L}^1\)-convergence and uniqueness of limits in
probability
prove \eqref{EqBoundaryRepresentation}.
\end{proof}

Using sublinear large-scale variance in the preceding representation, this
proposition shows that the expected boundary correction is \(o(L)\).  This is
the boundary estimate used in the proof of
Proposition~\ref{PropCircularEntropy}.

\begin{prop}\label{PropBoundary}
We have,
\begin{equation}\label{EqBoundarySmall}
 \lim_{L\to\infty}\frac1L\E_{\bP_L^{\mathrm{rep}}}\left[
 |\sB_{I_L}(\bY_L,\bX_L,\bC_{I_L^c})|\right]=0.
\end{equation}
\end{prop}

\begin{proof}
By Lemmas~\ref{LemBoundaryRepresentation} and
\ref{LemReplacementVariance},
\[
 \E_{\bP_L^{\mathrm{rep}}}\left[
 |\sB_{I_L}(\bY_L,\bX_L,\bC_{I_L^c})|\right]
 \leq C\int_0^L\int_0^\infty
 \frac{\sqrt{s\,v(t)}}{(s+t)^2}\dd t\dd s.
\]
The discrepancy estimates \eqref{EqVarianceAllScales} and
Proposition~\ref{PropInputs}(c) show that
\(\sqrt{v(t)/t}\) is bounded and tends to zero as \(t\to\infty\).
After the substitutions \(s=Lu\) and \(t=Lr\),
\begin{equation}\label{EqBoundaryScaled}
 \frac{\E_{\bP_L^{\mathrm{rep}}}\left[
 |\sB_{I_L}(\bY_L,\bX_L,\bC_{I_L^c})|\right]}{L}
 \leq C\int_0^1\int_0^\infty
 \sqrt{\frac{v(Lr)}{Lr}}\frac{\sqrt{ur}}{(u+r)^2}\dd r\dd u,
\end{equation}
For every \(u>0\),
\[
 \int_0^\infty\frac{\sqrt{ur}}{(u+r)^2}\dd r
 =\int_0^\infty\frac{\sqrt x}{(1+x)^2}\dd x
 =\frac{\pi}{2};
\]
thus the kernel is integrable on \((0,1)\times(0,\infty)\).
For every \(r>0\), \(v(Lr)/(Lr)\to0\), while its square root is bounded.
Dominated convergence in \eqref{EqBoundaryScaled} proves
\eqref{EqBoundarySmall}.
\end{proof}

\section{The relative entropy estimate}
\label{SectionCircularEntropy}

This section proves the fundamental relative entropy estimate by comparing two
Gibbs laws on the same rigid particle-number fibre.  After pulling the
circular ensemble back to \(I_L\), the binomial reference and all
configuration-independent sector constants cancel.  The remaining terms are
the averaged dilation statistic, the periodic-to-intrinsic errors under the
two laws, and the gauge-invariant exterior boundary correction.

\subsection{The periodic error under the candidate law}

For \(a>0\), extend the kernel notation of
Lemma~\ref{LemPeriodicIntrinsic} by
\[
 k_a(t)
 \defeq
 \left(\frac{\pi}{a}\right)^2
 \csc^2\left(\frac{\pi t}{a}\right)-\frac1{t^2},
 \qquad 0<|t|<a,
\]
and put \(k_a(0)=\pi^2/(3a^2)\).  In square Lebesgue integrals we assign
the harmless value \(k_a(\pm a)=0\).  Recall that
\begin{equation}\label{EqKernelBoundContinuous}
 0\leq k_a(t)
 \leq
 C\left(
 \frac1{a^2}+\frac1{(a-|t|)^2}
 \right),
 \qquad |t|<a,
\end{equation}
with a universal constant \(C\).

\begin{lem}\label{LemActualPeriodicError}
For every \(L>0\),
\[
 \mathsf{Err}_{\sN_L}^{\mathrm{per}}(\bZ_L)
 \in\mathrm L^1(\bP).
\]
As \(L\to\infty\),
\begin{equation}\label{EqActualPeriodicError}
 \E_{\bP}\left[|\mathsf{Err}_{\sN_L}^{\mathrm{per}}(\bZ_L)|\right]=o(L).
\end{equation}
\end{lem}

\begin{proof}
Lemma~\ref{LemSimplicity} makes \(\bC\) simple almost surely.  A stationary
finite-intensity process has no point at a prescribed deterministic location,
so \(\bC\) has no point at \(0\) or \(L\) almost surely.  On
\(\{\sN_L=n\geq1\}\), put
\[
 Y_{\bZ_L}(u)=\bZ_L([0,u])-u,
 \qquad 0\leq u\leq n.
\]
The deterministic integration-by-parts identity in
Lemma~\ref{LemPeriodicIntrinsic} gives
\begin{equation}\label{EqActualPeriodicIBP}
 \mathsf{Err}_n^{\mathrm{per}}(\bZ_L)
 =
 \int_0^n\int_0^n
 k_n(u-v)Y_{\bZ_L}(u)Y_{\bZ_L}(v)
 \dd u\dd v.
\end{equation}
If \(u=ns/L\), then
\[
 Y_{\bZ_L}(u)
 =
 \bC([0,s])-\frac{ns}{L}
 =
 \sD_s-\frac{s}{L}\sD_L
 =
 \mathsf{Br}_L(s).
\]
Moreover,
\[
 k_n\left(\frac nL t\right)
 =
 \left(\frac Ln\right)^2k_L(t).
\]
Changing variables in \eqref{EqActualPeriodicIBP} therefore gives the exact representation
\begin{equation}\label{EqActualPeriodicBridge}
 \mathsf{Err}_{\sN_L}^{\mathrm{per}}(\bZ_L)
 =
 \int_0^L\int_0^L
 k_L(s-t)\mathsf{Br}_L(s)\mathsf{Br}_L(t)\dd s\dd t.
\end{equation}
On \(\{\sN_L=0\}\), one has
\(\sD_s=-s\) for \(0\leq s\leq L\), and hence \(\mathsf{Br}_L(s)=0\).
Thus \eqref{EqActualPeriodicBridge} also holds on the empty sector under
the preceding convention.

Define
\[
 \overline q_L(x)\defeq\frac{q_L(Lx)}{L},
 \qquad 0\leq x\leq1.
\]
Equation~\eqref{EqBridgeVarianceLinearBound} gives
\begin{equation}\label{EqRescaledBridgeBound}
 0\leq\overline q_L(x)
 \leq2C_{\bP}x(1-x).
\end{equation}
By \eqref{EqRescaledBridgeLimit},
\(\overline q_L(x)\to0\) for every \(0<x<1\).

Since \(k_L\geq0\), \eqref{EqActualPeriodicBridge}, Cauchy--Schwarz, and
Tonelli's theorem give
\[
 \E_{\bP}[|\mathsf{Err}_{\sN_L}^{\mathrm{per}}(\bZ_L)|]
 \leq
 \int_0^L\int_0^L
 k_L(s-t)\sqrt{q_L(s)q_L(t)}
 \dd s\dd t.
\]
Using
\[
 k_L(Lz)=L^{-2}k_1(z),
 \qquad
 q_L(Lx)=L\overline q_L(x),
\]
and rescaling the preceding inequality yields
\begin{equation}\label{EqActualPeriodicRescaled}
 \frac1L
 \E_{\bP}[|\mathsf{Err}_{\sN_L}^{\mathrm{per}}(\bZ_L)|]
 \leq
 \int_0^1\int_0^1
 k_1(x-y)\sqrt{\overline q_L(x)\overline q_L(y)}
 \dd x\dd y.
\end{equation}
By \eqref{EqKernelBoundContinuous} and
\eqref{EqRescaledBridgeBound}, the integrand is bounded by a constant
multiple of
\[
 \left(1+\frac1{(1-|x-y|)^2}\right)
 \sqrt{x(1-x)y(1-y)}.
\]
This function is integrable on \((0,1)^2\).  Only the two opposite corners
require verification.  Near \((x,y)=(1,0)\), put \(a=1-x\) and \(c=y\).
The singular part is bounded by
\[
 C\frac{\sqrt{ac}}{(a+c)^2}.
\]
On \(0<a,c<\delta\), the change of variables
\(r=a+c\), \(u=a/(a+c)\) has
\(\dd a\,\dd c=r\,\dd r\,\dd u\).  Enlarging its image to
\((0,2\delta)\times(0,1)\), the integral is bounded by a constant multiple of
\[
 \int_0^{2\delta}\int_0^1
 \sqrt{u(1-u)}\dd u\dd r<\infty
\]
for a sufficiently small \(\delta>0\).
The corner \((0,1)\) is identical.  Thus the right-hand side of
\eqref{EqActualPeriodicRescaled} is finite for every \(L>0\), and dominated
convergence proves \eqref{EqActualPeriodicError}.
\end{proof}

\subsection{Direct Gibbs comparison}

For \(n\geq1\), write
\[
 \mathsf T_{n,L}\betaeta\defeq(T_{n/L})_*\betaeta.
\]
Thus \(\mathsf T_{n,L}\) sends \(\bR_{n,L}\) to
\(\bQ_{n,\beta}\).

Define
\[
 \nu_{L,\bxi}
 =(T_{n_L(\bxi)/L})_*K_{L,\bxi}
 \quad\text{when }n_L(\bxi)\geq1,
 \qquad
 \nu_{L,\bxi}=\delta_\varnothing
 \quad\text{when }n_L(\bxi)=0.
\]
The law \(\nu_{L,\bxi}\) is the conditional candidate law transported
from \(I_L\) to the neutral interval \(I_{n_L(\bxi)}\).  It is therefore
the version of \(K_{L,\bxi}\) that lives on the same space as
\(\bQ_{n_L(\bxi),\beta}\). Since \(n_L\) is Borel with countable range and dilation is Borel on every
positive sector, \(\bxi\mapsto\nu_{L,\bxi}\) is a Borel probability kernel.
Moreover, observe that,
\[
 \int_{\Conf(I_L^c)}\nu_{L,\bxi}\,\balpha_L(\dd\bxi)
 =\operatorname{Law}_{\bP}(\bZ_L).
\]

We have the following bound for the relative entropy. Each term on the right hand side of \eqref{EqDirectFibreComparison} can then be controlled.

\begin{lem}\label{LemDirectGibbsComparison}
For \(\balpha_L\)-almost every \(\bxi\), put
\(n=n_L(\bxi)\) and \(D=n-L\).  If \(n\geq1\), then
\begin{equation}\label{EqDirectFibreComparison}
\begin{aligned}
 \Ent(K_{L,\bxi}|\bR_{n,L})
 &\leq b \bigg[
 -K_{L,\bxi}
   \bigl(\mathsf{Err}_n^{\per}\circ\mathsf T_{n,L}\bigr)
 +\E_{\bQ_{n,\beta}}\bigl[\mathsf{Err}_n^{\per}\bigr]-\frac{2D}{n}
 K_{L,\bxi}\bigl(\mathcal A_n\circ\mathsf T_{n,L}\bigr)
 \bigg]\\
 &+\beta
 \int_{\Conf(I_L)}K_{L,\bxi}(\dd\bX)
 \int_{\Conf(I_L)}\bR_{n,L}(\dd\bY)\,
 \sB_{I_L}(\bY,\bX,\bxi).
\end{aligned}
\end{equation}
On the sector \(n=0\), every quotient-dependent expression is assigned
the value zero; then
\(K_{L,\bxi}=\bR_{0,L}=\delta_\varnothing\), and both the entropy and
all remaining comparison terms are zero.
\end{lem}

\begin{proof}
If \(n=0\), rigidity and the definition of the reference give
\[
 K_{L,\bxi}=\bR_{0,L}=\delta_\varnothing,
\]
so the assertion is immediate.  Assume \(n\geq1\).  Let
\(\mathscr G_{n,L}\) be the Borel set of simple \(n\)-point configurations in
\(I_L\) having no point at either endpoint.  Both \(K_{L,\bxi}\) and
\(\bR_{n,L}\) assign full mass to this set and are equivalent there to
\(\Bin_{n,I_L}\).

On \(\mathscr G_{n,L}\), let us define
\begin{equation}\label{EqDirectPotentialDifference}
 \Phi_{L,\bxi}(\betaeta)
 \defeq
 b\left[
  W_n^{\per}(\mathsf T_{n,L}\betaeta)
  -W_{I_L}^{\intE}(\betaeta)
 \right]-\beta\widehat A_{\bxi}(\betaeta),
\end{equation}
Extend this function by zero outside \(\mathscr G_{n,L}\).  Thus the
definition involves only finite quantities and no difference of extended-real
energies.  Equation~\eqref{EqConditionalGibbs}, together with
\eqref{EqMasterFibreIdentity}, gives
\[
 \frac{\dd K_{L,\bxi}}{\dd\Bin_{n,I_L}}(\betaeta)
 \propto
 \exp\left[
  -bW_{I_L}^{\intE}(\betaeta)-\beta\widehat A_{\bxi}(\betaeta)
 \right].
\]
Similarly, Definition~\ref{DefCircularEnsemble} and
\eqref{EqPeriodicEnergyIdentity} give
\[
 \frac{\dd\bR_{n,L}}{\dd\Bin_{n,I_L}}(\betaeta)
 \propto
 \exp\left[-bW_n^{\per}(\mathsf T_{n,L}\betaeta)\right].
\]
Consequently,
\begin{equation}\label{EqDirectDensityRatio}
 \frac{\dd K_{L,\bxi}}{\dd\bR_{n,L}}
 =
 \frac{\exp(\Phi_{L,\bxi})}
 {\bR_{n,L}[\exp(\Phi_{L,\bxi})]}.
\end{equation}
Thus \(\Phi_{L,\bxi}\) is the unnormalised logarithmic density ratio of
\(K_{L,\bxi}\) relative to \(\bR_{n,L}\).  This representation allows
Jensen's inequality to bound the relative entropy by a difference of
expectations, in which every fibrewise constant cancels.
Also observe that, the denominator is finite and strictly positive because it is the ratio of
the two finite Gibbs normalising constants.

The functions \(\widehat A_{\bxi}\) and
\(\mathcal A_n\circ\mathsf T_{n,L}\) are bounded on the fixed fibre.
Lemma~\ref{LemPeriodicIntrinsic} gives
\[
 \bR_{n,L}
 \bigl(|\mathsf{Err}_n^{\per}\circ\mathsf T_{n,L}|\bigr)
 =
 \E_{\bQ_{n,\beta}}[|\mathsf{Err}_n^{\per}|]<\infty.
\]
Moreover, Lemma~\ref{LemActualPeriodicError} and disintegration imply, after
removing an exterior null set, that
\[
 K_{L,\bxi}
 \bigl(|\mathsf{Err}_n^{\per}\circ\mathsf T_{n,L}|\bigr)<\infty.
\]
Equations~\eqref{EqAffineIdentity} and
\eqref{EqPeriodicIntrinsicDifference} give, on \(\mathscr G_{n,L}\),
\begin{equation}\label{EqDirectAffineExpansion}
 \begin{aligned}
 W_n^{\per}(\mathsf T_{n,L}\betaeta)-W_{I_L}^{\intE}(\betaeta)
 &=-\mathsf{Err}_n^{\per}(\mathsf T_{n,L}\betaeta)
   -\frac{2D}{n}\mathcal A_n(\mathsf T_{n,L}\betaeta)+c_{n,L},\\
 c_{n,L}&=D^2\log(L)-\frac32D^2-n\log\left(\frac Ln\right).
 \end{aligned}
\end{equation}
It follows that
\(\Phi_{L,\bxi}\in\mathrm L^1(K_{L,\bxi})
\cap\mathrm L^1(\bR_{n,L})\).  Equation~\eqref{EqDirectDensityRatio} and
Jensen's inequality therefore give
\begin{align}
 \Ent(K_{L,\bxi}|\bR_{n,L})
 &=
 K_{L,\bxi}(\Phi_{L,\bxi})
 -\log\bR_{n,L}[\exp(\Phi_{L,\bxi})]\notag\\
 &\leq
 K_{L,\bxi}(\Phi_{L,\bxi})
 -\bR_{n,L}(\Phi_{L,\bxi}).
\label{EqDirectPairwiseJensen}
\end{align}

The right-hand side of \eqref{EqDirectPairwiseJensen} is the integral of
\(\Phi_{L,\bxi}(\bX)-\Phi_{L,\bxi}(\bY)\) against
\(K_{L,\bxi}(\dd\bX)\bR_{n,L}(\dd\bY)\).  The constant \(c_{n,L}\)
cancels from this difference.  Furthermore, the
first intensity measure of \(\bQ_{n,\beta}\) is Lebesgue measure by
Lemma~\ref{LemPeriodicIntrinsic}; hence
\[
 \bR_{n,L}
 \bigl(\mathcal A_n\circ\mathsf T_{n,L}\bigr)
 =
 \E_{\bQ_{n,\beta}}[\mathcal A_n]
 =
 \int_{I_n}U_n(u)
 \dd\bigl(\lambda_n-\mathrm{Leb}_{I_n}\bigr)(u)=0.
\]
Also,
\[
 \bR_{n,L}
 \bigl(\mathsf{Err}_n^{\per}\circ\mathsf T_{n,L}\bigr)
 =
 \E_{\bQ_{n,\beta}}[\mathsf{Err}_n^{\per}],
\]
while, before any averaging over the exterior,
\[
 \widehat A_{\bxi}(\bY)-\widehat A_{\bxi}(\bX)
 =\sB_{I_L}(\bY,\bX,\bxi).
\]
Substitution proves \eqref{EqDirectFibreComparison}.
\end{proof}

The kernels \(\bxi\mapsto K_{L,\bxi}\) and
\(\bxi\mapsto\bR_{n_L(\bxi),L}\) are Borel.  The standard representation of
relative entropy as the supremum over a countable sequence of finite
measurable partitions therefore shows that
\[
 \bxi\longmapsto
 \Ent(K_{L,\bxi}|\bR_{n_L(\bxi),L})
\]
is Borel.  For \(\balpha_L\)-almost every \(\bxi\) with
\(n_L(\bxi)\geq1\), the fibrewise dilation is a bimeasurable bijection
and therefore leaves relative entropy unchanged.  When
\(n_L(\bxi)=0\), both laws are \(\delta_\varnothing\).  Hence
\begin{align}
 \mathcal H_L^{\mathrm{circ}}
 &\defeq
 \int_{\Conf(I_L^c)}
 \Ent(K_{L,\bxi}|\bR_{n_L(\bxi),L})
 \balpha_L(\dd\bxi)=
 \int_{\Conf(I_L^c)}
 \Ent(\nu_{L,\bxi}|\bQ_{n_L(\bxi),\beta})
 \balpha_L(\dd\bxi).
\label{EqCircularEntropyFixedSpace}
\end{align}
Finally, put
\[
 p_{L,n}\defeq\bP(\sN_L=n),\qquad n\in\mathbb Z_+.
\]

We can now prove our main estimate. 

\begin{prop}\label{PropCircularEntropy}
We have
\begin{equation}\label{EqCircularEntropy}
 \mathcal H_L^{\mathrm{circ}}=o(L).
\end{equation}
\end{prop}

\begin{proof}
We use throughout the conventions
\[
 \bQ_{0,\beta}=\bR_{0,L}=\delta_\varnothing,
 \qquad
 \mathsf{Err}_0^{\per}(\varnothing)
 =\mathcal A_0(\varnothing)=0,
\]
and assign the value zero on \(\{n=0\}\) to every expression involving
\(\mathsf T_{n,L}\) or \(D/n\).  Thus the zero-particle sector makes no
contribution to any of the terms below.

We first control the periodic error under the reference laws. Observe that,
Lemma~\ref{LemPeriodicIntrinsic} and the elementary inequality
\[
 \log^2(2+n)\leq C(1+\sqrt n)
\]
give the following,
\[
\begin{aligned}
 \sum_{n\geq0}p_{L,n}
 \E_{\bQ_{n,\beta}}
 \bigl[|\mathsf{Err}_n^{\per}|\bigr]
 \leq
 C_\beta\sum_{n\geq0}p_{L,n}(1+\sqrt n)=C_\beta
 \left(1+\E_{\bP}[\sqrt{\sN_L}]\right)&\leq
 C_\beta
 \left(1+\sqrt{\E_{\bP}[\sN_L]}\right)\\
 &=C_\beta(1+\sqrt L)
 =o(L).
\end{aligned}
\]
Here Jensen's inequality was used in the second inequality, while
Proposition~\ref{PropInputs}\textup{(b)} gives
\(\E_{\bP}[\sN_L]=L\).  In particular, the reference periodic-error
mixture is finite.

We next make explicit how the remaining fibrewise quantities are
averaged.  Introduce the joint probability measure
\[
 \mathbb M_L(\dd\bxi,\dd\bX,\dd\bY)
 \defeq
 \balpha_L(\dd\bxi)
 K_{L,\bxi}(\dd\bX)
 \bR_{n_L(\bxi),L}(\dd\bY).
\]
By disintegration of \(\bP\), number rigidity, and
\eqref{EqReplacementJointLaw}, \(\mathbb M_L\) is the push-forward of
\(\bP_L^{\mathrm{rep}}\) under
\[
 (\bC,\bY_L)
 \longmapsto
 (\bC_{I_L^c},\bC_{I_L},\bY_L).
\]
Consequently, under this identification,
\[
 n_L(\bxi)=\bX(I_L)=\bY(I_L)=\sN_L,
 \qquad
 n_L(\bxi)-L=\sD_L,
\]
and, on \(\{\sN_L\geq1\}\),
\[
 \mathsf T_{\sN_L,L}\bX
 =(T_{\sN_L/L})_*\bC_{I_L}
 =\bZ_L.
\]
Rigidity also gives
\[
 \balpha_L\{\bxi:n_L(\bxi)=n\}
 =\bP(\sN_L=n)=p_{L,n}.
\]

For \(\balpha_L\)-almost every \(\bxi\), put
\[
 n=n_L(\bxi),\qquad D=n-L.
\]
Taking absolute values term by term in
\eqref{EqDirectFibreComparison},
gives
\[
\begin{aligned}
 \Ent(K_{L,\bxi}|\bR_{n,L})
 &\leq{}
 bK_{L,\bxi}\left(
   |\mathsf{Err}_n^{\per}\circ\mathsf T_{n,L}|
 \right)
 +b\E_{\bQ_{n,\beta}}
   \bigl[|\mathsf{Err}_n^{\per}|\bigr]+
 2bK_{L,\bxi}\left(
   \left|
   \frac{D}{n}
   \mathcal A_n\circ\mathsf T_{n,L}
   \right|
 \right)\\
 &+
 \beta
 \int_{\Conf(I_L)}K_{L,\bxi}(\dd\bX)
 \int_{\Conf(I_L)}\bR_{n,L}(\dd\bY)\,
 |\sB_{I_L}(\bY,\bX,\bxi)|.
\end{aligned}
\]
We now identify the average of each term on the right-hand side. First, disintegration and the definition of \(\bZ_L\) give
\[
\begin{aligned}
 &\int_{\Conf(I_L^c)}
 K_{L,\bxi}\left(
 |\mathsf{Err}_{n_L(\bxi)}^{\per}
       \circ\mathsf T_{n_L(\bxi),L}|
 \right)\balpha_L(\dd\bxi)=
 \E_{\bP}\left[
 |\mathsf{Err}_{\sN_L}^{\per}(\bZ_L)|
 \right]
 =o(L),
\end{aligned}
\]
where the final estimate is
Lemma~\ref{LemActualPeriodicError}. Second, the law of \(n_L(\bC_{I_L^c})\) is the law of \(\sN_L\), and
therefore
\[
\begin{aligned}
 &\int_{\Conf(I_L^c)}
 \E_{\bQ_{n_L(\bxi),\beta}}
 \bigl[|\mathsf{Err}_{n_L(\bxi)}^{\per}|\bigr]
 \balpha_L(\dd\bxi)=
 \sum_{n\geq0}p_{L,n}
 \E_{\bQ_{n,\beta}}
 \bigl[|\mathsf{Err}_n^{\per}|\bigr]
 =o(L).
\end{aligned}
\]
Third, another application of the disintegration identity gives
\[
\begin{aligned}
 &\int_{\Conf(I_L^c)}
 K_{L,\bxi}\left(
 \left|
 \frac{n_L(\bxi)-L}{n_L(\bxi)}
 \mathcal A_{n_L(\bxi)}
 \circ\mathsf T_{n_L(\bxi),L}
 \right|
 \right)\balpha_L(\dd\bxi)=
 \E_{\bP}\left[
 \left|
 \frac{\sD_L}{\sN_L}
 \mathcal A_{\sN_L}(\bZ_L)
 \right|
 \right]\leq C\sqrt{Lv(L)}
 =o(L)
\end{aligned}
\]
by Lemma~\ref{LemAveragedDilationStatistic}.  As above, the quotient is
defined to be zero on \(\{\sN_L=0\}\). Finally, the push-forward description of \(\mathbb M_L\) yields
\[
\begin{aligned}
 &\int_{\Conf(I_L^c)}
 \int_{\Conf(I_L)}K_{L,\bxi}(\dd\bX)
 \int_{\Conf(I_L)}\bR_{n_L(\bxi),L}(\dd\bY)\,
 |\sB_{I_L}(\bY,\bX,\bxi)|
 \balpha_L(\dd\bxi)\\
 &\qquad=
 \E_{\bP_L^{\mathrm{rep}}}\left[
 \left|
 \sB_{I_L}(\bY_L,\bX_L,\bC_{I_L^c})
 \right|
 \right]
 =o(L)
\end{aligned}
\]
by Proposition~\ref{PropBoundary}.

Putting everything together using
\eqref{EqCircularEntropyFixedSpace}, we obtain
\[
\begin{aligned}
 0\leq\mathcal H_L^{\mathrm{circ}}
 &\leq{}
 b\E_{\bP}\left[
 |\mathsf{Err}_{\sN_L}^{\per}(\bZ_L)|
 \right]+
 b\sum_{n\geq0}p_{L,n}
 \E_{\bQ_{n,\beta}}
 \bigl[|\mathsf{Err}_n^{\per}|\bigr]+
 2b\E_{\bP}\left[
 \left|
 \frac{\sD_L}{\sN_L}
 \mathcal A_{\sN_L}(\bZ_L)
 \right|
 \right]\\
 &+
 \beta\E_{\bP_L^{\mathrm{rep}}}\left[
 \left|
 \sB_{I_L}(\bY_L,\bX_L,\bC_{I_L^c})
 \right|
 \right].
\end{aligned}
\]
and since as shown above each of the four terms on the right-hand side are \(o(L)\), we get
\[
 \mathcal H_L^{\mathrm{circ}}=o(L),
\]
which proves \eqref{EqCircularEntropy}.
\end{proof}

\section{Cyclic-gap transport and Palm identification}
\label{SectionPalm}

Proposition~\ref{PropCircularEntropy}, equivalently
\eqref{EqCircularEntropy}, gives \(\mathcal H_L^{\mathrm{circ}}=o(L)\).
We now show how this estimate can be leveraged to compare the limiting Palm measures of the conditional distribution of a DLR solution $\bP$ to the limiting Palm measure of the circular $\beta$ ensemble. We need some preliminary constructions.

\subsection{Fixed-sector averaging}

Recall that \(p_{L,n}=\bP(\sN_L=n)\).

\begin{defn}
For \(p_{L,n}>0\), define
\begin{equation}\label{EqFixedSectorAverage}
 \overline\nu_{L,n}
 \defeq
 \frac{1}{p_{L,n}}
 \int_{\{n_L(\bxi)=n\}}
 \nu_{L,\bxi}\,\balpha_L(\dd\bxi).
\end{equation}
When \(p_{L,n}=0\), set
\(\overline\nu_{L,n}=\bQ_{n,\beta}\) (the actual choice on such a sector is
irrelevant).  
\end{defn}
Convexity of relative entropy gives
\begin{equation}\label{EqFixedSectorEntropy}
 \sum_{n\geq0}p_{L,n}
 \Ent(\overline\nu_{L,n}|\bQ_{n,\beta})
 \leq \mathcal H_L^{\mathrm{circ}}=o(L).
\end{equation}
Indeed, apply convexity on each set \(\{n_L(\bxi)=n\}\) and then sum over
\(n\). 

For \(p_{L,n}>0\), observe that we have,
\[
 \overline\nu_{L,n}
 =\operatorname{Law}_{\bP}(\bZ_L\mid\sN_L=n).
\]

For \(n\geq1\), let
\[
 \mathcal S_n^\circ
 \defeq
 \left\{(x_1,\ldots,x_n)\in(0,\infty)^n:
 \sum_{i=1}^nx_i=n\right\}.
\]
If
\(\bzeta=\sum_{i=1}^n\delta_{u_i}\), with
\(0<u_1<\cdots<u_n<n\), extend the indices by
\(u_{i+n}=u_i+n\) and define the cyclic gap vector rooted at \(u_j\) by
\[
 \mathsf{Gap}_{n,j}(\bzeta)
 =
 (u_{j+1}-u_j,\ldots,u_{j+n}-u_{j+n-1})
 \in\mathcal S_n^\circ.
\]
The uniformly rooted gap kernel is
\[
\mathsf{Gap}_n(\bzeta)
 \defeq
 \frac1n\sum_{j=1}^n
 \delta_{\mathsf{Gap}_{n,j}(\bzeta)}.
\]
On the complement of the Borel set of configurations admitting the unique
ordering \(0<u_1<\cdots<u_n<n\), define
\(\mathsf{Gap}_n(\bzeta)=\delta_{(1,\ldots,1)}\).  The ordering map on the
simple sector is Borel, so this makes \(\mathsf{Gap}_n\) a globally defined
Borel probability kernel.  If \(p_{L,n}>0\), the conditional Gibbs formula
gives \(\overline\nu_{L,n}\ll\Bin_{n,I_n}\), while
\(\bQ_{n,\beta}\) is equivalent to that binomial law on the collision-free
sector.  Thus
\[
 \overline\nu_{L,n}\ll\Bin_{n,I_n}\sim\bQ_{n,\beta}.
\]
For
\(p_{L,n}=0\), it follows from the convention
\(\overline\nu_{L,n}=\bQ_{n,\beta}\).
\begin{defn}
We define,
\begin{equation}\label{EqGapLaws}
 \rho_{L,n}
 \defeq
 \int_{\Conf_n(I_n)}
 \mathsf{Gap}_n(\bzeta)\,
 \overline\nu_{L,n}(\dd\bzeta),
 \qquad
 \mu_n
 \defeq
 \int_{\Conf_n(I_n)}
 \mathsf{Gap}_n(\bzeta)\,
 \bQ_{n,\beta}(\dd\bzeta).
\end{equation}
\end{defn}
The measure \(\overline\nu_{L,n}\) averages the exterior-conditioned
candidate laws at fixed particle number, so that convexity converts the
averaged fibre entropy into a deterministic sectorwise bound.  The laws
\(\rho_{L,n}\) and \(\mu_n\) then retain only a uniformly particle-rooted
cyclic gap vector under the candidate and circular sectors,
respectively, removing absolute position and allowing cyclic symmetry
to distribute the entropy cost among the \(n\) gaps.

Applying the same probability kernel to the preceding absolute-continuity
relation gives \(\rho_{L,n}\ll\mu_n\).
Both laws are invariant under the cyclic shift
\[
 \sigma_n(x_1,\ldots,x_n)=(x_2,\ldots,x_n,x_1).
\]
Using the fact that relative entropy cannot increase when the same probability kernel is applied to both laws, used first under the uniform-root kernel and then under the
gap map, gives the entropy inequality
\begin{equation}\label{EqGapEntropyContraction}
 \Ent(\rho_{L,n}|\mu_n)
 \leq
 \Ent(\overline\nu_{L,n}|\bQ_{n,\beta}).
\end{equation}

\subsection{Transport on the cyclic-gap simplex}

Let us define,
\[
 \psi(q)\defeq q-1-\log(q),\qquad q>0.
\]

The optimal-transport theorem below is a so-called above-tangent inequality.  For a
Gibbs law \(\mu(\dd x)\propto\mathrm e^{-V(x)}\dd x\), it constructs, for
each \(\rho\ll\mu\), a coupling \((\mathsf X,\mathsf Y)\) of \(\mu\) and
\(\rho\) such that
\[
 \E\!\left[
 V(\mathsf Y)-V(\mathsf X)
 -\nabla V(\mathsf X)\mathbin{\cdot}
  (\mathsf Y-\mathsf X)
 \right]
 \leq \Ent(\rho\mid\mu).
\]
The integrand is the non-negative amount by which \(V(\mathsf Y)\) lies
above the tangent plane to the convex function \(V\) at \(\mathsf X\).
For the circular gap potential below, this remainder controls
\(\beta\sum_i\psi(\mathsf Y_i/\mathsf X_i)\). In the following lemma we first write $\mu_n$ as a Gibbs measure with potential $V_n$ from \eqref{EqCircularGapPotential} and then apply standard transport results from the literature \cite{BLTransport,CETransport}. The only issue is $V_n$ is rather singular and we need an approximation scheme in order to apply these results.

\begin{lem}\label{LemCyclicGapTransport}
For every \(n\geq3\) and every probability measure
\(\rho\) on \(\mathcal S_n^\circ\) with
\(\Ent(\rho|\mu_n)<\infty\), there is a coupling
 \(\pi\) of \(\mu_n\) and \(\rho\).  Denote its coordinate random
 elements by \((\mathsf X,\mathsf Y)\).  Then
\begin{equation}\label{EqCyclicGapTransport}
 \beta\,\E_\pi\left[
 \sum_{i=1}^n
 \psi\left(\frac{\mathsf Y_i}{\mathsf X_i}\right)\right]
 \leq\Ent(\rho|\mu_n).
\end{equation}
If \(\rho\) is invariant under \(\sigma_n\), the coupling can be chosen
invariant under simultaneous cyclic shifts.  In that case, for every
\(1\leq i\leq n\),
\begin{equation}\label{EqOneGapTransport}
 \E_\pi\left[
 \psi\left(\frac{\mathsf Y_i}{\mathsf X_i}\right)\right]
 \leq\frac{\Ent(\rho|\mu_n)}{\beta n}.
\end{equation}
\end{lem}

\begin{proof}
Recall the circular kernel \(\mathsf{g}_n\) from
Section~\ref{SectionCircularRecovery}.  With cyclic indices, define on
\(\mathcal S_n^\circ\)
\begin{equation}\label{EqCircularGapPotential}
 V_n(x)
 \defeq
 \frac{\beta}{2}
 \sum_{i=1}^n\sum_{k=1}^{n-1}
 \mathsf{g}_n(x_i+\cdots+x_{i+k-1}).
\end{equation}

We first identify the law \(\mu_n\) explicitly.  For
\(x=(x_1,\ldots,x_n)\in\mathcal S_n^\circ\), put
\[
 s_0(x)=0,
 \qquad
 s_k(x)=\sum_{\ell=1}^k x_\ell,
 \quad 1\leq k\leq n-1,
\]
and, for \(t\in[0,n)\), define
\[
 \bzeta_{t,x}
 \defeq
 \sum_{k=0}^{n-1}\delta_{(t+s_k(x))\bmod n}.
\]
Thus \(t\) is the position of the marked particle and \(x\) is the vector
of clockwise gaps starting from that particle.

For cyclic indices, set
\[
 r_{i,k}(x)
 \defeq x_i+\cdots+x_{i+k-1},
 \qquad
 1\leq i\leq n,\quad 1\leq k\leq n-1.
\]
The quantity \(r_{i,k}(x)\) is the clockwise distance from the \(i\)-th
particle to the \(k\)-th subsequent particle.  Its reverse-oriented distance
is
\[
 r_{i+k,n-k}(x)=n-r_{i,k}(x).
\]
Since
\[
 \mathsf{g}_n(n-r)=\mathsf{g}_n(r),\qquad 0<r<n,
\]
the two orientations contribute the same interaction.  Every unordered pair
of particles occurs exactly twice in the double sum, once in each
orientation.  Consequently,
\[
 \frac12
 \sum_{i=1}^n\sum_{k=1}^{n-1}\mathsf{g}_n(r_{i,k}(x))
 =
 \sum_{0\leq i<j\leq n-1}
 \mathsf{g}_n\bigl(s_j(x)-s_i(x)\bigr)
 =
 \mathsf H_n^{\per}(\bzeta_{t,x}),
\]
and hence
\[
 V_n(x)=\beta\mathsf H_n^{\per}(\bzeta_{t,x}).
\]
In particular, the right-hand side is independent of the absolute rotation
\(t\).

Let
\[
 A_n\defeq
 \left\{x\in\R^n:\sum_{i=1}^n x_i=n\right\},
\]
and let \(\mathcal L_{A_n}\) denote its induced
\((n-1)\)-dimensional Lebesgue measure.  Mark one of the \(n\) particles
uniformly and list the particles clockwise from the marked particle.  Apart
from the null sets corresponding to collisions or to a particle at the fixed
cut, this gives the coordinates
\[
 (t,x_1,\ldots,x_{n-1}),
 \qquad
 x_n=n-\sum_{i=1}^{n-1}x_i.
\]
Before reduction modulo \(n\), the corresponding lifted particle locations
are
\[
 t,\quad
 t+x_1,\quad
 t+x_1+x_2,\quad\ldots,\quad
 t+x_1+\cdots+x_{n-1}.
\]
On each of the finitely many regions determined by which of these locations
cross the cut at \(n\), the derivative of this change of variables is
triangular and has constant non-zero determinant.  Passing from
\(\dd x_1\cdots\dd x_{n-1}\) to \(\mathcal L_{A_n}(\dd x)\) changes this
determinant only by another constant.  In particular, no factor depending on
\(x\) appears.

Therefore, for every bounded Borel \(F\geq0\),
\[
 \begin{aligned}
 \int_{\mathcal S_n^\circ}F(x)\mu_n(\dd x)
 &=
 \frac1n
 \int_{\Conf_n(I_n)}
 \sum_{j=1}^n
 F\bigl(\mathsf{Gap}_{n,j}(\bzeta)\bigr)
 \bQ_{n,\beta}(\dd\bzeta)\\
 &=
 C_{n,\beta}
 \int_0^n\int_{\mathcal S_n^\circ}
 F(x)\exp\!\left[-\beta
 \mathsf H_n^{\per}(\bzeta_{t,x})\right]
 \mathcal L_{A_n}(\dd x)\dd t\\
 &=
 nC_{n,\beta}
 \int_{\mathcal S_n^\circ}
 F(x)\mathrm e^{-V_n(x)}
 \mathcal L_{A_n}(\dd x),
 \end{aligned}
\]
where \(C_{n,\beta}>0\) is independent of \(F\), \(t\), and \(x\).
Normalising at \(F=1\), we obtain
\begin{equation}\label{EqCircularGapDensity}
 Z_n^{\mathrm{gap}}
 \defeq
 \int_{\mathcal S_n^\circ}
 \mathrm e^{-V_n(x)}\mathcal L_{A_n}(\dd x),
 \qquad
 \mu_n(\dd x)
 =
 \frac{\1_{\mathcal S_n^\circ}(x)\mathrm e^{-V_n(x)}}
 {Z_n^{\mathrm{gap}}}\mathcal L_{A_n}(\dd x).
\end{equation}
Thus \(\mu_n\) is the Gibbs law with potential \(V_n\) on the cyclic-gap
simplex.  Uniform rooting introduces no additional factor depending on the gap lengths.
Since
\begin{equation}\label{EqCircularPairCurvature}
 \mathsf{g}_n''(r)
 =
 \left(\frac{\pi}{n}\right)^2
 \csc^2\left(\frac{\pi r}{n}\right)
 \geq\frac1{r^2},
 \qquad 0<r<n,
\end{equation}
the function \(V_n\) is convex.  Extend it by \(+\infty\) from
\(\mathcal S_n^\circ\) to \(A_n\).  The resulting function is proper,
lower semicontinuous, and convex.  Indeed, if \(x_i\to0\), the two
nearest-neighbour terms corresponding to that gap have total contribution
\[
 \beta \mathsf{g}_n(x_i)=-\beta\log(x_i)+O_{n,\beta}(1),
\]
whereas every other circular pair term is bounded below by
\(-\log(2)\).  Therefore
\[
 V_n(x)\longrightarrow+\infty
 \qquad\text{as }x\in\mathcal S_n^\circ
 \text{ approaches }\partial\mathcal S_n^\circ.
\]
This proves lower semicontinuity of the \(+\infty\)-extension and also shows
that
\[
 M\defeq\inf_{\mathcal S_n^\circ}V_n\in\R.
\]

We use the following finite-potential transport theorem
\cite[Proposition~1.1]{CETransport}; see also
\cite[Section~4]{BLTransport}.  Let \(E\) be a finite-dimensional Euclidean
space, let \(V:E\to\R\) be finite, differentiable, and convex, and put
\[
 \mu_V(\dd x)=Z_V^{-1}\mathrm e^{-V(x)}\dd x.
\]
Under the moment condition
\[
 \int_E
 \left(1+|x|^2+|\nabla V(x)|^2\right)
 \mathrm e^{-V(x)}\dd x<\infty,
\]
the theorem asserts that, for every probability measure
\(\nu\ll\mu_V\) with finite relative entropy,
\[
 \inf_{\pi\in\Pi(\mu_V,\nu)}
 \int_{E^2}
 \left[
 V(y)-V(x)-\nabla V(x)\mathbin{\cdot}(y-x)
 \right]\pi(\dd x,\dd y)
 \leq\Ent(\nu\mid\mu_V).
\]
Here \(\Pi(\mu_V,\nu)\) denotes the set of couplings of \(\mu_V\) and
\(\nu\).

This theorem does not apply directly to \(V_n\), because its
lower-semicontinuous extension equals \(+\infty\) outside
\(\mathcal S_n^\circ\) and diverges at its boundary.  We therefore
approximate \(V_n\) by finite differentiable convex potentials, apply the
theorem to those approximations, and then pass to the limit.

Identify \(A_n\) isometrically with \(\R^{n-1}\), and write \(\dd z\) for
the resulting Lebesgue measure.  Below, \(\nabla_{A_n}\) denotes the gradient
in this affine space.  Put
\[
 \Omega=\mathcal S_n^\circ,
\]
and let \(\overline V_n\) equal \(V_n\) on \(\Omega\) and \(+\infty\) on
\(A_n\setminus\Omega\).  For \(m\geq1\), define the Moreau envelope
\[
 V_n^{(m)}(z)
 \defeq
 \inf_{w\in A_n}
 \left\{
 \overline V_n(w)+\frac m2|z-w|^2
 \right\}.
\]
Define its normalising constant and Gibbs law by
\[
 Z_m\defeq\int_{A_n}\mathrm e^{-V_n^{(m)}(z)}\dd z,
 \qquad
 \mu_n^{(m)}(\dd z)
 \defeq
 \frac{\mathrm e^{-V_n^{(m)}(z)}}{Z_m}\dd z.
\]
Also put
\[
 Z\defeq\int_\Omega\mathrm e^{-V_n(x)}\dd x.
\]
By \eqref{EqCircularGapDensity}, \(Z\) is the normalising constant of
\(\mu_n\), up to the fixed constant arising from the chosen isometric
identification of \(A_n\).

Because \(\overline V_n\) is proper, lower semicontinuous, and convex with
bounded effective domain, \(V_n^{(m)}\) is finite, continuously
differentiable, and convex.  Moreover,
\[
 V_n^{(m)}(z)
 \geq
 M+\frac m2\operatorname{dist}(z,\Omega)^2,
\]
so \(V_n^{(m)}\) is quadratically coercive.  Standard properties of Moreau
envelopes give
\[
 V_n^{(m)}\uparrow\overline V_n,
 \qquad
 V_n^{(m)}\longrightarrow V_n,
 \qquad
 \nabla_{A_n}V_n^{(m)}
 \longrightarrow\nabla_{A_n}V_n
\]
locally uniformly on \(\Omega\).

For completeness, let \(p_m(z)\) be the proximal minimiser in the definition
of \(V_n^{(m)}\).  For every compact set \(K\Subset\Omega\),
\[
 \sup_{z\in K}|p_m(z)-z|^2
 \leq
 \frac2m\sup_{z\in K}\bigl[V_n(z)-M\bigr]
 \longrightarrow0.
\]
For sufficiently large \(m\), one has \(p_m(K)\Subset\Omega\), and the
first-order condition gives
\[
 \nabla_{A_n}V_n^{(m)}(z)
 =
 \nabla_{A_n}V_n(p_m(z)).
\]
Uniform continuity of \(\nabla_{A_n}V_n\) on a compact neighbourhood of
\(K\) proves the asserted local uniform convergence of the gradients.

Since
\[
 \mathrm e^{-V_n^{(m)}}\leq\mathrm e^{-V_n^{(1)}}
\]
and the right-hand side is integrable, dominated convergence gives
\begin{equation}\label{EqMoreauReferenceConvergence}
 Z_m\longrightarrow Z,
 \qquad
 \mu_n^{(m)}\longrightarrow\mu_n
 \quad\text{in total variation}.
\end{equation}

Put
\[
 H\defeq\Ent(\rho\mid\mu_n)<\infty.
\]
For every \(a\in(0,1)\),
\[
 \begin{aligned}
 \int_\Omega\mathrm e^{a(V_n(x)-M)}\mu_n(\dd x)
 &=
 \frac{\mathrm e^{-aM}}{Z}
 \int_\Omega\mathrm e^{-(1-a)V_n(x)}\dd x<\infty,
 \end{aligned}
\]
because \(\Omega\) is bounded and \(V_n\geq M\).  Apply the entropy
variational inequality to
\[
 a(V_n-M)\wedge R
\]
and then let \(R\to\infty\).  This gives
\[
 a\int_\Omega(V_n(x)-M)\rho(\dd x)
 \leq
 H+\log\left(
 \int_\Omega
 \mathrm e^{a(V_n(x)-M)}\mu_n(\dd x)
 \right)
 <\infty.
\]
Thus \(V_n\) is \(\rho\)-integrable.  Since, on \(\Omega\),
\[
 M\leq V_n^{(m)}\leq V_n,
\]
dominated convergence and \eqref{EqMoreauReferenceConvergence} imply
\begin{align}
 \Ent(\rho\mid\mu_n^{(m)})
 &=
 H+\log\left(\frac{Z_m}{Z}\right)
 +\int_\Omega
 \bigl(V_n^{(m)}(x)-V_n(x)\bigr)\rho(\dd x)
 \longrightarrow H.
\label{EqMoreauEntropyConvergence}
\end{align}
We now verify the moment hypothesis of this theorem for
\(V_n^{(m)}\).  The proximal representation gives
\[
 \nabla_{A_n}V_n^{(m)}(z)=m(z-p_m(z)).
\]
Since \(p_m(z)\in\Omega\), the set \(\Omega\) is bounded, and
\[
 V_n^{(m)}(z)
 \geq
 M+\frac m2\operatorname{dist}(z,\Omega)^2,
\]
we have
\[
 \int_{A_n}
 \left(
 1+|z|^2+
 |\nabla_{A_n}V_n^{(m)}(z)|^2
 \right)
 \mathrm e^{-V_n^{(m)}(z)}\dd z
 <\infty.
\]
Thus this theorem applies to \(V_n^{(m)}\).  Moreover,
\(\rho\ll\mu_n\) implies \(\rho\ll\mu_n^{(m)}\), and
\eqref{EqMoreauEntropyConvergence} shows that
\[
 \Ent(\rho\mid\mu_n^{(m)})<\infty
\]
for all sufficiently large \(m\).  We may therefore choose a coupling
\(\pi_m\in\Pi(\mu_n^{(m)},\rho)\) such that, with
\[
 c_m(x,y)
 \defeq
 V_n^{(m)}(y)-V_n^{(m)}(x)
 -\nabla_{A_n}V_n^{(m)}(x)\mathbin{\cdot}(y-x),
\]
one has
\begin{equation}\label{EqMoreauTransport}
 \int_{A_n^2}c_m(x,y)\pi_m(\dd x,\dd y)
 \leq
 \Ent(\rho\mid\mu_n^{(m)})+m^{-1}.
\end{equation}

The first marginals \(\mu_n^{(m)}\) converge to \(\mu_n\) in total
variation, and the second marginal is always \(\rho\).  Hence the sequence
\((\pi_m)\) is tight.  Passing to a subsequence, not relabelled,
\[
 \pi_m\overset{\mathrm w}{\longrightarrow}\pi
\]
for some coupling \(\pi\in\Pi(\mu_n,\rho)\).  Define, for
\(x,y\in\Omega\),
\[
 c(x,y)
 \defeq
 V_n(y)-V_n(x)
 -\nabla_{A_n}V_n(x)\mathbin{\cdot}(y-x).
\]
Choose functions
\[
 0\leq\chi_j\uparrow1,
 \qquad
 \chi_j\in C_c(\Omega^2),
\]
and extend them by zero to \(A_n^2\).  The local uniform convergence of the
potentials and their gradients gives, for each fixed \(j\),
\[
 \begin{aligned}
 \int_{\Omega^2}\chi_j(x,y)c(x,y)\pi(\dd x,\dd y)
 &=
 \lim_{m\to\infty}
 \int_{A_n^2}\chi_j(x,y)c_m(x,y)\pi_m(\dd x,\dd y)\leq H.
 \end{aligned}
\]
All the Bregman costs are non-negative by convexity.  Letting \(j\to\infty\)
and applying monotone convergence therefore yields
\begin{equation}\label{EqAboveTangentApplication}
 \E_\pi\left[
 V_n(\mathsf Y)-V_n(\mathsf X)
 -\nabla_{A_n}V_n(\mathsf X)
  \mathbin{\cdot}(\mathsf Y-\mathsf X)
 \right]
 \leq\Ent(\rho\mid\mu_n).
\end{equation}
Notice that this approximation avoids requiring the boundary integrability
of
\[
 |\nabla_{A_n}V_n|^2\mathrm e^{-V_n},
\]
and therefore works for every \(\beta>0\).

It remains to convert the Bregman cost of \(V_n\) into a cost on the
individual gaps.  For \(0<x,y<n\), integrating
\eqref{EqCircularPairCurvature} along the line segment from \(x\) to \(y\)
gives
\begin{equation}\label{EqScalarBregmanGap}
 \mathsf{g}_n(y)-\mathsf{g}_n(x)-\mathsf{g}_n'(x)(y-x)
 \geq
 \frac{y}{x}-1-\log\left(\frac{y}{x}\right)
 =
 \psi\left(\frac{y}{x}\right).
\end{equation}
For each \(i\), the terms labelled \((i,1)\) and \((i+1,n-1)\) in
\eqref{EqCircularGapPotential} correspond, respectively, to \(x_i\) and
\(n-x_i\).  Their total coefficient in \(V_n\) is \(\beta\).  For
\(n\geq3\), these are distinct terms.  Furthermore,
\[
 \mathsf{g}_n(n-r)=\mathsf{g}_n(r),
 \qquad
 \mathsf{g}_n'(n-r)=-\mathsf{g}_n'(r),
\]
so their Bregman divergences, viewed as functions of the gap coordinate
\(r\), agree.  Every remaining term in
\eqref{EqCircularGapPotential} has non-negative Bregman divergence by
convexity.  It follows that
\[
 \begin{aligned}
 &V_n(\mathsf Y)-V_n(\mathsf X)
 -\nabla_{A_n}V_n(\mathsf X)
  \mathbin{\cdot}(\mathsf Y-\mathsf X)\geq
 \beta\sum_{i=1}^n
 \psi\left(\frac{\mathsf Y_i}{\mathsf X_i}\right).
 \end{aligned}
\]
Combining this inequality with
\eqref{EqAboveTangentApplication} proves
\eqref{EqCyclicGapTransport}.

Finally, suppose that \(\rho\) is invariant under the cyclic shift
\(\sigma_n\).  The reference law \(\mu_n\) is also cyclically invariant.
Average \(\pi\) under the \(n\) transformations
\[
 (\mathsf X,\mathsf Y)
 \longmapsto
 (\sigma_n^j\mathsf X,\sigma_n^j\mathsf Y),
 \qquad 0\leq j\leq n-1.
\]
The resulting coupling has the same marginals, is invariant under
simultaneous cyclic shifts, and satisfies the same bound
\eqref{EqCyclicGapTransport}.  Its \(n\) coordinate expectations are equal,
so
\[
 \E_\pi\left[
 \psi\left(\frac{\mathsf Y_i}{\mathsf X_i}\right)
 \right]
 =
 \frac1n
 \E_\pi\left[
 \sum_{j=1}^n
 \psi\left(\frac{\mathsf Y_j}{\mathsf X_j}\right)
 \right]
 \leq
 \frac{\Ent(\rho\mid\mu_n)}{\beta n}.
\]
This proves \eqref{EqOneGapTransport}.
\end{proof}

For each \(n\geq3\) with \(p_{L,n}>0\), apply
Lemma~\ref{LemCyclicGapTransport} with \(\rho=\rho_{L,n}\), and symmetrise
the resulting coupling.  Denote its law by \(\pi_{L,n}\) and write it
 as \((\mathsf X^{L,n},\mathsf Y^{L,n})\).  When \(p_{L,n}=0\), set instead
\(\pi_{L,n}=(\operatorname{Id},\operatorname{Id})_*\mu_n\); this convention
has no effect on any weighted sum.  Equations
\eqref{EqFixedSectorEntropy}, \eqref{EqGapEntropyContraction}, and
\eqref{EqOneGapTransport} give
\begin{equation}\label{EqGoodSectorOneGap}
 \sum_{\substack{n\geq3\\L/2\leq n\leq2L}}
 p_{L,n}
 \E_{\pi_{L,n}}\left[
 \psi\left(\frac{\mathsf Y_1^{L,n}}{\mathsf X_1^{L,n}}\right)\right]
 \leq
 \frac{2\mathcal H_L^{\mathrm{circ}}}{\beta L}
 =o(1).
\end{equation}
Because each \(\pi_{L,n}\) is cyclically invariant, the same bound holds
after replacing coordinate \(1\), separately in each sector, by any choice
\(i_n\in\{1,\ldots,n\}\).

Observe that, the reference gaps have mean one:
\begin{equation}\label{EqReferenceGapMean}
 \E_{\pi_{L,n}}[\mathsf X_i^{L,n}]=1.
\end{equation}
Indeed, their sum is \(n\), and \(\mu_n\) is cyclically invariant.
For \(\varepsilon>0\) and \(M>1\),
\begin{align}
 \pi_{L,n}\left(
 |\mathsf Y_i^{L,n}-\mathsf X_i^{L,n}|>\varepsilon\right)
 &\leq
 \pi_{L,n}(\mathsf X_i^{L,n}>M)+
 \pi_{L,n}\left(
 \left|\frac{\mathsf Y_i^{L,n}}{\mathsf X_i^{L,n}}-1\right|
 >\frac{\varepsilon}{M}\right)\notag\\
 &\leq
 \frac1M+
 \frac{
 \E_{\pi_{L,n}}[\psi(\mathsf Y_i^{L,n}/\mathsf X_i^{L,n})]
 }{
 \displaystyle
 \inf_{\{q>0:\,|q-1|\geq\varepsilon/M\}}\psi(q)
 }.
\label{EqGapDifferenceFromPsi}
\end{align}
First let \(L\to\infty\) in the weighted version of this estimate, and
then let \(M\to\infty\).  For every fixed \(k\), all the displayed
coordinates exist once \(L\) is sufficiently large.  A union bound then
shows that
\begin{equation}\label{EqFixedGapBlock}
 \sum_{\substack{n\geq3\\L/2\leq n\leq2L}}
 p_{L,n}\,
 \pi_{L,n}\left(
 \max_{1\leq i\leq k}
 |\mathsf Y_i^{L,n}-\mathsf X_i^{L,n}|
 \mathbin{\vee}
 \max_{n-k+1\leq j\leq n}
 |\mathsf Y_j^{L,n}-\mathsf X_j^{L,n}|>\varepsilon
 \right)
 \longrightarrow0.
\end{equation}

\subsection{The Palm limit of the circular ensembles}

Recall from Section~\ref{SectionCircularRecovery} that \(\bP_n\) is the
stationary periodic lift of \(\bQ_{n,\beta}\).  Lemma
\ref{LemSineFiniteEnergy}, specifically
\eqref{EqCircularLocalAndPalmLimit}, gives the unrooted and reduced-Palm
limits simultaneously.

We next identify this Palm law with the uniformly rooted cyclic gap
 law.  Extend \(x=(x_1,\ldots,x_n)\in\mathcal S_n^\circ\) periodically
 by \(x_{i+n}=x_i\), and put
\[
 s_0(x)=0,\qquad
 s_k(x)=\sum_{i=1}^kx_i\quad(k\geq1),\qquad
 s_{-k}(x)=-\sum_{i=0}^{k-1}x_{-i}\quad(k\geq1).
\]
Define the periodised reduced rooted configuration
\begin{equation}\label{EqRootedGapConfiguration}
 \mathsf{Root}_n(x)
 \defeq
 \sum_{k\in\mathbb Z\setminus\{0\}}\delta_{s_k(x)}.
\end{equation}
For the local vague topology, \(\mathsf{Root}_n\) is continuous on
\(\mathcal S_n^\circ\), hence Borel: on each compact window it is locally a
 finite sum of atoms whose partial-sum locations depend continuously on \(x\).
The atoms at nonzero multiples of \(n\) are the other periodic copies
of the selected root.  Rotation invariance and the Campbell formula on one
period give, for every bounded Borel \(F\),
\[
 \bP_n^{!0}(F)
 =\frac1n\int_{\Conf_n(I_n)}
   \sum_{j=1}^n
   F\!\left(\mathsf{Root}_n(\mathsf{Gap}_{n,j}(\bzeta))\right)
   \bQ_{n,\beta}(\dd\bzeta)
 =(\mathsf{Root}_n)_*\mu_n(F).
\]
Thus
\begin{equation}\label{EqReferenceGapPalmIdentity}
 (\mathsf{Root}_n)_*\mu_n=\bP_n^{!0}.
\end{equation}

Define the cyclically rooted candidate law
\begin{equation}\label{EqCyclicCandidateRootLaw}
 \widehat{\mathfrak R}_L
 \defeq
 p_{L,0}\delta_\varnothing+
 \sum_{n\geq1}p_{L,n}(\mathsf{Root}_n)_*\rho_{L,n}.
\end{equation}

\begin{prop}\label{PropCyclicCandidateLimit}
The cyclically rooted candidate laws from \eqref{EqCyclicCandidateRootLaw} satisfy
\begin{equation}\label{EqCyclicCandidateLimit}
 \widehat{\mathfrak R}_L
 \overset{\mathrm{w}}{\longrightarrow}
 (\Sine_\beta)^{!0}.
\end{equation}
\end{prop}

\begin{proof}
Chebyshev's inequality and Proposition~\ref{PropInputs}(c) give
\begin{equation}\label{EqRootGoodSector}
 \sum_{\{n<L/2\text{ or }n>2L\}}p_{L,n}
 \leq\frac{4v(L)}{L^2}=o(1).
\end{equation}
In particular, the sectors \(n=0,1,2\) have vanishing total
probability.

Put
\[
 \mathfrak q_L^{\mathrm{good}}\defeq
 \sum_{\substack{n\geq3\\L/2\leq n\leq2L}}p_{L,n}.
\]
Then \(\mathfrak q_L^{\mathrm{good}}\to1\), so
\(\mathfrak q_L^{\mathrm{good}}>0\) for all sufficiently large \(L\).  On those
\(L\), define the probability coupling
\[
 \pi_L^{\mathrm{good}}
 \defeq\frac1{\mathfrak q_L^{\mathrm{good}}}
 \sum_{\substack{n\geq3\\L/2\leq n\leq2L}}
 p_{L,n}
 (\mathsf{Root}_n,\mathsf{Root}_n)_*\pi_{L,n}.
\]
We claim that the distance between its two coordinates in the local vague
topology converges to zero in \(\pi_L^{\mathrm{good}}\)-probability.

 To see this without using discontinuous hard restrictions, choose
 \(\chi_m\in C_c(\R)\), with \(0\leq\chi_m\leq1\), equal to one on
 \([-m,m]\) and supported in \((-m-1,m+1)\).  For a locally finite
 measure \(\bgamma\), set
 \[
  (\chi_m\bgamma)(A)\defeq\int_A\chi_m(x)\bgamma(\dd x).
 \]
 For finite Radon measures, let
 \[
  \mathsf d_{\mathrm{fm}}(\bgamma,\betaeta)
  \defeq
  \sup_{\substack{\|f\|_\infty\leq1\\
                   \operatorname{Lip}(f)\leq1}}
  \left|\int_\R f(x)\bgamma(\dd x)
        -\int_\R f(x)\betaeta(\dd x)\right|.
 \]
 Now put
\[
 \mathsf d_{\mathrm{loc}}(\bgamma,\betaeta)
 \defeq
 \sum_{m\geq1}2^{-m}\left(
  1\wedge
  \mathsf d_{\mathrm{fm}}(\chi_m\bgamma,\chi_m\betaeta)
 \right).
\]
 This metric \(\mathsf d_{\mathrm{loc}}\) generates the local vague topology.

Fix \(m\).  Lemma~\ref{LemSineFiniteEnergy} and
\eqref{EqReferenceGapPalmIdentity} imply tightness of the first marginals
of \(\pi_L^{\mathrm{good}}\).  Hence, given \(\varepsilon>0\), one may choose
\(k\) so that, for all large \(L\), the first coordinate has at most \(k\)
atoms in \([-m-2,m+2]\), except on an event of probability at most
\(\varepsilon\).  On its complement, use \eqref{EqFixedGapBlock} with
\(k+1\).  If the first \(k+1\) gaps on both sides differ by at most
\(\delta/(k+1)\), corresponding atoms differ in position by at most
\(\delta<1\), and the \((k+1)\)-st target atom on each side remains outside
\((-m-1,m+1)\).  Thus every target atom in the support of \(\chi_m\) is
paired with one of the first \(k\) reference atoms.  Positivity and ordering
of the gaps place every subsequent target atom still farther outside; every
reference atom meeting the support has one of the paired ranks.  Writing \(\omega_m\)
for the modulus of continuity of
\(\chi_m\), every test function \(f\) with
\(\|f\|_\infty\leq1\) and \(\operatorname{Lip}(f)\leq1\) then satisfies
\[
 \left|
  \int_\R f(x)\chi_m(x)\dd\mathsf{Root}_n(\mathsf X^{L,n})(x)
  -\int_\R f(x)\chi_m(x)\dd\mathsf{Root}_n(\mathsf Y^{L,n})(x)
 \right|
 \leq C(k+1)[\delta+\omega_m(\delta)].
\]
The modulus term covers a paired atom that crosses the edge of
\(\operatorname{supp}(\chi_m)\).  The block-failure probability under
\(\pi_L^{\mathrm{good}}\) is the weighted sum in
\eqref{EqFixedGapBlock} divided by \(\mathfrak q_L^{\mathrm{good}}\); it
tends to zero because \(\mathfrak q_L^{\mathrm{good}}\to1\).  Therefore the
\(m\)-th finite-measure distance
converges to zero in \(\pi_L^{\mathrm{good}}\)-probability after first taking
\(L\to\infty\) and then \(\delta\downarrow0\).  For fixed \(M\), apply this
to \(m\leq M\); the remaining tail in \(\mathsf d_{\mathrm{loc}}\) is at most
\(2^{-M}\).  Finally let \(M\to\infty\) to prove the claim.

Let \(\mathsf d_{\mathrm{BL}}\) denote the bounded-Lipschitz metric on
probability laws induced by \(\mathsf d_{\mathrm{loc}}\).  Since
\(\mathsf d_{\mathrm{loc}}\leq1\), this convergence in probability
also gives convergence in mean.  Hence the coupling inequality yields
\[
 \mathsf d_{\mathrm{BL}}\!\left(
  (\operatorname{pr}_1)_*\pi_L^{\mathrm{good}},
  (\operatorname{pr}_2)_*\pi_L^{\mathrm{good}}
 \right)
 \leq
 \int\mathsf d_{\mathrm{loc}}(\bgamma,\betaeta)\,
 \pi_L^{\mathrm{good}}(\dd\bgamma,\dd\betaeta)
 \longrightarrow0.
\]

Since \(n\geq L/2\) on these sectors,
\eqref{EqCircularLocalAndPalmLimit} gives
\[
 \sup_{n\geq L/2}
 \mathsf d_{\mathrm{BL}}\bigl(\bP_n^{!0},(\Sine_\beta)^{!0}\bigr)
 \longrightarrow0.
 \]
Together
with \(\mathfrak q_L^{\mathrm{good}}\to1\) and
\eqref{EqRootGoodSector}, this also restores the
discarded sectors and proves
\eqref{EqCyclicCandidateLimit}.
\end{proof}

\subsection{Removing the cyclic seam and the random dilation}

The preceding law uses an artificial periodic continuation across the
cut of \(I_L\).  We now compare it to the ordinary uniformly rooted
restriction, namely the probability law \(\mathfrak R_L\) defined by
\eqref{EqActualRootLaw}.  

\begin{lem}\label{LemRemoveCyclicSeam}
For every bounded local continuous function \(F\),
\begin{equation}\label{EqRemoveCyclicSeam}
 \widehat{\mathfrak R}_L(F)-\mathfrak R_L(F)
 \longrightarrow0.
\end{equation}
\end{lem}

\begin{proof}
We construct an explicit coupling of the two rooted laws.  Sample
\(\bC\) under \(\bP\), put
\[
 n=\sN_L,
 \qquad
 \bzeta=\bZ_L=(T_{n/L})_*\bC_{I_L}
\]
on \(\{n\geq1\}\), and set \(\bzeta=\varnothing\) when \(n=0\).
Conditionally on \(n\geq1\), the configuration \(\bzeta\) is almost surely
simple and has no point at \(0\) or \(n\).  Write
\[
 \bzeta=\sum_{i=1}^n\delta_{u_i},
 \qquad
 0<u_1<\cdots<u_n<n,
\]
and, conditionally on \(\bzeta\), choose an index \(J\) uniformly from
\(\{1,\ldots,n\}\).

Define the ordinary reduced rooted configuration by
\[
 \Gamma_L^{\mathrm{ord}}
 \defeq
 \sum_{\substack{1\leq i\leq n\\i\neq J}}
 \delta_{u_i-u_J},
\]
and define its periodised version by
\[
 \Gamma_L^{\mathrm{cyc}}
 \defeq
 \sum_{\substack{q\in\mathbb Z,\ 1\leq i\leq n\\
                  (q,i)\neq(0,J)}}
 \delta_{u_i-u_J+qn}.
\]
On \(\{n=0\}\), set both configurations equal to \(\varnothing\).
Let \(\mathbb P_L^{\mathrm{cpl}}\) and
\(\mathbb E_L^{\mathrm{cpl}}\) denote probability and expectation under
this auxiliary construction.

On \(\{n\geq1\}\), the original position of the selected particle is
\((L/n)u_J\), and therefore
\[
 \Gamma_L^{\mathrm{ord}}
 =
 (T_{n/L})_*
 \theta_{(L/n)u_J}
 \left(
  \bC_{I_L}-\delta_{(L/n)u_J}
 \right).
\]
Moreover, by the definitions of the cyclic gaps and their periodised rooted
configuration,
\[
 \Gamma_L^{\mathrm{cyc}}
 =
 \mathsf{Root}_n\bigl(\mathsf{Gap}_{n,J}(\bzeta)\bigr).
\]
Since, conditionally on \(\{\sN_L=n\}\), the law of \(\bzeta\) is
\(\overline\nu_{L,n}\) and \(J\) is uniform on
\(\{1,\ldots,n\}\), it follows that
\begin{equation}\label{EqSeamCouplingMarginals}
 \operatorname{Law}_{\mathbb P_L^{\mathrm{cpl}}}
 (\Gamma_L^{\mathrm{cyc}})
 =\widehat{\mathfrak R}_L,
 \qquad
 \operatorname{Law}_{\mathbb P_L^{\mathrm{cpl}}}
 (\Gamma_L^{\mathrm{ord}})
 =\mathfrak R_L.
\end{equation}

Let \(F\) depend only on the restriction to a compact interval contained in
\([-R,R]\), and fix \(\varepsilon>0\).
Proposition~\ref{PropCyclicCandidateLimit} implies that
\((\widehat{\mathfrak R}_L)\) is tight.  Hence there are
\(L_0<\infty\) and a compact set
\(\mathcal K_\varepsilon\subset\Conf(\R)\) such that
\[
 \inf_{L\geq L_0}
 \widehat{\mathfrak R}_L(\mathcal K_\varepsilon)
 \geq1-\varepsilon.
\]
Compactness in the vague topology implies
\[
 \sup_{\bgamma\in\mathcal K_\varepsilon}
 \bgamma([-R,R])<\infty.
\]
Indeed, choose \(\varphi\in C_c(\R)\) with
\(\varphi\geq\1_{[-R,R]}\).  The map
\[
 \bgamma\longmapsto\int_\R\varphi(x)\bgamma(\dd x)
\]
is continuous and is therefore bounded on
\(\mathcal K_\varepsilon\).  Consequently, there exists \(k\in\N\) such
that
\begin{equation}\label{EqCyclicWindowCount}
 \sup_{L\geq L_0}
 \mathbb P_L^{\mathrm{cpl}}\left(
 \Gamma_L^{\mathrm{cyc}}([-R,R])>k
 \right)
 \leq\varepsilon.
\end{equation}

For \(n\geq1\), introduce the event that the selected root lies within
\(k\) particle ranks of the cut between \(u_n\) and \(u_1\):
\[
 \mathcal B_{n,k}
 \defeq
 \{J\leq k+1\}\cup\{J\geq n-k\}.
\]
Since \(J\) is conditionally uniform,
\begin{equation}\label{EqRootNearSeam}
 \mathbb P_L^{\mathrm{cpl}}
 \left(\mathcal B_{n,k}\mid\bzeta\right)
 \leq\frac{2(k+1)}{n}.
\end{equation}

We claim that, for all sufficiently large \(L\), on the event
\[
 \mathcal G_{L,R,k}
 \defeq
 \left\{\frac L2\leq n\leq2L\right\}
 \cap\mathcal B_{n,k}^{\,c}
 \cap
 \left\{\Gamma_L^{\mathrm{cyc}}([-R,R])\leq k\right\},
\]
one has
\begin{equation}\label{EqSeamLocalAgreement}
 (\Gamma_L^{\mathrm{cyc}})_{[-R,R]}
 =
 (\Gamma_L^{\mathrm{ord}})_{[-R,R]}.
\end{equation}

To prove the claim, take \(L\) sufficiently large that \(L/2>2R\).
On \(\{L/2\leq n\leq2L\}\), this gives \(n>2R\).
Therefore an atom of
\(\Gamma_L^{\mathrm{cyc}}\setminus\Gamma_L^{\mathrm{ord}}\) lying in
\([-R,R]\) can only have one of the following two forms:
\begin{equation}\label{EqPossibleWrappedAtoms}
 u_i-u_J+n\in(0,R],
 \quad i<J,
 \qquad\text{or}\qquad
 u_i-u_J-n\in[-R,0),
 \quad i>J.
\end{equation}

Suppose first that
\[
 u_i-u_J+n\in(0,R],
 \qquad i<J.
\]
Every intervening clockwise atom
\[
 u_{J+1}-u_J,\ldots,u_n-u_J,\,
 u_1-u_J+n,\ldots,u_i-u_J+n
\]
then also belongs to \((0,R]\).  The number of these atoms is
\[
 n-J+i\geq n-J+1.
\]
On \(\mathcal B_{n,k}^{\,c}\), one has \(J<n-k\), and hence
\[
 n-J+i\geq k+2.
\]
It follows that
\[
 \Gamma_L^{\mathrm{cyc}}([-R,R])\geq k+2,
\]
contrary to the definition of \(\mathcal G_{L,R,k}\).

Similarly, suppose that
\[
 u_i-u_J-n\in[-R,0),
 \qquad i>J.
\]
Every intervening counterclockwise atom also belongs to \([-R,0)\), and
their number is
\[
 J+n-i\geq J.
\]
On \(\mathcal B_{n,k}^{\,c}\), one has \(J>k+1\), so again
\[
 \Gamma_L^{\mathrm{cyc}}([-R,R])\geq k+2,
\]
which is a contradiction.  Thus no wrapped atom can enter \([-R,R]\) on
\(\mathcal G_{L,R,k}\), proving \eqref{EqSeamLocalAgreement}.

A union bound, \eqref{EqRootNearSeam}, and
\eqref{EqRootGoodSector} now give
\[
 \begin{aligned}
 &\mathbb P_L^{\mathrm{cpl}}\left(
 (\Gamma_L^{\mathrm{cyc}})_{[-R,R]}
 \neq
 (\Gamma_L^{\mathrm{ord}})_{[-R,R]}
 \right)\\
 &\quad\leq
 \sum_{\{n<L/2\text{ or }n>2L\}}p_{L,n}
 +
 \mathbb P_L^{\mathrm{cpl}}\left(
  \Gamma_L^{\mathrm{cyc}}([-R,R])>k
 \right)+
 \sum_{\substack{n\geq1\\L/2\leq n\leq2L}}
 p_{L,n}\frac{2(k+1)}{n}\\
 &\quad\leq
 \frac{4v(L)}{L^2}
 +\varepsilon
 +\frac{4(k+1)}{L}.
 \end{aligned}
\]
Consequently,
\[
 \limsup_{L\to\infty}
 \mathbb P_L^{\mathrm{cpl}}\left(
 (\Gamma_L^{\mathrm{cyc}})_{[-R,R]}
 \neq
 (\Gamma_L^{\mathrm{ord}})_{[-R,R]}
 \right)
 \leq\varepsilon.
\]
Since \(\varepsilon>0\) is arbitrary,
\begin{equation}\label{EqSeamMismatchProbability}
 \mathbb P_L^{\mathrm{cpl}}\left(
 (\Gamma_L^{\mathrm{cyc}})_{[-R,R]}
 \neq
 (\Gamma_L^{\mathrm{ord}})_{[-R,R]}
 \right)
 \longrightarrow0.
\end{equation}

Finally, locality of \(F\), \eqref{EqSeamCouplingMarginals}, and
\eqref{EqSeamMismatchProbability} yield
\[
 \begin{aligned}
 \left|
 \widehat{\mathfrak R}_L(F)-\mathfrak R_L(F)
 \right|
 &=
 \left|
 \mathbb E_L^{\mathrm{cpl}}\left[
 F(\Gamma_L^{\mathrm{cyc}})
 -F(\Gamma_L^{\mathrm{ord}})
 \right]\right|\\
 &\leq
 2\lVert F\rVert_\infty
 \mathbb P_L^{\mathrm{cpl}}\left(
 (\Gamma_L^{\mathrm{cyc}})_{[-R,R]}
 \neq
 (\Gamma_L^{\mathrm{ord}})_{[-R,R]}
 \right)\longrightarrow0.
 \end{aligned}
\]
This proves \eqref{EqRemoveCyclicSeam}.
\end{proof}
\begin{prop}\label{PropCandidatePalmLimit}
The ordinary rooted laws satisfy
\begin{equation}\label{EqCandidatePalmLimit}
 \mathfrak R_L\overset{\mathrm{w}}{\longrightarrow}\bP^{!0}.
\end{equation}
\end{prop}

\begin{proof}
By Lemma~\ref{LemSimplicity}, the reduced rooted configurations below are
well defined \(\bP\)-almost surely.
Let \(F:\Conf(\R)\to\R\) be bounded, local, and continuous.  Replacing
the random factor \(1/\sN_L\) by \(1/L\) in
\eqref{EqActualRootLaw} gives
\begin{align}
 &\left|
 \mathfrak R_L(F)
 -
 \frac1L\E_\bP\left[
  \1_{\{\sN_L>0\}}
  \sum_{x\in\bC\cap I_L}
 F\left(
 (T_{\sN_L/L})_*
 \theta_x(\bC_{I_L}-\delta_x)
 \right)\right]\right|\notag\\
 &\qquad\leq
 \|F\|_\infty
 \E_\bP\left[\left|\frac{\sN_L}{L}-1\right|\right]
 \leq
 \|F\|_\infty\frac{\sqrt{v(L)}}{L}
 \longrightarrow0.
\label{EqReplaceParticleNormalisation}
\end{align}
The empty sector is included correctly in this bound.

Introduce the probability measure
\begin{equation}\label{EqCampbellPairMeasure}
 \mathsf{Camp}_L(\dd\bC,\dd x)
 \defeq
 \frac1L\1_{I_L}(x)\bC(\dd x)\bP(\dd\bC).
\end{equation}
It is a probability measure because \(\bP\) has intensity one.
Under \(\mathsf{Camp}_L\), put
\[
 \mathsf U_L^{\mathrm{root}}=x,
\qquad
 \bC_L^{\mathrm{root}}
 =\theta_{\mathsf U_L^{\mathrm{root}}}
   (\bC-\delta_{\mathsf U_L^{\mathrm{root}}}),
 \qquad
 \mathsf{a}_L=\frac{\sN_L}{L}.
\]
The refined Campbell identity and stationarity imply that the
\(\bC_L^{\mathrm{root}}\)-marginal is exactly \(\bP^{!0}\), for every \(L\).

Moreover, for every \(\varepsilon>0\), one has
\begin{align}
 \mathsf{Camp}_L(|\sN_L-L|>\varepsilon L)
 &=\frac1L\E_\bP[\sN_L\1_{\{|\sN_L-L|>\varepsilon L\}}]\notag\\
 &\leq
 \bP(|\sN_L-L|>\varepsilon L)
 +\frac1L\E_\bP[|\sN_L-L|
   \1_{\{|\sN_L-L|>\varepsilon L\}}]\notag\\
 &\leq
 \frac{v(L)}{\varepsilon^2L^2}
 +\frac{v(L)}{\varepsilon L^2}
 \longrightarrow0.
\label{EqDilationUnderCampbellBias}
\end{align}
Thus \(\mathsf{a}_L\to1\) in \(\mathsf{Camp}_L\)-probability.  Every subsequential
limit of \((\mathsf{a}_L,\bC_L^{\mathrm{root}})\) has marginals \(\delta_1\) and \(\bP^{!0}\),
and hence equals \(\delta_1\otimes\bP^{!0}\).  Therefore
\[
 \operatorname{Law}_{\mathsf{Camp}_L}
   (\mathsf{a}_L,\bC_L^{\mathrm{root}})
 \overset{\mathrm{w}}{\longrightarrow}
 \delta_1\otimes\bP^{!0}.
\]
The map
\[
 (a,\betaeta)\longmapsto(T_a)_*\betaeta
\]
is continuous on \((0,\infty)\times\Conf(\R)\).  Indeed, if \(a_j\to a>0\),
then for every compactly supported continuous \(f\), the functions
\(f(a_j\,\cdot)\) have common compact support and converge uniformly to
\(f(a\,\cdot)\).  Consequently,
\begin{equation}\label{EqFullRootRandomDilation}
 \int_{\Conf(\R)\times\R}
 F((T_{\mathsf{a}_L})_*\bC_L^{\mathrm{root}})\,
 \mathsf{Camp}_L(\dd\bC,\dd x)
 \longrightarrow\bP^{!0}(F).
\end{equation}

It remains only to replace the full configuration by its restriction
to \(I_L\).  Suppose that \(F\) depends on the configuration in
\([-R,R]\).  On \(\{1/2\leq\mathsf{a}_L\leq2\}\), the restricted and full
rooted configurations agree in the window used by \(F\), unless
\[
 \mathsf U_L^{\mathrm{root}}\in[0,2R]\cup[L-2R,L].
\]
The \(\mathsf{Camp}_L\)-probability of this endpoint event is at most
\[
 \frac1L
 \E_\bP\left[
  \bC([0,2R]\cup[L-2R,L])\right]
 \leq\frac{4R}{L}.
\]
The complementary dilation event has probability tending to zero by
\eqref{EqDilationUnderCampbellBias}.  Therefore the restriction in
\eqref{EqReplaceParticleNormalisation} may be replaced by the full
configuration at an \(o(1)\) error.  Equations
\eqref{EqReplaceParticleNormalisation} and
\eqref{EqFullRootRandomDilation} prove the result for every bounded
local continuous \(F\), a convergence-determining class on
\(\Conf(\R)\); it contains the compactly supported Laplace functionals.
\end{proof}

Proposition~\ref{PropCyclicCandidateLimit} and
Lemma~\ref{LemRemoveCyclicSeam} give
\(\mathfrak R_L\overset{\mathrm{w}}{\longrightarrow}(\Sine_\beta)^{!0}\), while
Proposition~\ref{PropCandidatePalmLimit} gives
\(\mathfrak R_L\overset{\mathrm{w}}{\longrightarrow}\bP^{!0}\).  Together
these limits yield
\begin{equation}\label{EqPalmEquality}
 \bP^{!0}=(\Sine_\beta)^{!0}.
\end{equation}

\subsection{Palm inversion}

We finally recall the well-known fact that equality of reduced Palm laws determines the
stationary process.  Let \(\bR\) be a stationary simple point process of finite positive
intensity \(\lambda\), and let \(\bR^{!0}\) be its reduced Palm law.
For every \(\betaeta\in\Conf(\R)\), put
\[
 x_1(\betaeta)
 \defeq
 \inf\{r>0:\betaeta((0,r])\geq1\}\in[0,\infty],
 \qquad \inf\varnothing\defeq\infty.
\]
This is a Borel function and, on simple configurations possessing a point
to the right, is the location of the first such point.
The one-dimensional Palm inversion formula is
\begin{align}
 \bR(F)
 =
 \left[
 1-\lambda
 \int_{\Conf(\R)}x_1(\betaeta)\bR^{!0}(\dd\betaeta)
 \right]F(\varnothing)+
 \lambda
 \int_{\Conf(\R)}
 \int_0^{x_1(\betaeta)}
 F\bigl(\theta_t(\betaeta+\delta_0)\bigr)
 \dd t\,\bR^{!0}(\dd\betaeta)
\label{EqPalmInversion}
\end{align}
for every bounded Borel function \(F\).  The coefficient of
\(F(\varnothing)\) is the possible vacuum mass.

For completeness let us spell out the derivation of this identity, first let \(F\geq0\) be bounded.  On the non-vacuum
event, stationarity excludes a leftmost or rightmost point: either would be
a real-valued translation-covariant random variable with a
translation-invariant probability law.  Hence the half-open successor
cells partition \(\R\) almost surely.  For \(x\in\bgamma\), set
\[
 \tau_x(\bgamma)
 \defeq\inf\{r>0:\bgamma((x,x+r])\geq1\},
\]
so, on the marked incidence space,
\[
 \tau_x(\bgamma)=x_1\bigl(\theta_x(\bgamma-\delta_x)\bigr)
\]
is Borel.  For fixed \(T,M\), only atoms
\(x\in[-T-M,T]\) can contribute below; hence the following is a finite
Borel particle sum.  Define the truncated occupation average
\[
 A_{T,M}(\bgamma)
 \defeq\frac1{2T}\sum_{x\in\bgamma}
 \int_0^{\tau_x(\bgamma)\wedge M}
 \1_{[-T,T]}(x+t)F(\theta_{x+t}\bgamma)\dd t.
\]
The refined Campbell formula, followed by integration in the root location,
gives the exact identity
\[
 \E_{\bR}[A_{T,M}(\bC)]
 =\lambda
 \int_{\Conf(\R)}
 \int_0^{x_1(\betaeta)\wedge M}
 F\bigl(\theta_t(\betaeta+\delta_0)\bigr)
 \dd t\,\bR^{!0}(\dd\betaeta),
\]
because \((2T)^{-1}\int_\R\1_{[-T,T]}(x+t)\dd x=1\).
As \(M\to\infty\), the successor-cell partition gives
\[
 A_{T,M}(\bgamma)\uparrow
 \frac1{2T}\int_{[-T,T]}F(\theta_s\bgamma)\dd s
 \quad\text{for \(\bR\)-almost every \(\bgamma\) on }
 \{\bgamma\neq\varnothing\},
\]
and the left-hand side is zero on the vacuum configuration.  Stationarity
and monotone convergence therefore give
\begin{equation}\label{EqPalmInversionNonvacuum}
 \E_{\bR}[F(\bC);\,\bC\neq\varnothing]
 =
 \lambda
 \int_{\Conf(\R)}
 \int_0^{x_1(\betaeta)}
 F\bigl(\theta_t(\betaeta+\delta_0)\bigr)
 \dd t\,\bR^{!0}(\dd\betaeta).
\end{equation}
Taking \(F=1\) proves simultaneously that
\[
 x_1<\infty\quad\bR^{!0}\text{-almost surely},
 \qquad
 \lambda\int_{\Conf(\R)}x_1(\betaeta)\bR^{!0}(\dd\betaeta)
 =\bR(\bC\neq\varnothing),
\]
so the first coefficient is exactly \(\bR(\bC=\varnothing)\).
Adding the vacuum atom to \eqref{EqPalmInversionNonvacuum} proves
\eqref{EqPalmInversion} for non-negative bounded \(F\), and hence for every
bounded Borel \(F\) by decomposing into positive and negative parts.  In particular,
the intensity and the reduced Palm law determine the stationary
process, including its vacuum mass.

\subsection{Completion of the proof of Theorem~\ref{ThmMain}}

\begin{proof}[Proof of Theorem~\ref{ThmMain}]
Lemma~\ref{LemSineFiniteEnergy} gives
\(\cW^{\mathrm{div}}(\Sine_\beta)<\infty\).  Since the stationary periodic
lifts in \eqref{EqCircularLocalAndPalmLimit} converge weakly to
\(\Sine_\beta\), the latter is stationary.  Lemma~\ref{LemDHLMKernelBridge}
gives the canonical DLR equations.  Thus \(\Sine_\beta\) belongs to the class
in the theorem.

Conversely, let \(\bP\) belong to that class.  The preceding argument gives
\eqref{EqPalmEquality}.  Both processes are simple by
Lemma~\ref{LemSimplicity}, stationary, and have intensity one by
Proposition~\ref{PropInputs}\textup{(b)}.  Applying the Palm inversion
formula \eqref{EqPalmInversion} gives \(\bP=\Sine_\beta\).
\end{proof}

\section{Counterexamples without the finite-energy condition}
\label{SectionCounterexamples}

This section shows that the finite-energy hypothesis in
Theorem~\ref{ThmMain} is indispensable.  

\begin{lem}
\label{LemFiniteEnergyDetectsDensity}
If \(\bgamma\in\Conf(\R)\) satisfies
\(\widetilde{\cW}^{\mathrm{div}}(\bgamma)<\infty\), then
\begin{equation}\label{EqFiniteEnergyDensityOne}
 \lim_{R\to\infty}\frac{\bgamma([-R,R])}{2R}=1.
\end{equation}
Consequently, any configuration whose asymptotic density exists and is not
one has \(\widetilde{\cW}^{\mathrm{div}}=+\infty\).
\end{lem}

\begin{proof}
Choose a divergence-compatible field \(\mathrm E\) for which the energy in
Definition~\ref{DefElectricEnergies} is finite, and fix a sufficiently small
cutoff \(\eps>0\).  The definition together with the standard cutoff
comparison gives
\[
 \int_{[-R,R]\times\R}|\mathrm E_\eps|^2\dd X
 \leq C_\eps R
\]
for all sufficiently large \(R\).  The Gauss--flux proof of
\cite[Lemma~2.1]{PS} uses only the truncated divergence identity and this
strip-energy bound, and not that the field is a gradient.  It therefore
applies to the present divergence-compatible class and gives
\eqref{EqFiniteEnergyDensityOne}.  The final assertion follows by
contraposition and the universal lower bound stated in
Section~\ref{SectionSetting}.
\end{proof}

For an integer \(m\geq1\), define the compression map on configurations
\begin{equation}\label{EqDilationDefinition}
 \begin{aligned}
 \mathsf{Comp}_m\bgamma&\defeq(T_{1/m})_*\bgamma,\\
 (\mathsf{Comp}_m\bgamma)(A)&=\bgamma(mA),
 &mA&\defeq\{mx:x\in A\},
 \qquad A\subset\R\text{ Borel}.
 \end{aligned}
\end{equation}
Thus an atom at \(x\) is sent to \(x/m\).  For a point-process law \(\bP\),
write \(\bP^{(m)}\defeq(\mathsf{Comp}_m)_*\bP\).

This integer compression preserves particle number and changes the
finite-volume Gibbs weight only by configuration-independent constants. The following lemma is rather straightforward.

\begin{lem}
\label{LemDLRDilation}
Let \(m\in\N\).  If \(\bP\) satisfies the canonical \(\beta\)-DLR equations,
then \(\bP^{(m)}\) satisfies the same canonical \(\beta\)-DLR equations.  If in
addition \(\bP\) is stationary of intensity \(\lambda\), then \(\bP^{(m)}\) is
stationary of intensity \(m\lambda\).
\end{lem}

\begin{proof}
Fix \(\Lambda\), put \(\widehat\Lambda=m\Lambda\), and let
\(\bgamma\) be \((\beta,\widehat\Lambda)\)-admissible.  Write
\(\bgamma^{(m)}=\mathsf{Comp}_m\bgamma\), set
\(n=\bgamma(\widehat\Lambda)\), and, for
\(\betaeta\in\Conf_n(\Lambda)\), put
\(\widehat\betaeta=\mathsf{Comp}_m^{-1}\betaeta\).
Since \(m\Lambda_p=\Lambda_{mp}\) and
\[
 \mathsf{g}(u/m-v/m)=\mathsf{g}(u-v)+\log(m),
\]
the added constant is multiplied by
\((\widehat\betaeta-\bgamma_{\widehat\Lambda})(\widehat\Lambda)=0\).
Consequently,
\[
 \mathsf M_{\Lambda,\Lambda_p}(\betaeta,\bgamma^{(m)})
 =
 \mathsf M_{\widehat\Lambda,\Lambda_{mp}}
   (\widehat\betaeta,\bgamma).
\]
Because \(m\) is an integer, the right-hand side is a subsequence of the
cutoffs defining the move function for \(\bgamma\).  Its uniform convergence
on the fixed particle-number fibre proves that \(\bgamma^{(m)}\) is
\((\beta,\Lambda)\)-admissible and that the same identity holds for the
limiting move functions.

Moreover, with the convention \(+\infty\) for nonsimple configurations,
\[
 \mathsf H_\Lambda(\betaeta)
 =
 \mathsf H_{\widehat\Lambda}(\widehat\betaeta)
 +\binom{n}{2}\log(m),
 \qquad
 (\mathsf{Comp}_m)_*\Bin_{n,\widehat\Lambda}=\Bin_{n,\Lambda}.
\]
Hence
\[
 \begin{split}
 Z_{\Lambda,\R}^{\beta}(\bgamma^{(m)})
 &=m^{-\beta\binom{n}{2}}
   Z_{\widehat\Lambda,\R}^{\beta}(\bgamma),\\
 \mathsf G_{\Lambda,\R}^{\beta}
   (\mathord\cdot,\bgamma^{(m)})
 &=(\mathsf{Comp}_m)_*
   \mathsf G_{\widehat\Lambda,\R}^{\beta}
   (\mathord\cdot,\bgamma).
 \end{split}
\]
The Borel isomorphism \(\mathsf{Comp}_m\) identifies the corresponding
canonical sigma-fields, transports their \(\bP\)- and
\(\bP^{(m)}\)-completions, and preserves universal measurability.  It
therefore transports the universally measurable full admissibility set and
completed-measurable kernel extension for \(\bP\) to ones for
\(\bP^{(m)}\).  The preceding kernel
identity, inserted into the DLR equation for \(\bP\), proves the DLR equation
for \(\bP^{(m)}\).

Finally,
\[
 \theta_t\mathsf{Comp}_m=\mathsf{Comp}_m\theta_{mt},
 \qquad
 \E_{\bP^{(m)}}[\bC(A)]
 =\E_\bP[\bC(mA)]=m\lambda|A|,
\]
which proves stationarity and the asserted intensity.
\end{proof}

Now, applying compression covariance to \(\Sine_\beta\), and the density obstruction
to the resulting intensities, produces explicit stationary DLR phases outside
the finite-energy class.  The vacuum gives a further degenerate example.

\begin{prop}
\label{PropExplicitDLRCounterexamples}
For every integer \(m\geq2\), let
\begin{equation}\label{EqScaledSineDefinition}
 \bP_{\beta,m}
 \defeq(\mathsf{Comp}_m)_*\Sine_\beta.
\end{equation}
Then \(\bP_{\beta,m}\) is a stationary simple solution of the canonical
\(\beta\)-DLR equations, has intensity \(m\), is different from
\(\Sine_\beta\), and satisfies
\[
 \cW^{\mathrm{div}}(\bP_{\beta,m})=+\infty.
\]
The vacuum law \(\delta_{\varnothing}\) is another stationary canonical
\(\beta\)-DLR solution for every \(\beta>0\), and
\(\cW^{\mathrm{div}}(\delta_{\varnothing})=+\infty\).
\end{prop}

\begin{proof}
Lemmas~\ref{LemDHLMKernelBridge} and \ref{LemSimplicity} give the DLR
property and simplicity of \(\Sine_\beta\), and
Lemma~\ref{LemDLRDilation} then gives the corresponding assertions for
\(\bP_{\beta,m}\), together with stationarity and intensity \(m\).

The lower bound
\(\widetilde{\cW}^{\mathrm{div}}\geq-C_{\mathrm{el}}\) and
Lemma~\ref{LemSineFiniteEnergy} imply that
\(\widetilde{\cW}^{\mathrm{div}}(\bC)<\infty\)
\(\Sine_\beta\)-almost surely.  Lemma~\ref{LemFiniteEnergyDetectsDensity}
therefore gives density one almost surely under \(\Sine_\beta\).  If
\(\bgamma^{(m)}=\mathsf{Comp}_m\bgamma\), then
\[
 \frac{\bgamma^{(m)}([-R,R])}{2R}
 =
 m\frac{\bgamma([-mR,mR])}{2mR}\longrightarrow m.
\]
Thus \(\widetilde{\cW}^{\mathrm{div}}=+\infty\)
\(\bP_{\beta,m}\)-almost surely, and
\(\cW^{\mathrm{div}}(\bP_{\beta,m})=+\infty\).
The different intensities also show that
\(\bP_{\beta,m}\neq\Sine_\beta\).

For the vacuum, the canonical fibre in every \(\Lambda\) is the singleton
\(\Conf_0(\Lambda)=\{\varnothing\}\).  Hence its canonical kernel is the
constant kernel \(\delta_{\varnothing}\), and the DLR equation is
tautological.  Its asymptotic density is zero, so
Lemma~\ref{LemFiniteEnergyDetectsDensity} gives the asserted infinite
energy.
\end{proof}

The preceding examples have nonunit intensity.  The next one shows that an
intensity-one assumption by itself still does not restore uniqueness.

\begin{prop}
\label{PropUnitIntensityDLRCounterexample}
Put
\begin{equation}\label{EqUnitIntensityMixture}
 \bP_\beta^{\mathrm{mix}}
 \defeq\frac12\delta_{\varnothing}+\frac12\bP_{\beta,2}.
\end{equation}
Then \(\bP_\beta^{\mathrm{mix}}\) is stationary, simple, has intensity one,
satisfies the canonical \(\beta\)-DLR equations, and
\[
 \bP_\beta^{\mathrm{mix}}\neq\Sine_\beta,
 \qquad
 \cW^{\mathrm{div}}(\bP_\beta^{\mathrm{mix}})=+\infty.
\]
\end{prop}

\begin{proof}
Stationarity and simplicity are preserved by the mixture, and its intensity
is \(\frac12\cdot0+\frac12\cdot2=1\).

It remains only to check the completed-measurability requirement in the DLR
definition.  Fix \(\Lambda\) and set
\[
 \mathsf D_{\Lambda,2}
 \defeq
 \left\{\bgamma:
   \lim_{k\to\infty}
   \frac{\bgamma(\Lambda_k\setminus\Lambda)}{k}=2\right\}.
\]
This event belongs to \(\scrF_{\Lambda^c}\).  By the density computation
above, \(\bP_{\beta,2}\)-almost every configuration has density two; since
\(\Lambda\) is bounded and configurations are locally finite, removing
\(\Lambda\) changes the numerator by only \(O_{\bgamma}(1)\).  Hence
\(\bP_{\beta,2}(\mathsf D_{\Lambda,2})=1\), whereas
\(\delta_\varnothing(\mathsf D_{\Lambda,2})=0\).  Let
\(\mathscr A_{\Lambda,2}^{\mathrm{DLR}}\) and \(K_2^\Lambda\) be a full
admissibility set and a completed-measurable DLR-kernel extension for
\(\bP_{\beta,2}\), and define
\[
 K_{\mathrm{mix}}^\Lambda(\mathord\cdot,\bgamma)
 =
 \begin{cases}
  K_2^\Lambda(\mathord\cdot,\bgamma),
    &\bgamma\in\mathsf D_{\Lambda,2},\\
  \delta_{\varnothing},
    &\bgamma\notin\mathsf D_{\Lambda,2}.
 \end{cases}
\]
This is measurable for the
\(\bP_\beta^{\mathrm{mix}}\)-completion of
\(\scrF_\Lambda^{\mathrm{can}}\).  Indeed, if \(N\) is a
\(\bP_{\beta,2}\)-null exceptional set for the completed-measurable kernel,
then
\[
 \bP_\beta^{\mathrm{mix}}(N\cap\mathsf D_{\Lambda,2})
 =\tfrac12\bP_{\beta,2}(N\cap\mathsf D_{\Lambda,2})=0,
\]
because \(\varnothing\notin\mathsf D_{\Lambda,2}\); off that event the
kernel is constant.  On the universally measurable,
\(\bP_\beta^{\mathrm{mix}}\)-full set
\[
 \bigl(\mathscr A_{\Lambda,2}^{\mathrm{DLR}}
       \cap\mathsf D_{\Lambda,2}\bigr)
 \cup\{\varnothing\},
\]
it agrees with the canonical kernel.  Averaging the DLR identities of the
two components proves the DLR identity for
\(\bP_\beta^{\mathrm{mix}}\).

The asymptotic density under the mixture is either zero or two, whereas it
is one almost surely under \(\Sine_\beta\).  The laws are therefore mutually
singular, and Lemma~\ref{LemFiniteEnergyDetectsDensity} makes the pointwise
energy infinite on both mixture components.  Hence
\(\cW^{\mathrm{div}}(\bP_\beta^{\mathrm{mix}})=+\infty\).
\end{proof}

\bibliographystyle{acm}
\bibliography{ReferencesDLR}

@article{BLTransport,
  author  = {Bobkov, Sergey G. and Ledoux, Michel},
  title   = {From {B}runn--{M}inkowski to {B}rascamp--{L}ieb and to logarithmic {S}obolev inequalities},
  journal = {Geometric and Functional Analysis},
  volume  = {10},
  year    = {2000},
  pages   = {1028--1052},
}

@article{CETransport,
  author  = {Cordero-Erausquin, Dario},
  title   = {Transport inequalities for log-concave measures, quantitative forms, and applications},
  journal = {Canadian Journal of Mathematics},
  volume  = {69},
  number  = {3},
  year    = {2017},
  pages   = {481--501},
}

@article{DHLM,
  author  = {Dereudre, David and Hardy, Adrien and Lebl{\'e}, Thomas and Ma{\"i}da, Myl{\`e}ne},
  title   = {{DLR} equations and rigidity for the sine-beta process},
  journal = {Communications on Pure and Applied Mathematics},
  volume  = {74},
  number  = {1},
  year    = {2021},
  pages   = {172--222},
}

@article{EHL,
  author  = {Erbar, Matthias and Huesmann, Martin and Lebl{\'e}, Thomas},
  title   = {The one-dimensional log-gas free energy has a unique minimizer},
  journal = {Communications on Pure and Applied Mathematics},
  volume  = {74},
  number  = {3},
  year    = {2021},
  pages   = {615--675},
}

@article{LS,
  author  = {Lebl{\'e}, Thomas and Serfaty, Sylvia},
  title   = {Large deviation principle for empirical fields of log and {R}iesz gases},
  journal = {Inventiones Mathematicae},
  volume  = {210},
  number  = {3},
  year    = {2017},
  pages   = {645--757},
}

@article{PS,
  author  = {Petrache, Mircea and Serfaty, Sylvia},
  title   = {Next order asymptotics and renormalized energy for {R}iesz interactions},
  journal = {Journal of the Institute of Mathematics of Jussieu},
  volume  = {16},
  number  = {3},
  year    = {2017},
  pages   = {501--569},
}

@article{NV,
  author  = {Najnudel, Joseph and Vir{\'a}g, B{\'a}lint},
  title   = {Uniform point variance bounds in classical beta ensembles},
  journal = {Random Matrices: Theory and Applications},
  volume  = {10},
  number  = {4},
  year    = {2021},
  pages   = {2150033},
}

@article{VV,
  author  = {Valk{\'o}, Benedek and Vir{\'a}g, B{\'a}lint},
  title   = {Operator limit of the circular beta ensemble},
  journal = {The Annals of Probability},
  volume  = {48},
  number  = {3},
  year    = {2020},
  pages   = {1286--1316},
}

@misc{SerfatyLecturesCoulombRiesz,
  author        = {Serfaty, Sylvia},
  title         = {Lectures on {C}oulomb and {R}iesz gases},
  year          = {2024},
  eprint        = {2407.21194},
  archivePrefix = {arXiv},
  primaryClass  = {math-ph},
  url           = {https://arxiv.org/abs/2407.21194},
  note          = {arXiv:2407.21194},
}

@book{SerfatyCoulombGinzburg,
  author    = {Serfaty, Sylvia},
  title     = {Coulomb Gases and {G}inzburg--{L}andau Vortices},
  series    = {Zurich Lectures in Advanced Mathematics},
  volume    = {21},
  publisher = {European Mathematical Society},
  address   = {Z{"u}rich},
  year      = {2015},
}

@article{SandierSerfaty2D,
  author  = {Sandier, {'E}tienne and Serfaty, Sylvia},
  title   = {{2D} {C}oulomb gases and the renormalized energy},
  journal = {The Annals of Probability},
  volume  = {43},
  number  = {4},
  year    = {2015},
  pages   = {2026--2083},
}

@article{SandierSerfaty1D,
  author  = {Sandier, {'E}tienne and Serfaty, Sylvia},
  title   = {{1D} log gases and the renormalized energy: crystallization at vanishing temperature},
  journal = {Probability Theory and Related Fields},
  volume  = {162},
  number  = {3--4},
  year    = {2015},
  pages   = {795--846},
}

@article{RougerieSerfaty,
  author  = {Rougerie, Nicolas and Serfaty, Sylvia},
  title   = {Higher-dimensional {C}oulomb gases and renormalized energy functionals},
  journal = {Communications on Pure and Applied Mathematics},
  volume  = {69},
  number  = {3},
  year    = {2016},
  pages   = {519--605},
}

@misc{LebleDLR2DOCP,
  author        = {Lebl{\'e}, Thomas},
  title         = {{DLR} equations, number-rigidity and translation-invariance for infinite-volume limit points of the {2DOCP}},
  year          = {2024},
  eprint        = {2410.04958},
  archivePrefix = {arXiv},
  primaryClass  = {math.PR},
  url           = {https://arxiv.org/abs/2410.04958},
  note          = {arXiv:2410.04958},
}

@misc{LebleHDR,
  author       = {Lebl{\'e}, Thomas},
  title        = {Les log-gas et leurs amis},
  year         = {2026},
  howpublished = {Habilitation {\`a} diriger des recherches, Universit{\'e} Paris Cit{\'e}},
  url          = {https://tleble.perso.math.cnrs.fr/assets/papers/HDR_Leble.pdf},
}

@article{Osada2012,
  author  = {Osada, Hirofumi},
  title   = {Infinite-dimensional stochastic differential equations related to random matrices},
  journal = {Probability Theory and Related Fields},
  volume  = {153},
  number  = {1--2},
  year    = {2012},
  pages   = {471--509},
}

@article{Osada2013,
  author  = {Osada, Hirofumi},
  title   = {Interacting {B}rownian motions in infinite dimensions with logarithmic interaction potentials},
  journal = {The Annals of Probability},
  volume  = {41},
  number  = {1},
  year    = {2013},
  pages   = {1--49},
}

@article{OsadaTanemura2016,
  author  = {Osada, Hirofumi and Tanemura, Hideki},
  title   = {Strong {M}arkov property of determinantal processes with extended kernels},
  journal = {Stochastic Processes and their Applications},
  volume  = {126},
  number  = {1},
  year    = {2016},
  pages   = {186--208},
}

@misc{OsadaOsadaCoulombISDE,
  author        = {Osada, Hirofumi and Osada, Shota},
  title         = {Infinite-dimensional stochastic differential equations for {C}oulomb random point fields},
  year          = {2025},
  eprint        = {2508.21658},
  archivePrefix = {arXiv},
  primaryClass  = {math.PR},
  url           = {https://arxiv.org/abs/2508.21658},
  note          = {arXiv:2508.21658},
}

@article{SuzukiErgodicity,
  author  = {Suzuki, Kohei},
  title   = {On the ergodicity of interacting particle systems under number rigidity},
  journal = {Probability Theory and Related Fields},
  volume  = {188},
  number  = {1--2},
  year    = {2024},
  pages   = {583--623},
}

@article{SuzukiCurvature,
  author  = {Suzuki, Kohei},
  title   = {Curvature bound of {D}yson {B}rownian motion},
  journal = {Communications in Mathematical Physics},
  volume  = {406},
  number  = {7},
  year    = {2025},
  pages   = {Paper No. 154, 56},
}

@misc{AssiotisSuzukiCollision,
  author        = {Assiotis, Theodoros and Suzuki, Kohei},
  title         = {Collision and non-collision for diffusions on configuration space},
  year          = {2026},
  eprint        = {2607.22333},
  archivePrefix = {arXiv},
  primaryClass  = {math.PR},
  url           = {https://arxiv.org/abs/2607.22333},
  note          = {arXiv:2607.22333},
}

@article{DysonBrownianMotion,
  author  = {Dyson, Freeman J.},
  title   = {A {B}rownian-motion model for the eigenvalues of a random matrix},
  journal = {Journal of Mathematical Physics},
  volume  = {3},
  number  = {6},
  year    = {1962},
  pages   = {1191--1198},
}

@article{LiValkoCircularJacobi,
  author  = {Li, Yun and Valk{\'o}, Benedek},
  title   = {Operator level limit of the circular {J}acobi {$\beta$}-ensemble},
  journal = {Random Matrices: Theory and Applications},
  volume  = {11},
  number  = {4},
  year    = {2022},
  pages   = {Paper No. 2250043},
}

@misc{AssiotisNajnudel,
  author        = {Assiotis, Theodoros and Najnudel, Joseph},
  title         = {Moments of {$\mathrm{C}\beta\mathrm{E}$} field partition function, {$\mathrm{Sine}_\beta$} correlations and stochastic zeta},
  year          = {2026},
  eprint        = {2602.08739},
  archivePrefix = {arXiv},
  primaryClass  = {math.PR},
  url           = {https://arxiv.org/abs/2602.08739},
  note          = {arXiv:2602.08739},
}

@article{ValkoViragCarousel,
  author  = {Valk{\'o}, Benedek and Vir{\'a}g, B{\'a}lint},
  title   = {Continuum limits of random matrices and the {B}rownian carousel},
  journal = {Inventiones Mathematicae},
  volume  = {177},
  number  = {3},
  year    = {2009},
  pages   = {463--508},
}

@article{ValkoViragSineOperator,
  author  = {Valk{\'o}, Benedek and Vir{\'a}g, B{\'a}lint},
  title   = {The {$\mathrm{Sine}_\beta$} operator},
  journal = {Inventiones Mathematicae},
  volume  = {209},
  number  = {1},
  year    = {2017},
  pages   = {275--327},
}

@article{KillipStoiciu,
  author  = {Killip, Rowan and Stoiciu, Mihai},
  title   = {Eigenvalue statistics for {CMV} matrices: from {P}oisson to clock via random matrix ensembles},
  journal = {Duke Mathematical Journal},
  volume  = {146},
  number  = {3},
  year    = {2009},
  pages   = {361--399},
}

@book{ForresterBook,
  author    = {Forrester, Peter J.},
  title     = {Log-Gases and Random Matrices},
  series    = {London Mathematical Society Monographs Series},
  volume    = {34},
  publisher = {Princeton University Press},
  address   = {Princeton, NJ},
  year      = {2010},
}

@incollection{MontgomeryPairCorrelation,
  author    = {Montgomery, Hugh L.},
  title     = {The pair correlation of zeros of the zeta function},
  booktitle = {Analytic Number Theory (St. Louis, 1972)},
  series    = {Proceedings of Symposia in Pure Mathematics},
  volume    = {24},
  publisher = {American Mathematical Society},
  address   = {Providence, RI},
  year      = {1973},
  pages     = {181--193},
}

@article{KeatingSnaithZeta,
  author  = {Keating, Jonathan P. and Snaith, Nina C.},
  title   = {Random matrix theory and {$\zeta(1/2+it)$}},
  journal = {Communications in Mathematical Physics},
  volume  = {214},
  number  = {1},
  year    = {2000},
  pages   = {57--89},
}

@incollection{BourgadeKeating,
  author    = {Bourgade, Paul and Keating, Jonathan P.},
  title     = {Quantum chaos, random matrix theory, and the {R}iemann zeta-function},
  booktitle = {Chaos: Poincar{\'e} Seminar 2010},
  series    = {Progress in Mathematical Physics},
  volume    = {66},
  publisher = {Birkh{\"a}user},
  address   = {Basel},
  year      = {2013},
  pages     = {125--168},
}

@article{KuijlaarsMinaDiaz,
  author  = {Kuijlaars, Arno B. J. and Mi{\~n}a-D{\'i}az, Erwin},
  title   = {Universality for conditional measures of the sine point process},
  journal = {Journal of Approximation Theory},
  volume  = {243},
  year    = {2019},
  pages   = {1--24},
}

@article{BufetovConditional,
  author  = {Bufetov, Alexander I.},
  title   = {Conditional measures of determinantal point processes},
  journal = {Functional Analysis and Its Applications},
  volume  = {54},
  number  = {1},
  year    = {2020},
  pages   = {7--20},
}

@article{BufetovQuasiSymmetries,
  author  = {Bufetov, Alexander I.},
  title   = {Quasi-symmetries of determinantal point processes},
  journal = {The Annals of Probability},
  volume  = {46},
  number  = {2},
  year    = {2018},
  pages   = {956--1003},
}

@article{BourgadeErdosYau,
  author  = {Bourgade, Paul and Erd{\H o}s, L{\'a}szl{\'o} and Yau, Horng-Tzer},
  title   = {Universality of general beta-ensembles},
  journal = {Duke Mathematical Journal},
  volume  = {163},
  number  = {6},
  year    = {2014},
  pages   = {1127--1190},
}

@article{BourgadeErdosYauNonconvex,
  author  = {Bourgade, Paul and Erd{\H o}s, L{\'a}szl{\'o} and Yau, Horng-Tzer},
  title   = {Bulk universality of general beta-ensembles with non-convex potential},
  journal = {Journal of Mathematical Physics},
  volume  = {53},
  number  = {9},
  year    = {2012},
  pages   = {095221},
}

@book{ErdosYauBook,
  author    = {Erd{\H o}s, L{\'a}szl{\'o} and Yau, Horng-Tzer},
  title     = {A Dynamical Approach to Random Matrix Theory},
  series    = {Courant Lecture Notes in Mathematics},
  volume    = {28},
  publisher = {American Mathematical Society},
  address   = {Providence, RI},
  year      = {2017},
}

@misc{BourgadeHuangLoop,
  author        = {Bourgade, Paul and Huang, Jiaoyang},
  title         = {Loop equations characterize random matrix statistics},
  year          = {2026},
  eprint        = {2607.07617},
  archivePrefix = {arXiv},
  primaryClass  = {math.PR},
  url           = {https://arxiv.org/abs/2607.07617},
  note          = {arXiv:2607.07617},
}

@article{Dobrushin1968,
  author  = {Dobrushin, Roland L.},
  title   = {The description of a random field by means of conditional probabilities and conditions of its regularity},
  journal = {Theory of Probability and Its Applications},
  volume  = {13},
  number  = {2},
  year    = {1968},
  pages   = {197--224},
}

@article{LanfordRuelle,
  author  = {Lanford, Oscar E. and Ruelle, David},
  title   = {Observables at infinity and states with short range correlations in statistical mechanics},
  journal = {Communications in Mathematical Physics},
  volume  = {13},
  year    = {1969},
  pages   = {194--215},
}

@book{Georgii,
  author    = {Georgii, Hans-Otto},
  title     = {Gibbs Measures and Phase Transitions},
  edition   = {2},
  series    = {De Gruyter Studies in Mathematics},
  volume    = {9},
  publisher = {De Gruyter},
  address   = {Berlin},
  year      = {2011},
}

@article{YauRelativeEntropy,
  author  = {Yau, Horng-Tzer},
  title   = {Relative entropy and hydrodynamics of {G}inzburg--{L}andau models},
  journal = {Letters in Mathematical Physics},
  volume  = {22},
  number  = {1},
  year    = {1991},
  pages   = {63--80},
}

@article{FoellmerEntropy,
  author  = {F{\"o}llmer, Hans},
  title   = {On entropy and information gain in random fields},
  journal = {Zeitschrift f{\"u}r Wahrscheinlichkeitstheorie und Verwandte Gebiete},
  volume  = {26},
  number  = {3},
  year    = {1973},
  pages   = {207--217},
}

@article{GeorgiiZessin,
  author  = {Georgii, Hans-Otto and Zessin, Hans},
  title   = {Large deviations and the maximum entropy principle for marked point random fields},
  journal = {Probability Theory and Related Fields},
  volume  = {96},
  number  = {2},
  year    = {1993},
  pages   = {177--204},
}

@article{DereudreVariational,
  author  = {Dereudre, David},
  title   = {Variational principle for {G}ibbs point processes with finite range interaction},
  journal = {Electronic Communications in Probability},
  volume  = {21},
  year    = {2016},
  pages   = {Paper No. 10, 11},
}

@article{GuoPapanicolaouVaradhan,
  author  = {Guo, Maozheng and Papanicolaou, George C. and Varadhan, S. R. S.},
  title   = {Nonlinear diffusion limit for a system with nearest neighbor interactions},
  journal = {Communications in Mathematical Physics},
  volume  = {118},
  number  = {1},
  year    = {1988},
  pages   = {31--59},
}

\bigskip 

\noindent{\sc School of Mathematics, University of Edinburgh, James Clerk Maxwell Building, Peter Guthrie Tait Rd, Edinburgh EH9 3FD, U.K.}\newline
\href{mailto:theo.assiotis@ed.ac.uk}{\small theo.assiotis@ed.ac.uk}

\end{document}